\documentclass[12pt,reqno]{amsart} 
\usepackage{amssymb}
\usepackage{enumerate}
\usepackage[T1]{fontenc}
\usepackage[all]{xy}
\usepackage{mathrsfs}
\usepackage{rotating}

\usepackage[usenames]{color}
\definecolor{darkgreen}{rgb}{0,1,0}
\definecolor{bluegreen}{rgb}{0,0.2,0.8}
\definecolor{newcolor}{rgb}{0.8,0,0}
\definecolor{newercolor}{rgb}{0.6,0.6,0}
\definecolor{darkyellow}{rgb}{0.7,0.7,0}
\definecolor{orange}{rgb}{0.7,0.35,0}

\newcommand{\UUU}{\mathbf{U}}
\newcommand{\UU}{U}
\newcommand{\genG}{\varGamma}
\newcommand{\widebar}[1]{\overset{\mskip2mu\hrulefill\mskip2mu}{#1}
		\vphantom{#1}}
\newcommand{\4}[1]{\widebar{#1}}
\newcommand{\5}[1]{\widehat{#1}}

\newcommand{\7}{^\vee}
\newcommand{\9}[1]{{}^{#1}\!}   
\newcommand{\EE}[1]{\mathbf{E}_{#1}}
\newcommand{\autff}{\autf\7}
\newcommand{\outff}{\outf\7}
\newcommand{\xx}{\mathbf{x}}
\newcommand{\xa}{\mathbf{a}}
\newcommand{\uu}{\mathbf{u}}
\newcommand{\too}{\longrightarrow}

\newcommand{\dbl}[2]{\renewcommand{\arraystretch}{1.0}%
\setlength{\tabcolsep}{0pt}%
\begin{tabular}{c}\rule{0pt}{12pt}$#1$\\$#2$\end{tabular}}

\newlength{\short}
\newlength{\shorter}
\newcommand{\boldd}[1]{{\mathversion{bold}\textbf{#1}}}

\newcommand{\lie}[3]{\def\test{#2}\def\tst{G}\ifx\test\tst{{}^{#1}#2_{#3}}
\else{{}^{#1}\!#2_{#3}}\fi}

\let\oldcirc=\circ
\renewcommand{\circ}{\mathchoice
    {\mathbin{\scriptstyle\oldcirc}}{\mathbin{\scriptstyle\oldcirc}}
    {\mathbin{\scriptscriptstyle\oldcirc}}
    {\mathbin{\scriptscriptstyle\oldcirc}}}

\newlength{\upto}\newlength{\dnto}
\SelectTips{cm}{10} \UseTips   

\numberwithin{equation}{section}

\numberwithin{table}{section}

\let\endpf=\endproof
\renewcommand{\endproof}{\endpf\setcounter{equation}{0}}

\mathcode`\:="003A
\mathchardef\cdot="0201

\def\beq#1\eeq{\begin{equation*}#1\end{equation*}}
\def\beqq#1\eeqq{\begin{equation}#1\end{equation}}

\let\emptyset=\varnothing
\renewcommand{\:}{\colon}   
\newcommand{\longline}{\centerline{\hbox to 5cm{\hrulefill}}}

\newcommand{\mxtwo}[4]{\left(\begin{smallmatrix}#1&#2\\#3&#4
\end{smallmatrix}\right)}

\newcommand{\mxfourb}[8]{#1&#2&#3&#4\\#5&#6&#7&#8\end{smallmatrix}\right)}
\newcommand{\mxfoura}[8]{\left(\begin{smallmatrix}#1&#2&#3&#4\\#5&#6&#7&#8\\}

\DeclareMathAlphabet\EuR{U}{eur}{m}{n}
\SetMathAlphabet\EuR{bold}{U}{eur}{b}{n}

\newcommand{\higherlim}[2]{\displaystyle\setbox1=\hbox{\rm lim}
	\setbox2=\hbox to \wd1{\leftarrowfill} \ht2=0pt \dp2=-1pt
	\setbox3=\hbox{$\scriptstyle{#1}$}
	\def\test{#1}\ifx\test\empty
	\mathop{\mathop{\vtop{\baselineskip=5pt\box1\box2}}}\nolimits^{#2}
	\else
	\ifdim\wd1<\wd3
	\mathop{\hphantom{^{#2}}\vtop{\baselineskip=5pt\box1\box2}^{#2}}_{#1}
	\else
	\mathop{\mathop{\vtop{\baselineskip=5pt\box1\box2}}_{#1}}%
	\nolimits^{#2}
	\fi\fi}
\newcommand{\higherlimm}[2]{\setbox1=\hbox{\rm lim}
	\setbox2=\hbox to \wd1{\leftarrowfill} \ht2=0pt \dp2=-1pt
	\mathop{\mathop{\vtop{\baselineskip=5pt\box1\box2}}}\limits_{#1}
	\nolimits^{#2}}

\newcounter{let} 
\loop\stepcounter{let}
\expandafter\edef\csname cal\alph{let}\endcsname%
{\noexpand\mathcal{\Alph{let}}}
\ifnum\thelet<26\repeat

\newcommand{\tdef}[2][]{\expandafter\newcommand\csname#2\endcsname%
{#1\textup{#2}}}
\tdef{Iso}   \tdef{Aut}    \tdef{Out}    \tdef{Inn}    \tdef{Hom}
\tdef{End}   \tdef{Inj}    \tdef{map}    \tdef{Ker}    \tdef{Ob}
\tdef{Mor}   \tdef{Res}    \tdef{Spin}   \tdef{Id}     \tdef{Fr}
\tdef{rk}    \tdef{conj}   \tdef{incl}   \tdef{proj}   \tdef{diag} 
\tdef{trf}   \tdef{Sol}    \tdef{He}     \tdef{Sz}     \tdef{cj}
\tdef{free}  \tdef{supp}   \tdef{pr}     \tdef{Irr}    \tdef{Ind}
\tdef{Gal}   \tdef[_]{typ} \tdef[^]{op}  \tdef[^]{ab}  \tdef{Tor} 
\tdef{Ext}   \tdef{Ann}

\newcommand{\fdef}[1]{\expandafter\newcommand\csname#1\endcsname%
{\mathfrak{#1}}}
\fdef{X}  \fdef{red}  \fdef{foc}  \fdef{hyp}

\newcommand{\pp}{\mathfrak{p}}

\newcommand{\bbdef}[1]{\expandafter\newcommand%
\csname#1\endcsname{\mathbb{#1}}}
\bbdef{C} \bbdef{F} \bbdef{R} \bbdef{Z} \bbdef{Q}

\newcommand{\itdef}[1]{\expandafter\newcommand\csname#1\endcsname%
{\textit{#1}}}
\itdef{PSL}  \itdef{PSU}  \itdef{SL}  \itdef{SU}  \itdef{GL} \itdef{GU}
\itdef{UT}   \itdef{SO}   \itdef{PGL} \itdef{PSp} \itdef{Sp} \itdef{PSO}

\newcommand{\PSSL}{\textit{P}{\varSigma}\textit{L}}

\newcommand{\scal}{\textup{sc}}

\newcommand{\gee}{\varepsilon}

\newcommand{\gen}[1]{\langle{#1}\rangle}
\newcommand{\Gen}[1]{\bigl\langle{#1}\bigr\rangle}

\let\nsg=\normal
\let\nnsg=\ntrianglelefteq

\newcommand{\syl}[2]{\textup{Syl}_{#1}(#2)}
\newcommand{\sylp}[1]{\syl{p}{#1}}
\newcommand{\autf}{\Aut_{\calf}}

\newcommand{\outf}{\Out_{\calf}}

\newcommand{\homf}{\Hom_{\calf}}

\newcommand{\sminus}{\smallsetminus}
\newcommand{\defeq}{\overset{\textup{def}}{=}}

\renewcommand{\Im}{\textup{Im}}

\newcommand{\sd}[1]{\overset{{#1}}{\rtimes}}

\let\til=\widetilde

\newcommand{\longleft}[1]{\;{\leftarrow%
\count255=0 \loop \mathrel{\mkern-6mu}%
    \relbar\advance\count255 by1\ifnum\count255<#1\repeat}\;}
\newcommand{\longright}[1]{\;{\count255=0 \loop \relbar\mathrel{\mkern-6mu}%
    \advance\count255 by1\ifnum\count255<#1\repeat\rightarrow}\;}
\newcommand{\Right}[2]{\overset{#2}{\longright#1}}
\newcommand{\RIGHT}[3]{\mathrel{\mathop{\kern0pt\longright#1}
	\limits^{#2}_{#3}}}

\newcommand{\LEFT}[3]{\mathrel{\mathop{\kern0pt\longleft#1}\limits^{#2}_{#3}}
}
\newcommand{\longleftright}[1]{\;{\leftarrow\mathrel{\mkern-6mu}%
    \count255=0\loop\relbar\mathrel{\mkern-6mu}%
    \advance\count255 by1\ifnum\count255<#1\repeat\rightarrow}\;} 
 
\newcommand{\onto}[1]{\;{\count255=0 \loop \relbar\joinrel
    \advance\count255 by1
    \ifnum\count255<#1 \repeat \twoheadrightarrow}\;}

\newcommand{\RLEFT}[3]{\mathrel{%
   \mathop{\vcenter{\baselineskip=0pt\hbox{$\kern0pt\longright#1$}%
   \hbox{$\kern0pt\longleft#1$}}}\limits^{#2}_{#3}}}

\renewenvironment{enumerate}[1][]
{\begin{enumerat}[#1]\setlength{\itemsep}{6pt}}{\end{enumerat}}

\renewenvironment{itemize}
{\begin{itemiz}\setlength{\itemsep}{6pt}\setlength{\itemindent}{-20pt}}
{\end{itemiz}}

\newenvironment{enuma}{\begin{enumerate}[{\rm(a) }]}{\end{enumerate}}
\newenvironment{enumi}{\begin{enumerate}[{\rm(i) }]}{\end{enumerate}}
\newenvironment{enum1}{\begin{enumerate}[{\rm(1) }]}{\end{enumerate}}

\newtheorem{Thm}{Theorem}[section]
\newtheorem{Prop}[Thm]{Proposition}
\newtheorem{Cor}[Thm]{Corollary}
\newtheorem{Lem}[Thm]{Lemma}

\newtheorem{Defi}[Thm]{Definition}   

\newtheorem{Not}[Thm]{Notation}

\newtheorem{Th}{Theorem}

\theoremstyle{definition}

\newtheorem{Ex}[Thm]{Example}

\theoremstyle{remark}

\newcommand{\starfin}{\textup{($*_{\textup{fin}}$)}}
\newcommand{\starinf}{\textup{($*_\infty$)}}

\title{Reduced fusion systems over $p$-groups with abelian subgroup of 
index $p$: III}

\author{Bob Oliver}
\address{LAGA, Institut Galil\'ee, Av. J-B Cl\'ement, 93430
Villetaneuse, France}
\email{bobol@math.univ-paris13.fr}
\thanks{B. Oliver is partially supported by UMR 7539 of the CNRS}

\author{Albert Ruiz}
\address{Departament de Mathem\`atiques, Edifici C, Universitat Aut\`onoma 
de Barcelona, 08193 Bellaterra, Spain}
\email{albert@mat.uab.cat}
\thanks{A. Ruiz is partially supported by MICINN-FEDER project number 
MTM2016-80439-P}


\begin{document}

\begin{abstract} 
We finish the classification, begun in two earlier papers, of all simple 
fusion systems over finite nonabelian $p$-groups with an abelian subgroup 
of index $p$. In particular, this gives many new examples illustrating the 
enormous variety of exotic examples that can arise. In addition, we 
classify all simple fusion systems over infinite nonabelian discrete $p$-toral 
groups with an abelian subgroup of index $p$. In all of these cases 
(finite or infinite), we 
reduce the problem to one of listing all $\F_pG$-modules (for $G$ finite) 
satisfying certain conditions: a problem which was solved in the earlier 
paper \cite{indp2} using the classification of finite simple groups.
\end{abstract}

\maketitle

A \emph{saturated fusion system} over a finite $p$-group $S$ is a category 
whose objects are the subgroups of $S$, and whose morphisms are injective 
homomorphisms between the subgroups, and which satisfy some additional 
conditions first formulated by Puig (who called them ``Frobenius 
$S$-categories'' in \cite{Puig}) and motivated in part by the Sylow 
theorems for finite groups. For example, if $G$ is a finite group and 
$S\in\sylp{G}$, then the category $\calf_S(G)$, whose objects are the 
subgroups of $S$ and whose morphisms are the homomorphisms between 
subgroups defined via conjugation in $G$, is a saturated fusion system over 
$S$. We refer to \cite{Puig}, \cite[Part I]{AKO}, or \cite{Craven} for the 
basic definitions and properties of saturated fusion systems.

A saturated fusion system is \emph{realizable} if it is isomorphic to 
$\calf_S(G)$ for some finite group $G$ and some $S\in\sylp{G}$; it is 
\emph{exotic} otherwise. Here, by an isomorphism of fusion systems we mean 
an isomorphism of categories that is induced by an isomorphism between the 
underlying $p$-groups. Exotic fusion systems over finite $p$-groups seem to 
be quite rare for $p=2$ (the only known examples are those constructed in 
\cite{Sol} and others easily derived from them), but many examples of them 
are known for odd primes $p$. 

A \emph{discrete $p$-torus} is a group of the form $(\Z/p^\infty)^r$ for 
some $r\ge0$, where $\Z/p^\infty$ is the union of the cyclic groups 
$\Z/p^k$ via the obvious inclusions $\Z/p^k<\Z/p^{k+1}$. A \emph{discrete 
$p$-toral group} is a group containing a discrete $p$-torus as a normal 
subgroup of $p$-power index. Saturated fusion systems over discrete 
$p$-toral groups were defined and studied in \cite{BLO3}, motivated by the 
special case of fusion systems for compact Lie groups and $p$-compact 
groups. 

A fusion system is \emph{simple} if it is saturated and contains no proper 
nontrivial normal fusion subsystems (see Definition \ref{d:E<|F}). As a 
special case, very rich in exotic examples, we have been looking at simple 
fusion systems $\calf$ over finite nonabelian $p$-groups $S$ with an 
abelian subgroup $A$ of index $p$. By \cite[Proposition 5.2(a)]{AOV2}, if 
$p=2$, then $S$ is dihedral, semidihedral, or a wreath product of the form 
$C_{2^k}\wr C_2$, and hence $\calf$ is isomorphic to the fusion system of 
$\PSL_2(q)$ or $\PSL_3(q)$ for some odd $q$. Fusion systems over 
extraspecial groups of order $p^3$ and exponent $p$ were listed in 
\cite{RV}, and by \cite[Theorem 2.1]{indp1}, these include the only simple 
fusion systems over nonabelian $p$-groups containing more than one abelian 
subgroup of index $p$. The other cases where $p$ is odd and $A$ is not 
essential (equivalently, not radical) in $\calf$ were handled in 
\cite[Theorem 2.8]{indp1}, while those where $A$ is essential and of 
exponent $p$ was handled in \cite{indp2}. So it remains to describe those 
cases where $A$ is essential and not elementary abelian (and the unique 
abelian subgroup of index $p$). This, together with analogous results about 
simple fusion systems over infinite discrete $p$-toral groups with abelian 
subgroup of index $p$, are the main results of this paper.

To simplify the following summary of our results, we use the term 
``index-$p$-triple'' to denote a triple $(\calf,S,A)$, where $S$ is a 
nonabelian discrete $p$-toral group (finite or infinite) with abelian 
subgroup $A$ of index $p$, and $\calf$ is a simple fusion system over $S$. 
Our main results are shown in Sections \ref{s:finite} and \ref{s:infinite}, 
where we handle separately the finite and infinite cases. In each of these 
sections, we first list, in Theorems \ref{t3:s/a} and \ref{t:s/a}, all 
index-$p$-triples $(\calf,S,A)$, for $S$ finite or infinite, in terms of 
the pair $(G,A)$ where $G=\autf(A)$ and $A$ is regarded as a 
$\Z_pG$-module. Theorem \ref{t3:s/a} is taken directly from \cite[Theorem 
2.8]{indp2}, while Theorem \ref{t:s/a} is new. For completeness in the 
infinite case, we also show that each index-$2$-triple $(\calf,S,A)$ with 
$|S|=\infty$ is isomorphic to that of $\SO(3)$ or $\PSU(3)$ (Theorem 
\ref{t:disc.2-toral}), and that for each $p$ there is (up to isomorphism) a 
unique index-$p$-triple $(\calf,S,A)$ where $|S|=\infty$ and $A$ is not 
essential (Theorem \ref{t:AnotinE}). 

The main theorems, Theorems \ref{ThA} and \ref{ThB}, appear at the ends of 
Sections \ref{s:finite} and \ref{s:infinite}, respectively. In Theorem 
\ref{ThA}, for $p$ odd, we prove that each index-$p$-triple 
$(\calf,S,A)$, where $A$ is finite, essential in $\calf$, and not 
elementary abelian, is determined by $G=\autf(A)$, 
$V=\Omega_1(A)$ regarded as an $\F_pG$-module, the 
exponent of $A$, and some additional information needed when $A$ is not 
homocyclic. In all cases, $\rk(A)\ge p-1$, and $A$ is homocyclic whenever 
$\rk(A)\ge p$. Also, $A$ is always isomorphic to some quotient of a 
$\Z_pG$-lattice.

Theorem \ref{ThB} can be thought of as a ``limiting case'' of the 
classification in Theorem \ref{ThA}. It says that each index-$p$-triple 
$(\calf,S,A)$ such that $A$ is infinite and essential in $\calf$ is 
determined by the pair $(G,V)$, where $G=\autf(A)$, and $V=\Omega_1(A)$ is 
regarded as an $\F_pG$-module. In all such cases, $A$ is a discrete 
$p$-torus of rank at least $p-1$. We also determine which of the 
fusion systems we list are realized as fusion systems of compact Lie groups 
or $p$-compact groups.

Theorems \ref{ThA} and \ref{ThB} reduce our classification problems to 
questions about $\F_pG$-modules with certain properties. These questions 
were already studied in \cite{indp2}, using the classification of finite 
simple groups, and the results in that paper that are relevant in this one 
are summarized in Section \ref{s:examples}. Theorems \ref{ThA} and 
\ref{ThB} together with Proposition \ref{p:reps} and Table \ref{tbl:reps} 
allow us to completely list all simple fusion systems over nonabelian 
discrete $p$-toral groups (finite or infinite) with abelian subgroup of 
index $p$ that is not elementary abelian. In particular, as in the earlier 
papers \cite{indp1} and \cite{indp2}, we find a very large, very rich 
variety of exotic fusion systems over finite $p$-groups (at least for 
$p\ge5$). 

This work was motivated in part by the following questions and 
problems, all of which are familiar to people working in this field.
\begin{enumerate}[\bf Q1: ]

\item For a fixed odd prime $p$, a complete classification of all 
simple fusion systems over finite $p$-groups, or even a conjecture 
as to how they could be classified, 
seems way out of reach for now. But based on the many examples already 
known, is there any meaningful way in which one could begin to 
systematize them; for example, by splitting up the problem into simpler 
cases? Alternatively, is there a class of simple fusion systems over 
finite $p$-groups, much less restrictive than the one we look at here, 
for which there might be some chance of classifying its members?

\item Find some criterion which can be used to prove that some (or at 
least one!) of the examples constructed here or earlier (over 
\emph{finite} $p$-groups for odd primes $p$) are 
exotic, without invoking the classification of finite simple groups.

\item A \emph{torsion linear group} in defining characteristic $q$ is a 
subgroup $\Gamma\le\GL_n(K)$, for some $n\ge1$ and some field $K$ of 
characteristic $q$, such that all elements of $\Gamma$ have finite order. 
If $p$ is a prime and $\Gamma$ is a torsion linear group in defining 
characteristic different from $p$, then by \cite[\S\,8]{BLO3}, there is a 
maximal discrete $p$-toral subgroup $S\le \Gamma$, unique up to 
conjugation, and $\calf_S(\Gamma)$ is a saturated fusion system. Are there 
any saturated fusion systems over discrete $p$-toral groups (for any prime 
$p$) which we can prove are \emph{not} fusion systems of torsion linear 
groups?

\end{enumerate}

The notation used in this paper is mostly standard. We let $A\circ B$ denote a 
central product of $A$ and $B$. When $g$ and $h$ are in a group $G$, 
we set ${}^gh=ghg^{-1}$ and $h^g=g^{-1}hg$. When $A$ is an 
abelian group and $\beta\in\Aut(A)$, we write 
$[\beta,A]=\gen{\beta(x)x^{-1}\,|\,x\in A}$. When $P$ is a $p$-group, 
we let $\Fr(P)$ denote its Frattini subgroup, and for $k\ge1$ set 
	\[ \Omega_k(P)=\gen{g\in P\,|\,g^{p^k}=1} \qquad\textup{and}\qquad
	\mho^k(P)=\gen{g^{p^k}\,|\,g\in P}. \] 

We would like to thank the Centre for Symmetry and Deformation at 
Copenhagen University, and the Universitat Aut\`onoma de Barcelona, for 
their hospitality in allowing us to get together on several different 
occasions. 

\bigskip

\section{Background}
\label{s:background}

We first recall some of the definitions and standard terminology used when 
working with fusion systems. Recall that a \emph{discrete $p$-toral group} 
is a group that contains a normal subgroup of $p$-power index isomorphic to 
$(\Z/p^\infty)^r$ for some $r\ge0$. A \emph{fusion system} over a 
discrete $p$-toral group $S$ is a category $\calf$ whose objects are the 
subgroups of $S$, and where for each $P,Q\le S$, the set $\homf(P,Q)$ is a 
set of injective homomorphisms from $P$ to $Q$ that includes all those 
induced by conjugation in $S$, and such that for each 
$\varphi\in\homf(P,Q)$, we have $\varphi\in\homf(P,\varphi(P))$ and 
$\varphi^{-1}\in\homf(\varphi(P),P)$. 

Define the \emph{rank} $\rk(S)$ of a discrete $p$-torus $S$ by setting 
$\rk(S)=r$ if $S\cong(\Z/p^\infty)^r$. If $S$ is a discrete $p$-toral group 
with normal discrete $p$-torus $S_0\nsg S$ of $p$-power index, then we 
refer to $S_0$ as the \emph{identity component} of $S$, and set 
$|S|=\bigl(\rk(S_0),|S/S_0|\bigr)$, where such pairs are ordered 
lexicographically. Thus if $T$ is another discrete $p$-toral group with 
identity component $T_0$, then $|S|\le|T|$ if $\rk(S_0)<\rk(T_0)$, or if 
$\rk(S_0)=\rk(T_0)$ and $|S/S_0|\le|T/T_0|$. Note that the identity 
component of $S$, and hence $|S|$, are uniquely determined since a discrete 
$p$-torus has no proper subgroups of finite index. 

\begin{Defi} \label{d:saturated}
Fix a prime $p$, a discrete $p$-toral group $S$, and a fusion system 
$\calf$ over $S$.  
\begin{itemize} 

\item For each $P\le S$ and each $g\in S$, $P^\calf$ denotes the set of 
subgroups of $S$ which are $\calf$-conjugate (isomorphic in $\calf$) to 
$P$, and $g^\calf$ denotes the $\calf$-conjugacy class of $g$ (the set of 
images of $g$ under morphisms in $\calf$).

\item A subgroup $P\le S$ is \emph{fully normalized} in $\calf$ 
(\emph{fully centralized} in $\calf$) if $|N_S(P)|\ge|N_S(Q)|$ 
($|C_S(P)|\le|C_S(Q)|$) for each $Q\in P^\calf$. 

\item A subgroup $P\le S$ is \emph{fully automized} in $\calf$ if 
$\outf(P)\defeq\autf(P)/\Inn(P)$ is finite and 
$\Out_S(P)\in\sylp{\outf(P)}$. The subgroup $P$ is \emph{receptive} 
in $\calf$ if for each $Q\in P^\calf$ and each $\varphi\in\Iso_\calf(Q,P)$, 
there is $\4\varphi\in\homf(N_\varphi,S)$ such that $\4\varphi|_P=\varphi$, 
where 
	\[ N_\varphi = \bigl\{ g\in N_S(Q) \,\big|\, 
	\varphi c_g\varphi^{-1}\in\Aut_S(P) \bigr\}\,. \]

\item The fusion system $\calf$ is \emph{saturated} if 
\begin{itemize} 
\item \emph{(Sylow axiom)} each fully normalized subgroup of $S$ is fully 
automized and fully centralized; 
\item \emph{(extension axiom)} each fully centralized subgroup of $S$ is 
receptive; and 
\item \emph{(continuity axiom, when $|S|=\infty$)} if $P_1\le P_2\le P_3\le 
\cdots$ is an increasing sequence of subgroups of $S$ with 
$P=\bigcup_{i=1}^\infty P_i$, and $\varphi\in\Hom(P,S)$ is such that 
$\varphi|_{P_i}\in\homf(P_i,S)$ for each $i\ge1$, then 
$\varphi\in\homf(P,S)$.
\end{itemize}

\end{itemize}
\end{Defi}

The above definition of a saturated fusion system is the one given in 
\cite{BLO2} and \cite[Definition 2.2]{BLO3}. It will not be used directly 
in this paper (saturation of the fusion systems we construct will be shown 
using later theorems), but we will frequently refer to the extension axiom 
as a property of saturated fusion systems.

We now need some additional definitions, to describe certain subgroups in a 
saturated fusion system.

\begin{Defi} \label{d:subgroups}
Fix a prime $p$, a discrete $p$-toral group $S$, and a saturated fusion 
system $\calf$ over $S$.  Let $P\le{}S$ be any subgroup. Note that 
by Definition \ref{d:saturated}, $\outf(P)$ is finite  whether or not 
$P$ is fully normalized (see also \cite[Proposition 2.3]{BLO3}).
\begin{itemize} 

\item $P$ is \emph{$\calf$-centric} if $C_S(Q)=Z(Q)$ for each 
$Q\in{}P^\calf$, and is \emph{$\calf$-radical} if $O_p(\outf(P))=1$. 

\item $P$ is \emph{$\calf$-essential} if $P<S$, $P$ is $\calf$-centric and 
fully normalized in $\calf$, and $\outf(P)$ contains 
a strongly $p$-embedded subgroup. Here, a proper subgroup $H<G$ of a finite 
group $G$ is \emph{strongly $p$-embedded} if $p\big||H|$, and 
$p{\nmid}|H\cap gHg^{-1}|$ for each $g\in{}G{\sminus}H$. Let $\EE\calf$ 
denote the set of all $\calf$-essential subgroups of $S$.  

\item $P$ is \emph{normal in $\calf$} ($P\nsg\calf$) if each morphism 
$\varphi\in\homf(Q,R)$ in $\calf$ extends to a morphism 
$\widebar{\varphi}\in\homf(PQ,PR)$ such that $\widebar{\varphi}(P)=P$. The 
maximal normal $p$-subgroup of a saturated fusion system $\calf$ is denoted 
$O_p(\calf)$.  

\item $P$ is \emph{strongly closed in $\calf$} if for each $g\in{}P$, 
$g^\calf\subseteq P$. 

\end{itemize}
\end{Defi}

\begin{Prop} \label{AFT-E}
Let $\calf$ be a saturated fusion system over a discrete $p$-toral group 
$S$. 
\begin{enuma} 

\item Each morphism in $\calf$ is a composite of restrictions of elements 
in $\autf(P)$ for $P\le S$ that is fully normalized in $\calf$, 
$\calf$-centric and $\calf$-radical.

\item Each morphism in $\calf$ is a composite of restrictions of elements 
in $\autf(P)$ for $P\in\EE\calf\cup\{S\}$.

\item For each $Q\nsg{}S$, $Q\nsg\calf$ if and only if for each 
$P\in\EE\calf\cup\{S\}$, $Q\le{}P$ and $Q$ is $\autf(P)$-invariant. 
\end{enuma}
\end{Prop}

\begin{proof} Point (a) is shown in \cite[Theorem 3.6]{BLO3}.

By \cite[Proposition 2.3]{BLO3}, $\outf(P)=\autf(P)/\Inn(P)$ is always 
finite. For each such $P<S$ that is not $\calf$-essential, $\autf(P)$ is 
generated by automorphisms that can be extended to strictly larger 
subgroups: this is shown in \cite[Proposition I.3.3]{AKO} in the finite 
case, and the same argument applies when $S$ is infinite. Point (b) now 
follows from (a) and induction, and (in the infinite case) since there are 
only finitely many $S$-conjugacy classes of subgroups of $S$ that are 
$\calf$-centric and $\calf$-radical \cite[Corollary 3.5]{BLO3}. 

Point (c) follows easily from (b), just as in the finite case 
\cite[Proposition I.4.5]{AKO}.
\end{proof}

\begin{Defi} \label{d:E<|F}
Let $\calf$ be a saturated fusion system over a discrete $p$-toral group 
$S$. A saturated fusion subsystem $\cale$ over $T\le S$ is \emph{normal} 
in $\calf$ ($\cale\nsg\calf$) if 
\begin{itemize} 
\item $T$ is strongly closed in $\calf$ (in particular, $T\nsg S$);

\item (invariance condition) each $\alpha\in\autf(T)$ is fusion preserving 
in the sense that it extends to an automorphism of $\cale$;  

\item (Frattini condition) for each $P\le T$ and each 
$\varphi\in\homf(P,T)$, there are $\alpha\in\autf(T)$ and 
$\varphi_0\in\Hom_\cale(P,T)$ such that $\varphi=\alpha\circ\varphi_0$; and 

\item (extension condition) each $\alpha\in\Aut_\cale(T)$ extends to some 
$\4\alpha\in\autf(TC_S(T))$ such that $[\4\alpha,C_S(T)]\le Z(T)$.
\end{itemize}
The fusion system $\calf$ is \emph{simple} if it contains no proper 
nontrivial normal subsystems.
\end{Defi}

For further discussion of the definition and properties of normal 
fusion subsystems, we refer to \cite[\S\,I.6]{AKO} or \cite[\S\S\,5.4 \& 
8.1]{Craven} (when $S$ and $T$ are finite) and to \cite[Definition 
2.8]{Gonzalez} (in the general case). Note in particular the 
different definition used in \cite{Craven} and in \cite{Gonzalez}: a 
saturated fusion subsystem $\cale\le\calf$ over a subgroup $T$ that is 
strongly closed in $\calf$ is normal if the extension condition holds, and 
also the \emph{strong invariance condition}: for each $P\le Q\le T$, and 
each $\varphi\in\Hom_\cale(P,Q)$ and $\psi\in\homf(Q,T)$, 
$\psi\circ\varphi\circ(\psi|_P)^{-1}\in \Hom_\cale(\psi(P),T)$. When $S$ is 
finite, this is equivalent to the above definition by \cite[Proposition 
I.6.4]{AKO}, and a similar argument (made more complicated because there 
can be infinitely many subgroups) applies when $S$ and $T$ are $p$-toral. 

Since the Frattini condition will be important in Section \ref{s:infinite}, 
we work here with the above definition. However, none of the examples 
over infinite discrete $p$-toral groups considered here contains a nontrivial 
proper strongly closed subgroup (see Lemma \ref{str.cl.}), so these details 
make no difference as to which of them are simple or not.

\begin{Prop} \label{G&F}
Fix a prime $p$, and let $\calf$ be a saturated fusion system over an 
infinite discrete $p$-toral group $S$. Let $S_0$ be the identity component 
of $S$, and assume that each element of $S$ is $\calf$-conjugate to an 
element of $S_0$.
\begin{enuma}

\item If $\calf$ is realized by a compact Lie group $G$ with identity 
connected component $G_0$, then $G/G_0$ has order prime to $p$. If in 
addition, $\calf$ is simple, then $\calf$ is realized by the connected, 
simple group $G_0/Z(G_0)$, where $Z(G_0)$ is finite of order prime to $p$.

\item If $\calf$ is realized by a $p$-compact group $X$, then $X$ is 
connected. If in addition, $\calf$ is simple, then so is $X$.

\end{enuma}
In either case, if $\calf$ is simple, then the action of the Weyl group 
$\autf(S_0)$ on the $\Q_p$-vector space 
$\Q\otimes_{\Z}\Hom(S_0,\Q_p/\Z_p)$ is irreducible and generated 
by pseudoreflections.
\end{Prop}

\begin{proof} If $\calf=\calf_S(G)$ where $G$ is a compact Lie group with 
identity connected component $G_0$ and maximal discrete $p$-toral subgroup 
$S$, then $S\cap G_0$ is strongly closed in $\calf$ since $G_0\nsg G$, 
$S_0\le S\cap G_0$, and so $S\le G_0$. Hence $G/G_0$ has order prime to 
$p$. Also, $\calf_S(G_0)\nsg\calf_S(G)$: the invariance and extension 
conditions are easily checked, and the Frattini condition holds since 
$G=G_0N_G(S)$ by the Frattini argument. 

If in addition, $\calf$ is simple, then $\calf_S(G_0)=\calf_S(G)=\calf$, 
and $Z(G_0)$ is finite of order prime to $p$ since $Z(\calf)=1$. Hence 
$\calf$ is also realized by $G_0/Z(G_0)$, which is simple. 

If $\calf\cong\calf_S(X)$ for some $p$-compact group $X$ with 
$S\in\sylp{X}$, then $X$ is connected by \cite[Proposition 4.9(a)]{GLR}. 
Then $X$ is a central product of connected, simple $p$-compact groups, and 
hence is simple if $\calf$ is simple. 

Whenever $\calf$ is realized by a connected $p$-compact group $X$ (possibly 
a compact connected Lie group), then by \cite[Theorem 9.7(ii)]{DW}, the 
action of the Weyl group $\autf(S_0)$ on $\Q\otimes_{\Z}H^2(BS_0;\Z_p)$ is 
generated by pseudoreflections, where 
	\beq H^2(BS_0;\Z_p) \cong H^2(S_0;\Z_p) \cong H^1(S_0;\Q_p/\Z_p) 
	\cong \Hom(S_0,\Q_p/\Z_p). \eeq
If this is not irreducible as a group generated by pseudoreflections, then 
by the classification of connected $p$-compact groups in \cite[Theorem 
1.2]{AGMV} (for $p$ odd) and in \cite[Theorem 1.1]{AG} or \cite[Corollary 
1.2]{Moller} (for $p=2$), $X$ must be a nontrivial central product of 
simple factors, and hence $\calf$ is not simple. 
\end{proof}

We also recall the definition of a \emph{reduced} fusion system, but only 
for fusion systems over finite $p$-groups. Recall \cite[\S\,I.7]{AKO} 
that in this setting, $O^p(\calf)$ and $O^{p'}(\calf)$ are the smallest 
(normal) fusion systems in $\calf$ of $p$-power index and of index prime to 
$p$, respectively. 

\begin{Defi} \label{d:red-simple}
A saturated fusion system $\calf$ over a finite $p$-group $S$ is 
\emph{reduced} if $O_p(\calf)=1$, and $O^p(\calf)=\calf=O^{p'}(\calf)$.  
\end{Defi}

For each saturated fusion system $\calf$ over a finite $p$-group $S$, 
$\calf_{O_p(\calf)}(O_p(\calf))$, $O^p(\calf)$, and $O^{p'}(\calf)$ are all 
normal fusion subsystems.  Hence $\calf$ is reduced if it is simple. 
Conversely, if $\cale\nsg\calf$ is any normal subsystem over the subgroup 
$T\nsg{}S$, then by definition of normality, $T$ is strongly closed in 
$\calf$. Since each normal fusion subsystem over $S$ itself has index prime 
to $p$, a reduced fusion system is simple if it has no proper nontrivial 
strongly closed subgroups.

When $\calf$ is a saturated fusion system over an infinite discrete 
$p$-toral group $S$, there are well defined normal subsystems $O^p(\calf)$ 
(see \cite[Appendix B]{Gonzalez}), and $O^{p'}(\calf)$ (see 
\cite[A.14--A.16]{GLR}), with the same properties as in the finite 
case. So we could define reduced fusion systems in this 
context just as in the finite case. However, to simplify the discussion, 
and because we don't know whether or not infinite reduced fusion systems 
have the same properties that motivated the definition in the finite case 
(see \cite[Theorems A \& B]{AOV1}), we restrict attention to simple fusion 
systems in the infinite setting. 

\section{Reduced or simple fusion systems over nonabelian discrete 
$p$-toral groups with index $p$ abelian subgroup}
\label{s:fin+inf}

In this section, $p$ is an arbitrary prime.  We want to study simple 
fusion systems over nonabelian discrete $p$-toral groups (possibly finite) 
which contain an abelian subgroup of index $p$. Most of the results here 
were shown in \cite{indp2}, but only in the case where $|A|<\infty$ and $p$ 
is odd.

We first fix some notation which will be used throughout the rest of the 
paper. As usual, for a group $S$, we define $Z_m(S)$ for all $m\ge1$ by 
setting $Z_1(S)=Z(S)$, and setting $Z_m(S)/Z_{m-1}(S)=Z(S/Z_{m-1}(S))$ 
for $m\ge2$.

\begin{Not} \label{n:not1}
Fix a nonabelian discrete $p$-toral group $S$ with a unique abelian subgroup 
$A$ of index $p$, and a saturated fusion system $\calf$ over $S$. Define 
	\[ S'=[S,S]=[S,A]\,,\quad Z=Z(S)=C_A(S)\,,\quad Z_0=Z\cap S'\,,\quad 
	Z_2=Z_2(S)\,. \]
Thus $Z_0\le Z\le Z_2$ and $Z_0\le S'\le A$.  Also, set
	\begin{align*} 
	\calh &= \bigl\{ Z\gen{x} \,\big|\, x\in S{\sminus}A \bigr\} &&&
	G &= \autf(A) \\
	\calb &= \bigl\{ Z_2\gen{x} \,\big|\, x\in S{\sminus}A \bigr\} 
	& \qquad& &
	\UUU &= \Aut_S(A)\in\sylp{G} \,. 
	\end{align*}
\end{Not}

Recall that by \cite[Theorem 2.1]{indp1}, if $p$ is odd, and $S$ is finite 
and nonabelian and has more than one abelian subgroup of index $p$, then 
either $S$ is extraspecial of order $p^3$ (and the reduced fusion systems 
over $S$ were described in \cite{RV}), or there are no reduced fusion 
systems over $S$. So in the finite case, the restriction about the 
uniqueness of $A$ is just a convenient way to remove certain cases that 
have already been handled. We will show later (Corollary \ref{>1abel}) that 
in the infinite case (also when $p=2$), $A$ is unique whenever 
$O_p(\calf)=1$.

\begin{Lem} \label{B&H}
Assume Notation \ref{n:not1}. Then 
$\EE\calf\subseteq\{A\}\cup\calh\cup\calb$, and $|N_S(P)/P|=p$ for each 
$P\in \EE\calf$. If $\EE\calf\not\subseteq\{A\}$, then $Z_2\le 
A$ and $|Z_2/Z|=p$.
\end{Lem}

\begin{proof} Fix some $P\in\EE\calf\sminus\{A\}$. Then $P\nleq A$ since 
$P$ is $\calf$-centric.  Set $P_0=P\cap{}A$, and fix some element 
$x\in{}P{\sminus}P_0$. Since $\outf(P)$ is finite (Definition 
\ref{d:saturated}) and contains a strongly $p$-embedded subgroup, we have 
that $O_p(\outf(P))=1$ (cf. \cite[Proposition A.7(c)]{AKO}).

We must show that $P\in\calh\cup\calb$, $|N_S(P)/P|=p$, $Z_2\le A$, and 
$|Z_2/Z|=p$. (Clearly, $|N_S(P)/P|=p$ if $P=A$.) 

\noindent\textbf{Case 1: } Assume $P$ is nonabelian.  Since $Z\le P$ ($P$ 
is $\calf$-centric), $Z(P)=C_{P_0}(x)=Z$.  For each 
$g\in{}N_A(P){\sminus}P$, $c_g$ is the identity on $P_0$ and on $P/P_0$.  
If $P_0$ is characteristic in $P$, then $c_g\in{}O_p(\autf(P))$ by Lemma 
\ref{mod-Fr}, which is impossible since $O_p(\outf(P))=1$.  Thus $P_0$ is 
not characteristic in $P$, and hence is not the unique abelian subgroup of 
index $p$ in $P$. So by Lemma \ref{1abel}, $|P_0/Z|=p$ and 
$|[P,P]|=|[x,P_0]|=p$. Also, $P/Z$ is abelian since $|P/Z|=p^2$, so 
$[x,P_0]\le Z$, and hence $P_0\le Z_2$. Note that $P_0>Z$, since $P$ is 
nonabelian.

If $P_0<Z_2$, then for $y\in Z_2\sminus P$, $[y,P]\le Z=Z(P)$, so $y\in 
N_S(P)\sminus P$ and $c_y\in O_p(\Aut(P))$, contradicting the assumption 
that $P$ is $\calf$-essential. Thus $P_0=Z_2$, so $P\in\calb$, $Z_2\le A$, 
and $|Z_2/Z|=p$. Finally, $|N_S(P)/P|=p$ by Lemma \ref{|S|=p}, applied with 
$\outf(P)$ in the role of $G$, $\Out_S(P)\cong N_S(P)/P$ in the role of 
$S$, and $P/Z_0$ (if $P\in\calb$) in the role of $A$. Note that 
$[P,P]=[x,Z_2]\le Z_0$ and $C_{P/Z_0}(N_S(P))=Z_2/Z_0$.

\noindent\textbf{Case 2: } If $P\in\EE\calf$ is abelian, 
then $P_0=Z$: it contains $Z$ since $P$ is centric, and cannot be larger 
since then $P$ would be nonabelian.  Hence $P=Z\gen{x}\in\calh$.  Also, 
$\Aut_A(P)=\Aut_S(P)\in\sylp{\autf(P)}$ centralizes $P_0$.  The conditions 
of Lemma \ref{|S|=p} thus hold (with $P$ and $\autf(P)$ in the roles of $A$ 
and $G$), so $|N_S(P)/P|=|\Aut_S(P)|=p$.  Since $[S:P]=|A/Z|>p$ by Lemma 
\ref{1abel} and since $A$ is the 
unique abelian subgroup of index $p$, this implies that $S/Z$ is 
nonabelian, so $[x,A]\nleq Z$, and $Z_2\le A$.  

For each $g\in{}A$, $g\in{}N_S(P)$ if and only if $[g,x]\in{}P_0=Z$, if and 
only if $gZ\in{}C_{A/Z}(x)=Z(S/Z)=Z_2/Z$.  Thus $N_A(P)=Z_2$, 
$N_S(P)=Z_2\gen{x}=Z_2P$, and $|Z_2/Z|=|N_S(P)/P|=p$.
\end{proof}

\begin{Lem} \label{exclusion} 
Let $S$, $\calf$, etc. be as in Notation \ref{n:not1}. If, for some $x\in 
S\sminus A$, $Z_2\gen{x}\in\EE\calf$, then $Z\gen{x}$ is not 
$\calf$-centric, and hence $Z\gen{x}\notin\EE\calf$.
\end{Lem}

\begin{proof} Assume $x\in S\sminus A$ and $Z_2\gen{x}\in\EE\calf$, and set 
$P=Z_2\gen{x}\in\calb$. In particular, $Z_2<A$ since $Z_2\gen{x}<S$. 
Also, $|Z_2/Z|=p$ by Lemma \ref{B&H}, so $|P/Z|=p^2$, and $Z_2$ is not 
normalized by $\autf(P)$ since $P$ is essential. 

Let $\calp$ be the set of all subgroups of index $p$ in $P$ which contain 
$Z$. Then $\calp\supsetneqq\{Z_2\}$, so $P/Z\cong C_p^2$ (i.e., is not 
cyclic), and $\Aut_S(P)$ permutes transitively the $p$ members of 
$\calp\sminus\{Z_2\}$. So $\autf(P)$ must act transitively on $\calp$, 
hence $Z\gen{x}$ is $\calf$-conjugate to $Z_2$, and is not $\calf$-centric 
(recall $Z_2<A$). 
So $Z\gen{x}\notin\EE\calf$ in this case. 
\end{proof}

\begin{Lem} \label{|Z0|=p}
Assume Notation \ref{n:not1}, and also $A\nnsg\calf$ ($\iff$ 
$\EE\calf\not\subseteq\{A\}$). Then $|Z_0|=p$, and $|Z_i(S)/Z_{i-1}(S)|=p$ 
for all $i>1$ such that $Z_i(S)<S$. If $|A|<\infty$, then $|A/ZS'|=p$.
\end{Lem}

\begin{proof} Fix $x\in{}S{\sminus}A$. Let $\psi\in\End(A)$ be the 
homomorphism $\psi(g)=[g,x]$. Thus $\Ker(\psi)=Z$ and $\Im(\psi)=S'$.

By Lemma \ref{B&H} and since $\EE\calf\not\subseteq\{A\}$, 
$(\calh\cup\calb)\cap\EE\calf\ne\emptyset$, $Z_2\le A$, and $|Z_2/Z|=p$. 
Since $Z_2/Z=C_{A/Z}(x)$, $Z_2=\psi^{-1}(Z)$, and so $\psi$ sends $Z_2$ 
onto $Z_0=Z\cap S'$ with kernel $Z$. Thus $|Z_0|=|Z_2/Z|=p$.

Set $Z_i=Z_i(S)$ for each $i\ge0$, and let $k>2$ be the smallest index such 
that $Z_k=S$. Thus $S/Z_{k-2}$ is nonabelian, so $Z_{k-2}\le A$, and 
$Z_{k-1}/Z_{k-2}=Z(S/Z_{k-2})\le A/Z_{k-2}$. Hence $Z_i\le A$, and 
$Z_i=\psi^{-1}(Z_{i-1})$, for all $i<k$. In particular, $\psi$ induces a 
monomorphism from $Z_i/Z_{i-1}$ into $Z_{i-1}/Z_{i-2}$ for each $3\le i<k$, 
so $|Z_i/Z_{i-1}|\le p$, with equality since $Z_i(S)>Z_{i-1}(S)$ whenever 
$Z_{i-1}(S)<S$.

If $|A|<\infty$, then $|ZS'|=|Z|\cdot|S'|\big/|Z_0|=|A|/|Z_0|$, and hence 
$ZS'$ has index $p$ in $A$.
\end{proof}

\begin{Lem} \label{l2:s/a}
Let $A\nsg{}S$, $\calf$, $\calh$, $\calb$, etc., be as in Notation 
\ref{n:not1}. Assume $P\in\EE\calf$ where $P\in\calh\cup\calb$, and set 
$X=1$ if $P\in\calh$ and $X=Z_0$ if $P\in\calb$. Define $P_1,P_2\le P$ by 
setting 
	\[ P_1/X = C_{P/X}(O^{p'}(\autf(P))) \qquad\textup{and}\qquad
	P_2 = [O^{p'}(\autf(P)),P]. \]
Then $O^{p'}(\outf(P))\cong\SL_2(p)$, and the following hold.
\begin{enuma} 

\item If $P\in\calh$, then $P_1<Z$, $Z=P_1\times Z_0$, $Z_0<P_2\cong C_p^2$, 
and $P=P_1\times P_2$. If $p$ is odd, then $P_1$ is the unique 
$\autf(Z)$-invariant subgroup of $Z$ such that $Z=P_1\times Z_0$. 

\item If $P\in\calb$, then $P_1=Z$, $P_2$ is extraspecial of order $p^3$, 
$P_2\cong Q_8$ if $p=2$ while $P_2$ has exponent $p$ if $p$ is odd, and 
$P_1\cap P_2 = Z(P_2) = Z_0 = [P,P]$. Thus $P=P_1\times_{Z_0}P_2$. 

\end{enuma}
\end{Lem}

\begin{proof} To simplify notation, set $H=\outf(P)$, $H_0=O^{p'}(H)$, and 
$T=\Out_S(P)\in\sylp{H}$. 

If $P\in\calb$, then $[P,P]\le Z(P)\cap S'=Z\cap S'=Z_0$, with equality 
since $|Z_0|=p$ by Lemma \ref{|Z0|=p}. Thus $P/X$ is abelian in both cases. 
Also, $[N_S(P),P/X]=Z_0$ (if $P\in\calh$) or $Z_2/Z_0$ (if $P\in\calb$), 
and thus has order $p$ in both cases. So by Proposition \ref{GonA}, applied 
to the $H$-action on $P/X$, we have $H_0\cong\SL_2(p)$, 
$P/X=(P_1/X)\times(P_2/X)$, and $P_2/X\cong C_p^2$. 

If $P\in\calh$ (so $X=1$), then $P_1=C_P(H_0)\le C_P(T)=Z$, and $[Z:P_1]=p$ 
since $[P:P_1]=p^2$. Also, $P_2\ge[T,P]=Z_0$, so $P_1\cap Z_0=1$, and 
$Z=P_1\times Z_0$. If $p$ is odd, then $N_{H_0}(T)$ is a semidirect product 
of the form $C_p\rtimes C_{p-1}$. Fix $\alpha\in N_{H_0}(T)$ of order 
$p-1$; then $\alpha$ acts on $Z_0=[T,P]$ with order $p-1$ and acts 
trivially on $P_1$. Thus $\alpha|_Z\in\autf(Z)$, and $P_1$ is the only 
subgroup which is a complement to $Z_0$ in $Z$ and could be normalized by 
$\autf(Z)$. Since $\autf(Z)$ has order prime to $p$, there is at least one 
such subgroup, and hence $P_1$ is $\autf(Z)$-invariant.

If $P\in\calb$, then $X=Z_0$, and $P_2/Z_0\cong C_p^2$. Also, 
$H_0\cong\SL_2(p)$ acts faithfully on $P_2$, and this is possible only if 
$Z(P_2)=Z_0$, and $P_2\cong Q_8$ (if $p=2$) or $P_2$ is extraspecial of 
exponent $p$ (if $p$ is odd). Also, $P_1$ has index $p^2$ in $P$ since 
$P_1\cap P_2=Z_0$, $P_1\le Z(P)$ since $P=P_1\times_{Z_0}P_2$ is a central 
product, and hence $P_1=Z(P)=Z$.
\end{proof}

\begin{Cor} \label{c2:s/a}
In the situation of Notation \ref{n:not1}, if $A\nnsg\calf$ (i.e., if 
$\EE\calf\not\subseteq\{A\}$) and $p$ is odd, then $S$ splits over $A$: 
there is $x\in S\sminus A$ of order $p$.
\end{Cor}

\begin{proof} Fix $P\in\EE\calf\sminus\{A\}$. By Lemma \ref{B&H}, 
$P\in\calh\cup\calb$. In either case, by Lemma \ref{l2:s/a}, there is 
$x\in P\sminus A$ of order $p$.
\end{proof}


We now restrict to the case where $p$ is odd. Recall that $G=\autf(A)$ 
by Notation \ref{n:not1}.

\begin{Lem} \label{O_p(F)=1}
Assume Notation \ref{n:not1}, and also that $p$ is odd and $A\nnsg\calf$. 
Then $O_p(\calf)=1$ if and only if either there are no nontrivial 
$G$-invariant subgroups of $Z$, or $\EE\calf\cap\calh\ne\emptyset$ and 
$Z_0$ is the only $G$-invariant subgroup of $Z$.
\end{Lem}

\begin{proof} The following proof is essentially the same as the proof in 
\cite[Lemma 2.7(a)]{indp2} in the finite case. 

Assume first that $Q\defeq O_p(\calf)\ne1$. Since $A\nnsg\calf$, there 
is $P\in\EE\calf\sminus\{A\}\subseteq\calb\cup\calh$. If $P\in\calh$, then 
$Q\le Z$: the intersection of the subgroups $S$-conjugate to $P$. If 
$P\in\calb$, then $Q\le Z_2$ by a similar argument, and then $Q\le Z$ since 
that is the intersection of the subgroups in the $\autf(P)$-orbit of $Z_2$. 
Thus $Q$ is a non-trivial $G$-invariant subgroup of $Z$. If $Q=Z_0$, then 
$\EE\calf\cap\calh=\emptyset$, since for $P\in\EE\calf\cap\calh$, $Z_0$ is 
not normalized by $\autf(P)$. This proves one implication.

Conversely, assume that $1\ne{}R\le Z$ is $G$-invariant. For each 
$\alpha\in\autf(S)$, $\alpha(A)=A$ since $A$ is the unique abelian subgroup 
of index $p$, so $\alpha|_A\in{}G$, and thus $\alpha(R)=R$. Since each 
element of $\autf(Z)$ extends to $S$ by the extension axiom, $R$ is also 
normalized by $\autf(Z)$. Also, for each $P\in\EE\calf\cap\calb$, $Z=Z(P)$ 
is characteristic in $P$ and so $R$ is also normalized by $\autf(P)$. In 
particular, if $\EE\calf\cap\calh=\emptyset$, then $R\nsg\calf$, and so 
$O_p(\calf)\ne1$. 

Now assume that $\EE\calf\cap\calh\ne\emptyset$, and also that $R\ne{}Z_0$. 
By \cite[Lemma 2.3(b)]{indp2} (the argument easily extends to the infinite 
case), there is a unique $\autf(Z)$-invariant factorization $Z=Z_0\times 
\til{Z}$.  Set $\til{R}=R\cap{}\til{Z}$. If $R\ge Z_0$, then $R=\til{R}\times Z_0$. 
Otherwise, $R\cap Z_0=1$ (recall $|Z_0|=p$), and since $R$ is 
$\autf(Z)$-invariant, the uniqueness of the splitting implies that $R\le 
\til{Z}$ and hence $R=\til{R}$. Since $R\ne{}Z_0$, we have $\til{R}\ne1$ in either 
case. 

For each $\varphi\in\autf(A)=G$, $\varphi(\til{R})\le R\le Z$, so by the 
extension axiom, $\varphi|_{\til{R}}$ extends to some 
$\4\varphi\in\autf(S)$, and $\varphi(\til{R})=\4\varphi(\til{R})=\til{R}$ 
since $\til{R}$ is $\autf(Z)$-invariant. So by the same arguments as those 
applied above to $R$, $\til{R}$ is normalized by $\autf(P)$ for each 
$P\in(\{S\}\cup\EE\calf)\sminus\calh$. If $P\in\EE\calf\cap\calh$, then for 
each $\alpha\in\autf(P)$, $\alpha(\til{Z})=\til{Z}$ by \cite[Lemma 
2.3(b)]{indp2}, so $\alpha|_{\til{Z}}$ extends to an element of $\autf(S)$ 
and hence of $\autf(Z)$, and in particular, $\alpha(\til{R})=\til{R}$. Thus 
$1\ne{}\til{R}\nsg\calf$, and hence $O_p(\calf)\ne1$. 
\end{proof}

Without the assumption that $p$ be odd in Lemma \ref{O_p(F)=1}, the 
2-fusion system $\calf$ of $\PSSL_2(q^2)$ is a counterexample for each 
prime power $q\equiv\pm1$ (mod $8$). Here, 
$\PSSL_2(q^2)=\PSL_2(q^2)\gen{\theta}$ where $\theta$ acts on 
$\PSL_2(q^2)$ as a field automorphism of order 2 (and $\theta^2=1$). Then 
$O_2(\calf)=1$, $S\cong D_{2^m}\times C_2$ for some $m\ge4$ depending on $q$, 
$A\cong C_{2^{m-1}}\times C_2$, $Z=Z(S)=\Omega_1(A)$, and $G=\Aut_S(A)$ acts 
trivially on $Z$ (so that all subgroups of $Z$ are $G$-invariant).

\begin{Lem} \label{A2<A1<A}
Assume Notation \ref{n:not1}, and also that $p$ is odd and $O_p(\calf)=1$. 
Let $A_2\le A_1\le A$ be $G$-invariant subgroups such that $A_1\le ZA_2$. 
Then either $A_1=A_2$, or $A_1=Z_0\times A_2$ and $Z_0$ is $G$-invariant.
\end{Lem}

\begin{proof} Fix a class $xA_2\in{}A_1/A_2$.  By assumption, we can assume 
$x\in{}Z$. Since $G$ acts on $A_1/A_2$ and $\UUU$ acts trivially on this 
quotient, $G_0=O^{p'}(G)$ also acts trivially. Hence $\alpha(x)\in xA_2$ 
for each $\alpha\in G_0$. Let $\alpha_1,\ldots,\alpha_k\in{}G_0$ be left 
coset representatives for $\UUU$ (so $p\nmid{}k=[G_0:\UUU]$), and set 
$y=\Bigl(\prod_{i=1}^k\alpha_i(x)\Bigr)^{1/k}$.  Then $y\in{}xA_2$ since 
$\alpha_i(x)\in xA_2$ for each $i$, and $y\in C_A(G_0)$.  This shows that 
$A_1\le C_A(G_0)A_2$.  

Now, $C_A(G_0)$ is a subgroup of $Z=C_A(\UUU)$ normalized by $G$. So by 
Lemma \ref{O_p(F)=1} and since $O_p(\calf)=1$, $C_A(G_0)\le Z_0$.  Thus 
$A_1\le Z_0A_2$. If $A_1>A_2$, then $A_1=A_2\times Z_0$ since $|Z_0|=p$, 
and $Z_0=C_A(G_0)$ is $G$-invariant.
\end{proof}

The following notation, taken from \cite[Notation 2.4]{indp2}, will be used 
throughout the rest of the paper.

\newcommand{\autv}{\Aut^\vee}

\begin{Not} \label{n:not2}
Assume Notation \ref{n:not1}, and also that $|Z_0|=p$. Set 
	\[ \Delta = (\Z/p)^\times \times (\Z/p)^\times\,,
	\qquad\textup{and}\qquad
	\Delta_i=\{(r,r^i)\,|\,r\in(\Z/p)^\times\}\le\Delta
	\quad\textup{(for $i\in\Z$).} \]
Set
	\begin{align*} 
	\autv(S)&=\bigl\{\alpha\in\Aut(S) \,\big|\, 
	[\alpha,Z]\le Z_0\bigr\}, & 
	\autv(A)&=\bigl\{\alpha|_A \,\big|\, \alpha\in\autv(S) \bigr\} \\
	\autff(S)&=\autv(S)\cap\autf(S) &
	\autff(A)&=\autv(A)\cap\autf(A) \\
	&&&\hskip-9mm= \bigl\{\beta\in N_{\autf(A)}(\Aut_S(A)) \,\big|\, 
	[\beta,Z]\le Z_0 \bigr\} .
	\end{align*}
Define 
	\[ \mu\:\autv(S)\Right5{}\Delta \qquad\textup{and}\qquad
	\mu_A\:\autv(A)\Right5{}\Delta \]
by setting, for $\alpha\in\autv(S)$, 
	\[ \mu(\alpha) = (r,s) \quad\textup{if}\quad
	\begin{cases} 
	\alpha(x)\in x^rA & \textup{for $x\in{}S\sminus A$} \\
	\alpha(g)=g^s & \textup{for $g\in{}Z_0$ }
	\end{cases} \]
and $\mu_A(\alpha|_A)=\mu(\alpha)$ if $\alpha\in\autff(S)$.
\end{Not}

\section{Minimally active modules}
\label{s:min.act.}

\newcommand{\GG}{\mathscr{G}_p}
\newcommand{\GGG}{\mathscr{G}^\wedge_p}

In the earlier paper \cite{indp2}, the concept of ``minimally active'' 
modules played a central role when identifying the pairs $(A,\autf(A))$ 
that can occur in a simple fusion system $\calf$ over a $p$-group $S$ that 
contains an elementary abelian group $A$ with index $p$. Before continuing 
to study the structure of such $\calf$, we need to recall some of the 
notation and results in that paper, beginning with \cite[Definitions 3.1 \& 
3.3]{indp2}, and describe how they relate to the more general situation 
here. 

\begin{Defi} \label{d:min.act.}
For each prime $p$, 
\begin{itemize} 
\item $\GG$ is the class of finite groups $\genG$ with $\UU\in\sylp{\genG}$ such 
that $|\UU|=p$ and $\UU\nnsg \genG$; and 
\item $\GGG$ is the class of those $\genG\in\GG$ such that $|\Out_\genG(\UU)|=p-1$ 
for $\UU\in\sylp{\genG}$. 
\end{itemize}
For $\genG\in\GG$, an $\F_p\genG$-module is \emph{minimally active} if its 
restriction to $\UU\in\sylp{\genG}$ has exactly one Jordan block with 
nontrivial action. 
\end{Defi}

The next lemma explains the importance of minimally active modules 
here. In particular, it means that many of the tables and results in 
\cite[\S\,4--5]{indp2} can be applied to get information about $\autf(A)$ 
and $\Omega_1(A)$.

\begin{Lem} \label{min.act.}
Assume Notation \ref{n:not1} and \ref{n:not2}, and also that $p$ is odd, 
$A\in\EE\calf$, and $O_p(\calf)=1$. Set $V=\Omega_1(A)$ and 
	\begin{align*} 
	\autff(V) &= \bigl\{ \beta\in N_{\autf(V)}(\Aut_S(V)) \,\big|\, 
	[\beta,\Omega_1(Z)]\le Z_0 \bigr\} \\
	&= \bigl\{ \alpha|_V \,\big|\, \alpha\in\autf(S),~ 
	[\alpha,\Omega_1(Z)]\le Z_0 \bigr\}, 
	\end{align*}
and define $\mu_V\:\autff(V)\too\Delta$ by setting 
$\mu_V(\alpha|_V)=\mu(\alpha)$. Then 
\begin{enuma} 

\item $G=\autf(A)\in\GGG$; 

\item $V$, and $A/\Fr(A)$ if $|A|<\infty$, are both faithful, minimally 
active, and indecomposable as $\F_pG$-modules; and 

\item $\mu_A(\autff(A)) = \begin{cases} 
	\mu_V(\autff(V)) & \textup{if $Z_0\nleq\Fr(A)$} \\
	\mu_V(\autff(V))\cap\Delta_0 & \textup{if $Z_0\le\Fr(A)$.} 
	\end{cases}$

\end{enuma}
\end{Lem}

\begin{proof} \textbf{(a) } By assumption, $\UUU=\Aut_S(A)\in\sylp{G}$ has 
order $p$. Since $A\in\EE\calf$, $\UUU$ is not normal in $G=\autf(A)$, and 
hence $G\in\GG$. 

Since $O_p(\calf)=1$, there is $P\in\EE\calf\cap(\calh\cup\calb)$. By Lemma 
\ref{l2:s/a}(a,b), $O^{p'}(\outf(P))\cong\SL_2(p)$. Choose $\alpha\in 
O^{p'}(\autf(P))$ of order $p-1$ whose class in $\outf(P)$ normalizes 
$\Out_S(P)\cong C_p$; then $\alpha$ extends to an element of 
$\autf(N_S(P))$ and hence (since $P$ is maximal among $\calf$-essential 
subgroups) to some $\4\alpha\in\autf(S)$. Then $\4\alpha|_A$ normalizes 
$\UUU$ and its class in $\Aut_G(\UUU)$ has order $p-1$, so $G\in\GGG$.

\smallskip

\noindent\textbf{(b) } Set $\4A=A/\Fr(A)$. If $|A|<\infty$, then since 
$G$ acts faithfully on $A$, \cite[Theorems 5.2.4 \& 5.3.5]{Gorenstein} 
imply that $C_G(V)$ and $C_G(\4A)$ 
are both normal $p$-subgroups of $G$. Since $\UUU$ is not normal by (a), $G$ 
acts faithfully on $V$ and on $\4A$ in this case. If $|A|=\infty$, then 
$G$ acts faithfully on $\Omega_m(A)$ for $m$ large enough, and hence acts 
faithfully on $V$ by the above argument.

Since $|Z_0|=p$ by Lemma 
\ref{|Z0|=p}, where $Z_0=S'\cap Z=[\UUU,V]\cap C_V(\UUU)$, the 
$\F_p\UUU$-module $V|_{\UUU}$ has exactly one Jordan block with nontrivial 
action of $\UUU$. So $V$ is minimally active. If $|A|<\infty$, 
then $ZS'=C_A(\UUU)[\UUU,A]$ has index $p$ in $A$ by Lemma \ref{|Z0|=p}, so 
$C_{\4A}(\UUU)[\UUU,\4A]$ has index at most $p$ in $\4A$, and hence $\4A$ is 
minimally active.

If $V=V_1\times V_2$, where each $V_i$ is a nontrivial 
$\F_pG$-submodule, then by \cite[Lemma 3.4(a)]{indp2}, we can assume (after 
exchanging indices if needed) that $V_1\le Z$ and (since it is a summand) 
$V_1\cap Z_0=1$. But this contradicts Lemma \ref{O_p(F)=1}.

Assume $|A|<\infty$. If $\4A=\4X\times \4Y$ where $\4X,\4Y\le\4A$ 
are $\F_pG$-submodules, then by the Krull-Schmidt theorem, one of the factors, 
say $\4X$, contains a nontrivial Jordan block, while $\UUU$ acts trivially 
on the other factor. Thus $\4X\ge[\UUU,\4A]$ and $\4X\nleq ZS'/\Fr(A)$. Let 
$X\le A$ be such that $\Fr(A)\le X$ and $X/\Fr(A)=\4X$. Then $X\ge S'$ and 
$X\nleq ZS'$, so $XZ=A$. By Lemma \ref{A2<A1<A}, either $X=A$ (and 
$\4X=\4A$) or $A=X\times Z_0$ (which is impossible since $Z_0\le S'\le X$). 
Thus $\4X=\4A$, and $\4A$ is indecomposable.

\smallskip

\noindent\textbf{(c) } By definition, the restriction to $V$ of each 
element in $\autff(A)$ lies in $\autff(V)$, and hence 
$\mu_A(\autff(A))\le\mu_V(\autff(V))$. If $Z_0\le\Fr(Z)$, then choose $z\in Z$ 
such that $1\ne z^p\in Z_0$. For each $\beta\in\autff(A)$, 
$\beta(z)=z^{pk+1}$ for some $k$ since $[\beta,Z]\le Z_0$, and hence 
$\beta|_{Z_0}=\Id$ and $\mu_A(\beta)\in\Delta_0$. Thus 
$\mu_A(\autff(A))$ is contained in the right hand side in (c).

Now assume that $\beta\in\autff(V)$, where $\beta=\alpha|_V$ for 
$\alpha\in\autf(S)$. If $Z_0\le\Fr(A)$, then assume also that 
$\mu_V(\beta)\le\Delta_0$; i.e., that $\beta|_{Z_0}=\Id$. Upon replacing 
$\beta$ by $\beta^{p^k}$ and $\alpha$ by $\alpha^{p^k}$ for appropriate 
$k$, we can also assume that $\alpha$ has order prime to $p$ without 
changing $\mu(\alpha)$. Then $Z=C_Z(\alpha)\times[\alpha,Z]$ by 
\cite[Theorem 5.2.3]{Gorenstein} and since $Z$ is the union of the finite 
abelian $p$-groups $\Omega_i(Z)$, and 
$[\alpha,\Omega_1(Z)]=[\beta,\Omega_1(Z)]\le Z_0$ since 
$\beta\in\autff(V)$. Also, 
$\Omega_1([\alpha,Z])=[\alpha,\Omega_1(Z)]\le Z_0$ (since it can't be any 
larger). If $Z_0\nleq\Fr(Z)$, then this implies that $[\alpha,Z]\le Z_0$, 
hence that $\alpha\in\autff(S)$. If $Z_0\le\Fr(Z)$, then 
$\Omega_1([\alpha,Z])\le Z_0\le C_Z(\alpha)$ implies that 
$\Omega_1([\alpha,Z])=1$ and hence $[\alpha,Z]=1$, so again 
$\alpha\in\autff(S)$. Thus 
$\mu_V(\beta)=\mu_A(\alpha|_A)\in\mu_A(\autff(A))$, and the right hand side in 
(c) is contained in $\mu_A(\autff(A))$. 
\end{proof}

The following basic properties of minimally active indecomposable 
modules, taken from \cite{indp2}, play an important role in the rest of the 
paper.

\begin{Lem}[{\cite[Proposition 3.7]{indp2}}] \label{min.act.props}
Fix an odd prime $p$, a finite group $\genG\in\GG$, and $\UU\in\sylp{\genG}$. Let 
$V$ be a faithful, minimally active, indecomposable $\F_p\genG$-module. Then 
\begin{enuma} 
\item $\dim(V)\le p$ implies that $V|_\UU$ is indecomposable and thus 
contains a unique Jordan block;
\item $\dim(V)\ge p+1$ implies that $V|_\UU$ is the direct sum of a Jordan 
block of dimension $p$ and a module with trivial action of $\UU$; and 
\item $\dim(C_V(\UU))=1$ if $\dim(V)\le p$, while 
$\dim(C_V(\UU))=\dim(V)-p+1$ if $\dim(V)\ge p$.
\end{enuma}
\end{Lem}

The next lemma is closely related to \cite[Lemma 1.11]{indp2}.

\begin{Lem} \label{l:top&bottom}
Fix an odd prime $p$, let $\genG$ be a finite group such that $\UU\in\sylp{\genG}$ 
has order $p$, and set $N=N_\genG(\UU)$. Let $V$ be a faithful, minimally 
active, indecomposable $\F_p\genG$-module such that $\dim(V)\le p$. Then 
$C_V(\UU)$ and $V/[\UU,V]$ are both $1$-dimensional, and the following 
hold.
\begin{enuma} 

\item If $\dim(V)=p$, then $V/[\UU,V]$ and $C_V(\UU)$ are 
$1$-dimensional, and isomorphic as $\F_p[N/\UU]$-modules.

\item The projective cover and the injective envelope of $V|_N$ are both 
$p$-dimensional.

\item If $\dim(V)=p-1$, and there is an $\F_p\genG$-submodule $V_0<V$ with 
$\dim(V_0)=1$, then there is a projective $\F_p\genG$-module $W$ such that 
$\dim(W)=p$ and $W$ has a submodule isomorphic to $V$. 

\item Let $W$ be another $\F_p\genG$-module such that $\dim(W)=\dim(V)$, and 
assume that $C_W(\UU)\cong C_V(\UU)$ as 
$\F_p[N_\genG(\UU)/\UU]$-modules. Then $W\cong V$ as 
$\F_p[N_\genG(\UU)]$-modules, and as $\F_p\genG$-modules if $\dim(V)<p$.

\end{enuma}
\end{Lem}

\begin{proof} \textbf{(a) } By Lemma \ref{min.act.props}, $\UU$ acts on 
$V$ with only one Jordan block, so $\dim(C_V(\UU))=1$ and 
$\dim(V/[\UU,V])=1$. By \cite[Lemma 1.11(b)]{indp2}, if $g\in N_\genG(\UU)$ 
and $t\in(\Z/p)^\times$ are such that $g$ acts on $V/[\UU,V]$ via 
multiplication by $t$, then for some $r\in(\Z/p)^\times$, $g$ acts on 
$C_V(\UU)$ via multiplication by $tr^{m-1}=tr^{p-1}=t$. Thus $V/[\UU,V]$ 
and $C_V(\UU)$ are isomorphic as $\F_p[N_\genG(\UU)/\UU]$-modules. 

\smallskip

\noindent\textbf{(b) } By the Schur-Zassenhaus theorem, there is $H<N$ 
of index $p$ such that $N=H\UU$. Set $V_0=C_V(\UU)$, regarded as an 
$\F_p[N/\UU]$-module, and also as an $\F_pH$-module via the natural 
isomorphism $H\cong N/\UU$. Set $\5V=\Ind_H^N(V_0)$: a projective and 
injective $p$-dimensional $\F_pN$-module. Then 
	\[ \5V/[\UU,\5V] \cong \F_p[N/\UU]\otimes_{\F_pN} \F_pN 
	\otimes_{\F_pH} V_0 \cong \F_p[N/\UU]\otimes_{\F_pH} V_0 \cong 
	V_0, \]
and so $C_{\5V}(\UU)\cong V_0$ by (a). Thus $\5V$ is the injective 
envelope of $V_0$ when regarded as an $\F_pN$-module. In particular, an 
isomorphism $C_V(\UU)\cong C_{\5V}(\UU)$ extends to an $\F_pN$-linear 
homomorphism $V\too \5V$ which is injective since it sends the socle 
$C_V(\UU)$ injectively. Thus $\5V$ is an injective envelope of $V|_N$. 
The statement about projective covers is shown in a 
similar way (or by dualizing). 

\smallskip

\noindent\textbf{(c) } Assume that $\dim(V)=p-1$, and that there is an 
$\F_p\genG$-submodule $V_0<V$ with $\dim(V_0)=1$. Then $V_0=C_V(\UU)$, since 
this is the unique $1$-dimensional submodule of $V$ as an 
$\F_p\UU$-module. By (b), there is an injective (hence projective) 
$\F_pN$-module $\5W$ containing $V|_N$ as a submodule. By (a), 
$\5W/V\cong V_0|_N$ as $\F_pN$-modules. 

Consider the homomorphisms
	\[ \Ext_{\F_p\genG}^1(V_0,V) \Right4{\Phi_1} \Ext_{\F_pN}^1(V_0,V)
	\Right4{\Phi_2} \Ext_{\F_p\UU}^1(V_0,V) \]
induced by restrictions of rings. Since $\UU$ has index prime to $p$ in 
$\genG$, $\Phi_1$ and $\Phi_2$ are injective, and the images of $\Phi_2$ and 
$\Phi_2\Phi_1$ are certain subgroups of stable elements (see 
\cite[Proposition 3.8.2]{Benson1} for this version of the stable elements 
theorem). Since $\Ext_{\F_p}^1(V_0,V)=0$, we need only consider stability 
of elements with respect to automorphisms of $\UU$, and hence 
$\Im(\Phi_2)=\Im(\Phi_2\Phi_1)$. 

Thus $\Phi_1$ is an isomorphism. Interpreted in terms of extensions, this 
implies that there is an extension $0\too V\too W\too V_0\too 0$ of 
$\F_p\genG$-modules such that $W|_N\cong\5W$. In particular, $W$ is projective 
since $\5W$ is projective as an $\F_pN$-module.

\smallskip

\noindent\textbf{(d) } By (b), the injective envelope $\5V$ of $V|_N$ is 
$p$-dimensional. Hence $C_{\5V}(\UU)\cong C_V(\UU)\cong C_W(\UU)$, so 
$\5V$ is also the injective envelope of $C_W(\UU)$, and hence of $W|_N$ 
since $C_W(\UU)$ is its socle. Since $\dim(V)=\dim(W)$, and $\5V$ contains 
a unique $\F_pN$-submodule of each dimension $m\le p$, we conclude that 
$V\cong W$ are isomorphic as $\F_pN$-modules. 

If $\dim(V)<p$, then $\UU$ is a vertex of $V$ and of $W$ and they are 
the Green correspondents of $V|_N$ and $W|_N$, respectively (see 
\cite[\S\,3.12]{Benson1}). So $V\cong W$ as $\F_p\genG$-modules. 
\end{proof}

Point (d) need not hold if $\dim(V)=p$. As an example, fix $p\ge5$, set 
$\genG=\SL_2(p)$ and choose $\UU\in\sylp{\genG}$, let $V$ be the simple 
$p$-dimensional $\F_p\genG$-module, and let $W$ be the projective cover of the 
trivial 1-dimensional $\F_p\genG$-module. Using the fact that $V$ is the 
$(p-1)$-st symmetric power of the natural 2-dimensional $\F_p\genG$-module, it 
is not hard to see that $C_V(\UU)$ is 1-dimensional with trivial 
$N_\genG(\UU)/\UU$-action. The same holds for $W$ by construction, where 
$\dim(W)=p$. We refer to \cite[pp. 48--52]{Alperin} or the discussion in 
\cite[\S\,6]{indp2} for more detail.

A minimally active indecomposable $\F_p\genG$-module of dimension at least 
$p+2$ is simple by \cite[Proposition 3.7(c)]{indp2}. This is not true 
for modules of dimension $p+1$, but the following lemma gives some 
information about such modules.

\begin{Lem} \label{rk(V)=p+1}
Fix a finite group $\genG\in\GG$ with $\UU\in\sylp{\genG}$. Let $V$ be a finite, 
minimally active, indecomposable $\F_p\genG$-module of rank $p+1$. If $0\ne 
V_0<V$ is a proper nontrivial submodule, then $V_0|_\UU$ and 
$(V/V_0)|_\UU$ are both indecomposable $\F_p\UU$-modules with nontrivial 
action. In particular, $2\le\dim(V_0)\le p-1$.
\end{Lem}

\begin{proof} Recall $|\UU|=p$ since $\genG\in\GG$. By \cite[Proposition 
3.7(a)]{indp2}, $V|_\UU\cong\F_p\UU\oplus\F_p$; i.e., $V|_\UU$ has 
Jordan blocks of dimension $p$ and $1$.

Fix a proper nontrivial submodule $0\ne V_0<V$, and assume that $V_0|_\UU$ 
is decomposable or has trivial action as an $\F_p\UU$-module. We will 
first show that $V_0$ always has a 1-dimensional $\F_p\genG$-submodule, and 
then show that this is impossible. In particular, this shows that 
$\dim(V_0)\ge2$. The corresponding results for $V/V_0$ then follow by 
dualizing.

If $V_0|_\UU$ is decomposable with nontrivial action, then $V_0|_\UU$ is 
the sum of a $1$-dimensional module with trivial action and an 
indecomposalble module of dimension at most $p-1$. By \cite[Proposition 
3.7(a)]{indp2}, $V_0$ is decomposable as an $\F_p\genG$-module, and thus has a 
1-dimensional summand. 

If $\UU$ acts trivially on $V_0$, then $\dim(V_0)\le2$ and $O^{p'}(\genG)$ 
(the normal closure of $\UU$ in $\genG$) acts trivially on $V_0$. If 
$\dim(V_0)=2$, then $V_0=C_V(\UU)$, and $[\UU,V]\cap V_0$ is a 
1-dimensional subspace normalized by $N_\genG(\UU)$. Since 
$\genG=O^{p'}(\genG)N_\genG(\UU)$ by the Frattini argument, $[\UU,V]\cap V_0$ is a 
1-dimensional $\F_p\genG$-submodule. 

We are thus reduced to the case where $\dim(V_0)=1$. If $V_0\ne[\UU,V]$, 
then $(V/V_0)|_\UU$ is indecomposable (consists of one Jordan block), and 
hence is $\F_p\UU$-free. So $V/V_0$ is projective, contradicting the 
assumption that $V$ is indecomposable.

Thus $V_0=[\UU,V]$ is an $\F_p\genG$-submodule, and $V/V_0$ has Jordan blocks 
of length $1$ and $p-1$. So by \cite[Proposition 3.7(a)]{indp2}, it is 
decomposable: there are submodules $W_1,W_2<V$ such that 
$V/V_0=(W_1/V_0)\oplus(W_2/V_0)$, where $\dim(W_1)=p$, and 
$(W_1/V_0)|_\UU$ is an (indecomposable) Jordan block. If $[\UU,W_1]=0$, 
then $[\UU,W_2]=V_0$, and $V|_\UU$ contains Jordan blocks of dimension 
$p-1$ and $2$, which we saw is impossible. Thus $[\UU,W_1]=V_0$, so 
$W_1|_\UU$ is indecomposable, $W_1$ is projective and injective, and this 
again contradicts the assumption that $V$ is indecomposable.

This proves that $\UU$ acts nontrivially on $V_0$, and in particular, 
$\dim(V_0)\ge2$. A similar argument applied to the dual $V^*$ shows that 
$\UU$ acts nontrivially on $V/V_0$, and that $\dim(V_0)\le p-1$. 
\end{proof}

The following definitions will be useful.

\begin{Defi} \label{d:ZG-lattice}
For a finite group $\genG$, a \emph{$\Z_p\genG$-lattice} is a finitely 
generated $\Z_p\genG$-module that is free as a $\Z_p$-module (hence a lattice 
in a finitely generated $\Q_p\genG$-module). A \emph{discrete $\genG$-$p$-torus} is 
a discrete $p$-torus equipped with an action of $\genG$ by automorphisms.
\end{Defi}

Let $\Q_p(\zeta)\supseteq\Z_p[\zeta]$ denote the extensions of 
$\Q_p\supseteq\Z_p$ by a primitive $p$-th root of unity $\zeta$. When 
$\UU$ is a group of order $p$, we regard $\Q_p(\zeta)$ and $\Z_p[\zeta]$ 
as $\Z_p\UU$-modules under some choice of identification 
$\UU\cong\gen{\zeta}$.

\begin{Lem} \label{min.act.|Z0|=p}
Fix an odd prime $p$, a group $\genG\in\GG$, and $\UU\in\sylp{\genG}$. 
\begin{enuma} 

\item Let $\Lambda$ be a $\Z_p\genG$-lattice such that $\Lambda/p\Lambda$ is 
faithful and minimally active as an $\F_p\genG$-module. Then 
$\Lambda/C_\Lambda(\UU)\cong\Z_p[\zeta]$ as $\Z_p\UU$-modules, and 
$[\UU,\Lambda]+C_\Lambda(\UU)$ has index $p$ in $\Lambda$.

\item Let $A$ be a discrete $\genG$-$p$-torus such that $\Omega_1(A)$ is 
faithful and minimally active as an $\F_p\genG$-module. Then 
$A/C_A(\UU)\cong\Q_p(\zeta)/\Z_p[\zeta]$ as $\Z_p\UU$-modules, and 
$|[\UU,A]\cap C_A(\UU)|=p$.

\item Let $X$ be a finite, faithful $\Z_p\genG$-module. Then $\genG$ acts 
faithfully on $\Omega_1(X)$ and on $X/pX$. Among the following 
conditions:
\begin{enum1} 
\item $\Omega_1(X)$ is minimally active as an $\F_p\genG$-module.
\item $X/pX$ is minimally active as an $\F_p\genG$-module.
\item $|[\UU,X]\cap C_X(\UU)|=p$. 
\item $[\UU,X]+C_X(\UU)$ has index $p$ in $X$. 
\end{enum1}
we have $(1)\,\Longleftarrow\,(3)\iff(4)\implies(2)$. 
If $X\cong\Lambda/\Lambda_0$ 
for some $\Z_p\genG$-lattice $\Lambda$ and some submodule $\Lambda_0\le 
p\Lambda$, then all four conditions are equivalent. 

\item If $X$ is a finite, faithful $\Z_p\genG$-module such that condition 
\textup{(c.4)} holds, then for each $x\in X\sminus([\UU,X]+C_X(\UU))$, 
$X=C_X(\UU)+\Z_p\UU\cdot x$.

\end{enuma}
\end{Lem}

\begin{proof} Fix a generator $u\in\UU$.

\smallskip

\noindent\textbf{(a) } Set $M=\Q_p\otimes_{\Z_p}\Lambda$. By Lemma 
\ref{ZU-modules}(a), $M/C_M(\UU)$ is isomorphic, as a $\Q_p\UU$-module, 
to a sum of copies of $\Q_p(\zeta)$, where $\zeta$ is a primitive $p$-th 
root of unity. In particular, each Jordan block for the action of $\UU$ on 
$\Lambda/(C_\Lambda(\UU)+p\Lambda)$ has length at most $p-1$. Since 
$\Lambda/p\Lambda$ is minimally active, it follows that 
$M/C_M(\UU)\cong\Q_p(\zeta)$, since otherwise 
$\Lambda/(C_\Lambda(\UU)+p\Lambda)$ would have rank at least $2(p-1)$ and 
hence $\UU$ would be fixed by a submodule of rank at least $p\ge3$. Hence 
$\Lambda/C_\Lambda(\UU)\cong\Z_p[\zeta]$ by Lemma \ref{ZU-modules}(c). 

Consider the short exact sequence
        \[ 0 \Right2{} C_\Lambda(\UU) \Right4{\incl} \Lambda \Right4{\varphi} 
        [\UU,\Lambda] \Right2{} 0, \]
where $\varphi(x)=u(x)-x$ for all $x\in\Lambda$. 
We just saw that $[\UU,\Lambda]\cong\Z_p[\zeta]$. Under this 
identification, $\varphi|_{[\UU,\Lambda]}$ is multiplication by $1-\zeta$, 
and so its image has index $p$ in $[\UU,\Lambda]$. Thus 
$C_\Lambda(\UU)+[\UU,\Lambda]$ has index $p$ in $\Lambda$.

\smallskip

\noindent\textbf{(b) } Let $A$ be a discrete $\genG$-$p$-torus such that 
$\Omega_1(A)$ is faithful and minimally active as an $\F_p\genG$-module. Set 
$\Lambda=\Hom_{\Z_p}(A,\Q_p/\Z_p)$, regarded as a $\Z_p\genG$-lattice. Then 
$\Omega_1(A)\cong\Lambda/p\Lambda$ by Proposition \ref{tori&lattices}, so 
$\Lambda/p\Lambda$ is minimally active. We just saw, in the proof of (a), 
that this implies that $\Lambda/C_\Lambda(\UU)\cong\Z_p[\zeta]$. So after 
taking tensor products with $\Q_p/\Z_p$ and applying Proposition 
\ref{tori&lattices} again, we get that 
$A/(\Q_p/\Z_p\otimes_{\Z_p}C_\Lambda(\UU))\cong\Q_p(\zeta)/\Z_p[\zeta]$. 
Hence $A/C_A(\UU)\cong\Q_p(\zeta)/R$ for some $\Z_p\UU$-lattice 
$R<\Q_p(\zeta)$ that contains $\Z_p[\zeta]$ with finite index. Then 
$R\cong\Z_p[\zeta]$ by Lemma \ref{ZU-modules}(c), $\Q_p(\zeta)=\Q_p\cdot 
R$, and so $A/C_A(\UU)\cong\Q_p(\zeta)/\Z_p[\zeta]$.

Consider the short exact sequence
        \[ 0 \Right2{} C_A(\UU) \Right4{\incl} A \Right4{\varphi} 
        [\UU,A] \Right2{} 0, \]
where $\varphi(x)=u(x)-x$ for all $x\in A$. 
We just saw that $[\UU,A]\cong\Q_p(\zeta)/\Z_p[\zeta]$. Under this 
identification, $\varphi|_{[\UU,A]}$ is multiplication by $1-\zeta$, and 
so its kernel has order $p$. Thus $|C_A(\UU)\cap[\UU,A]|=p$.

\smallskip

\noindent\textbf{(c) } We have $|X|=|C_X(\UU)|\cdot|[\UU,X]|$: $X$ 
is finite, and $[\UU,X]$ is the image of the 
homomorphism $X\Right2{1-u}X$ while $C_X(\UU)$ is its kernel. Hence 
(3) and (4) are equivalent.

If (3) holds, then $[\UU,\Omega_1(X)]\cap C_{\Omega_1(X)}(\UU)$ also has 
order $p$ (since it cannot be trivial). Since the rank of this intersection 
is the number of Jordan blocks in $\Omega_1(X)$ with nontrivial 
$\UU$-action, we see that $\Omega_1(X)$ is minimally active in this case. 
So (3) implies (1); and a similar argument shows that (4) implies (2).

Assume that $\Lambda_0<\Lambda$ are $\Z_p\genG$-lattices such that 
$\Lambda_0\le p\Lambda$ and $\Lambda/\Lambda_0\cong X$. In particular, 
$\Lambda/p\Lambda\cong X/pX$. So if $X/pX$ is minimally active, then 
$[\UU,\Lambda]+C_\Lambda(\UU)$ has index $p$ in $\Lambda$ by (a), 
and hence $[\UU,X]+C_X(\UU)$ has index $p$ in $X$ (since it cannot 
be all of $X$). Thus (2) implies (4) in this case. 

We continue to assume that $X\cong\Lambda/\Lambda_0$, and 
set $R_p=\Q_p/\Z_p$ for short. We have an exact sequence 
        \[ 0 \Right2{} \Tor_{\Z_p}(R_p,X) \Right3{} R_p\otimes_{\Z_p}\Lambda_0 
        \Right3{} R_p\otimes_{\Z_p}\Lambda \Right3{} R_p\otimes_{\Z_p}X \Right2{} 0. \]
Also, by tensoring the short exact sequence $0\to\Z_p\to\Q_p\to R_p\to0$ by 
$X$ and using the fact that $\Q_p$ is flat over $Z_p$, we see that 
$R_p\otimes_{\Z_p}X=0$, and $\Tor_{\Z_p}(R_p,X)\cong X$ as $\Z_p\genG$-modules. 
Thus $X$ is isomorphic to a subgroup of the discrete $\genG$-$p$-torus 
$A=R_p\otimes_{\Z_p}\Lambda_0$ (see Proposition \ref{tori&lattices}), where 
$\Omega_1(X)\cong\Omega_1(A)$. With the help of (b), we now see that (1) 
implies (3). 

\smallskip

\noindent\textbf{(d) } Set $\4X=X/pX$ for short. By (c), $\4X$ is minimally 
active as an $\F_p\genG$-module. Hence 
there is $\4y\in\4X$ such that $\4X=C_{\4X}(\UU)+\F_p\UU\cdot\4y$. Thus 
$[\UU,\4X]=(1-u)\F_p\UU\cdot\4y$. Choose $y\in X$ whose class modulo 
$p$ is $\4y$; then $[\UU,X]\le(1-u)\Z_p\UU\cdot y+pX$.

By assumption, $X=C_X(\UU)+[\UU,X]+\Z_p\cdot x$. So there are 
$\xi\in\Z_p\UU$ and $r\in\Z_p$ such that $y\in rx+(1-u)\xi\cdot 
y+C_X(\UU)+pX$. Then $(1-(1-u)\xi)y\in rx+C_X(\UU)+pX$, where 
$1-(1-u)\xi$ is invertible in $\Z_p\UU$ since $(1-u)^p\in 
p\Z_p\UU$. Thus $y\in\Z_p\UU\cdot x+C_X(\UU)+pX$, and hence 
$X=\Z_p\UU\cdot x+C_X(\UU)+pX$. Since $pX$ is the Frattini subgroup of 
$X$, it now follows that $X=\Z_p\UU\cdot x+C_X(\UU)$. 
\end{proof}

In the rest of the section, we look at questions of existence and 
uniqueness of finite $\Z_p\genG$-modules or discrete $\genG$-$p$-tori $A$ for which 
$\Omega_1(A)$ is isomorphic to a given minimally active, indecomposable 
$\F_p\genG$-module.

\begin{Prop} \label{dim(p-1)}
Fix an odd prime $p$ and a group $\genG\in\GG$, and let $V$ be a faithful, 
minimally active, indecomposable $\F_p\genG$-module. 
\begin{enuma} 

\item If $\dim(V)\ge p-1$, then there are a $\Q_p\genG$-module $M$ and a 
$\Z_p\genG$-lattice $\Lambda\le M$ such that $V\cong\Lambda/p\Lambda$.

\item If $\dim(V)=p$, and there is an $\F_p\genG$-submodule $V_1<V$ of 
dimension $1$, then $M$ and $\Lambda$ can be chosen as in (a) such that $M$ 
contains a $1$-dimensional $\Q_p\genG$-submodule.

\end{enuma}
\end{Prop}

\begin{proof} Fix $\UU\in\sylp{\genG}$, and let $u\in\UU$ be a generator. 
Let $\zeta$ be a $p$-th root of unity, and regard $\Q_p(\zeta)$ as a 
$\Q_p\UU$-module where $u$ acts by multiplication by $\zeta$. Thus 
$\Z_p[\zeta]$ is a $\Z_p\UU$-lattice in $\Q_p(\zeta)$. We also write 
$\F_p[\zeta]=\Z_p[\zeta]/p\Z_p[\zeta]$. Thus 
$\F_p[\zeta]\cong\F_p[u]/\gen{(1-u)^{p-1}}$, and 
$V|_\UU\cong\F_p[\zeta]$ since $V$ is minimally active and indecomposable 
\cite[Proposition 3.7(a)]{indp2}. 

\smallskip

\textbf{(a) }  If $\dim(V)\ge p$, then $V$ is a trivial 
source module by \cite[Proposition 3.7(b)]{indp2}, and hence $V$ is the mod 
$p$ reduction of some $\Z_p\genG$-lattice (see \cite[Corollary 
3.11.4.i]{Benson1}). So for the rest of the proof, we assume that 
$\dim(V)=p-1$.

Set $\5\Lambda=\Ind_\UU^\genG(\Z_p[\zeta])$ and $\5V=\Ind_\UU^\genG(V|_\UU)$. 
Then $\5V\cong\5\Lambda/p\5\Lambda$. Since induction of 
representations is adjoint to restriction, the identity on $V$ extends to a 
surjective $\F_p\genG$-linear homomorphism $\alpha\:\5V\too V$. This is split 
by an $\F_p\UU$-linear map, and hence (by averaging over cosets of $\UU$) 
by an $\F_p\genG$-linear homomorphism $\beta\:V\too\5V$. Set 
$e_0=\beta\alpha\in\End_{\F_p\genG}(\5V)$. Thus $e_0$ is an idempotent in this 
endomorphism ring, and $e_0\5V\cong V$.

We want to lift $e_0$ to an idempotent in $\End_{\Z_p\genG}(\5\Lambda)$. By the 
Mackey double coset formula, 
	\[ \5\Lambda|_\UU = \bigl(\Ind_\UU^\genG(\Z_p[\zeta])\bigr)\big|_\UU 
	\cong \bigl(\Z_p[\zeta]\bigr)^m \times \bigl(\Z_p\UU\bigr)^n \]
as $\Z_p\UU$-modules for some $m,n\ge0$. Hence as $\Z_p$-modules,
	\beq \begin{split} 
	\End_{\Z_p\genG}(\5\Lambda) \cong 
	\Hom_{\Z_p\UU}(\Z_p[\zeta],\5\Lambda) &\cong 
	\bigl(\End_{\Z_p\UU}(\Z_p[\zeta])\bigr)^m \times 
	\bigl(\Hom_{\Z_p\UU}(\Z_p[\zeta],\Z_p\UU)\bigr)^n \\
	&\cong (\Z_p[\zeta])^m \times ((1-u)\Z_p\UU)^n 
	\end{split} \eeq
where the last isomorphism follows upon sending a homomorphism $\varphi$ to 
$\varphi(1)$. Since 
	\[ \End_{\F_p\genG}(\5\Lambda) \cong (\F_p[\zeta])^m \times 
	((1-u)\F_p\UU)^n \]
by a similar argument, the natural homomorphism from 
$\End_{\Z_p\genG}(\5\Lambda)$ to $\End_{\F_p\genG}(\5V)$ is surjective (and 
reduction mod $p$). So $e_0$ lifts to an idempotent 
$e\in\End_{\Z_p\UU}(\5\Lambda)$ (see \cite[Proposition 1.9.4]{Benson1}). 

Now set $\Lambda=e\5\Lambda$. Then $\Lambda/p\Lambda\cong V$, and $\Lambda$ 
is a $\Z_p\genG$-lattice in the $\Q_p\genG$-module $M=Q_p\otimes_{\Z_p}\Lambda$.

\smallskip

\noindent\textbf{(b) } We repeat the proof of (a), but keeping control of 
the submodule as well as $V$. Set $V_2=V$ and $V_3=V_2/V_1$, and set 
$\5V_i=\Ind_\UU^\genG(V_i|_\UU)$ for $i=1,2,3$. We thus have short exact sequences 
	\[ 0 \Right2{} V_1 \Right2{f} V_2 \Right2{g} V_3 \Right2{} 0 
	\quad\textup{and}\quad
	0 \Right2{} \5V_1 \Right2{\5f} \5V_2 \Right2{\5g} \5V_3 \Right2{} 0. \]
Let $\alpha_i\:\5V_i\too V_i$ be the natural map, and let $\beta_i\:V_i\too 
\5V_i$ be the $\F_p\genG$-linear splitting of $\alpha_i$ obtained by taking the 
natural $\F_p\UU$-linear inclusion and then averaging over cosets of 
$\UU$ in $\genG$. Thus $\alpha_i\circ\beta_i=\Id_{V_i}$ (upon composing from 
right to left), while $e_i\defeq\beta_i\circ\alpha_i$ is an idempotent in 
$\End_{\F_p\genG}(\5V_i)$. All of these commute with the natural homomorphisms 
$f$, $\5f$, $g$, and $\5g$, and so we get a commutative diagram with exact 
rows:
	\beqq \vcenter{\xymatrix@C=30pt@R=25pt{
	0 \ar[r] & \5V_1 \ar[r]^-{\5f} \ar[d]_{e_1} & \5V_2 \ar[r]^-{\5g} 
	\ar[d]_{e_2} & \5V_3 \ar[r] \ar[d]_{e_3} & 0 \\
	0 \ar[r] & \5V_1 \ar[r]^-{\5f} & \5V_2 \ar[r]^-{\5g} & 
	\5V_3 \ar[r] & 0 \rlap{\,.}
	}} \label{e:V1<V} \eeqq

Now set $\Lambda_1^0=\Z_p$, $\Lambda_2^0=\Z_p\UU$, 
$\Lambda_3^0=\Z_p[\zeta]$, and let $\varphi$ and $\psi$ be such that 
	\[ 0 \Right2{} \Lambda_1^0 \Right2{\varphi} \Lambda_2^0 \Right2{\psi} 
	\Lambda_3^0 \Right2{} 0 \]
is a short exact sequence of $\Z_p\UU$-modules. We identify 
$V_i|_\UU=\Lambda_i^0/p\Lambda_i^0$ in such a way that $f$ and $g$ are the 
reductions modulo $p$ of $\varphi$ and $\psi$, respectively. Set 
$\5\Lambda_i=\Ind_\UU^\genG(\Lambda_i^0)$, so that 
$\5V_i=\5\Lambda_i/p\5\Lambda_i$ as $\F_p\genG$-modules, and let $\5\varphi$ 
and $\5\psi$ be the homomorphisms induced by $\varphi$ and $\psi$. 
We claim that the $e_i$ can be lifted to elements 
$\gee_i\in\End_{\Z_p\genG}(\5\Lambda_i)$ that make the following diagram 
commute:
	\beqq \vcenter{\xymatrix@C=30pt@R=25pt{
	0 \ar[r] & \5\Lambda_1 \ar[r]^-{\5\varphi} \ar[d]_{\gee_1} & 
	\5\Lambda_2 \ar[r]^-{\5\psi} \ar[d]_{\gee_2} & \5\Lambda_3 \ar[r] 
	\ar[d]_{\gee_3} & 0 \\
	0 \ar[r] & \5\Lambda_1 \ar[r]^-{\5\varphi} & \5\Lambda_2 
	\ar[r]^-{\5\psi} & \5\Lambda_3 \ar[r] & 0 \rlap{\,.}
	}} \label{e:L1<L} \eeqq

To see this, we identify 
	\[ \End_{\Z_p\genG}(\5\Lambda_i) \cong 
	\Hom_{\Z_p\UU}\bigl(\Lambda_i^0,(\Ind_\UU^\genG(\Lambda_i^0))|_\UU\bigr), 
	\]
where by the Mackey double coset formula, for some indexing sets $J$ and 
$K$ independent of $i\in\{1,2,3\}$,
	\[ (\Ind_\UU^\genG(\Lambda_i^0))|_\UU \cong 
	\bigoplus_{j\in J} \Lambda_i^0 \oplus \bigoplus_{k\in K} 
	(\Z_p\UU\otimes_{\Z_p}\Lambda_i^0). \]
Fix $\gee_2\in\End_{\Z_p\genG}(\5\Lambda_2)$, and set 
$\gee_2(1)=\bigl((u_j)_{j\in J},(v_k)_{k\in K}\bigr)$ with respect to the 
above decomposition (and where $1$ is the identity in 
$\Lambda_2^0=\Z_p\UU$). Then $\gee_2$ induces endomorphisms $\gee_1$ and 
$\gee_3$ such that \eqref{e:L1<L} commutes if and only if 
$\sum_{i=0}^{p-1}u^i(v_k)\in\Ker(\5\psi)$ for each $k\in K$. Note that 
$\5\psi$, after restriction to the summands for some $k\in K$, is a 
surjection of one free $\Z_p\UU$-module onto another. Hence $\gee_2$ can 
always be chosen (as a lifting of $e_2$) to induce $\gee_1$ and $\gee_3$ 
since the above condition holds modulo $p$ by the commutativity of 
\eqref{e:V1<V}.

Since $\gee_i$ is a lifting of the idempotent $e_i$, we have 
$\gee_i^2\equiv\gee_i$ (mod $p$). Hence 
$(\gee_i)^{2p^k}\equiv(\gee_i)^{p^k}$ (mod $p^{k+1}$) for each $k\ge1$. 
Upon replacing $\gee_i$ by the limit of the $(\gee_i)^{p^k}$, we can 
arrange that each $\gee_i$ is an idempotent in $\End_{\Z_p\genG}(\5\Lambda_i)$ 
(and that the above diagram still commutes). Set 
$\Lambda=\gee_2\5\Lambda_2$ and $M=\Q_p\otimes_{\Z_p}\Lambda$. Thus $M$ is 
a $\Q_p\genG$-module with a 1-dimensional submodule, and has a 
$\Z_p\genG$-lattice $\Lambda$ such that $\Lambda/p\Lambda\cong V$. 
\end{proof}

We now turn to questions of uniqueness, looking first at the finite 
case. When $R$ is a ring and $M$ is an $R$-module, we let 
$\Ann_R(x)$ denote the annihilator of an element $x\in M$.

\begin{Prop} \label{A1=A2}
Fix an odd prime $p$, a finite group $\genG\in\GG$, and $\UU\in\sylp{\genG}$. Let 
$A_1$ and $A_2$ be finite $\Z_p\genG$-modules, and assume that 
$\bigl|C_{A_i}(\UU)\cap[\UU,A_i]\bigr|=p$ for $i=1,2$. 
Set $\sigma=\sum_{u\in\UU}u\in\Z_p\UU$.

Assume that there is a $\Z_p\UU$-linear isomorphism $\varphi\:A_1\too A_2$ 
whose reduction modulo $p$ is $\F_p\genG$-linear. Then $A_1\cong A_2$ as 
$\Z_p\genG$-modules. In particular, this happens if $A_1/pA_1\cong A_2/pA_2$ 
and either 
\begin{enuma} 

\item $A_1$ and $A_2$ are homocyclic of the same exponent and 
$\sigma\cdot A_i\nleq pA_i$ for $i=1,2$; or 

\item there are elements $\xa_1\in A_1$ and $\xa_2\in A_2$ such that 
$A_1=\Z_p\UU\cdot\xa_1$, $A_2=\Z_p\UU\cdot\xa_2$, and 
$\Ann_{\Z_p\UU}(\xa_1)=\Ann_{\Z_p\UU}(\xa_2)$.
\end{enuma}
\end{Prop}

\begin{proof} Set $\4A_i=A_i/pA_i$ for $i=1,2$. For each $X\le A_i$ and 
$g\in A_i$, set $\4X=(X+pA_i)/pA_i$ and $\4g=g+pA_i$. Set 
$Z_i=C_{A_i}(\UU)$ and $S'_i=[\UU,A_i]$. By assumption, $|Z_i\cap 
S'_i|=p$, and hence $|A_i/(Z_i+S'_i)|=p$ by Lemma \ref{min.act.|Z0|=p}(c).

Assume $\varphi\:A_1\too A_2$ is a $\Z_p\UU$-linear isomorphism whose 
reduction $\4\varphi\:\4A_1\too\4A_2$ modulo $p$ is $\F_p\genG$-linear. Let 
$g_1,\dots,g_k$ be a set of representatives for the left cosets $g\UU$ in 
$\genG$ (where $k=|\genG/\UU|$ is prime to $p$), and define 
$\psi\:A_1\Right2{}A_2$ by setting 
$\psi(\lambda)=\frac1k\bigl(\sum_{j=1}^kg_j\varphi(g_j^{-1}\lambda)\bigr)$. 
Then $\psi$ is $\Z_p\genG$-linear, its reduction modulo $p$ is equal to 
$\4\varphi$ since $\4\varphi$ is $\F_p\genG$-linear, it is surjective since the 
reduction mod $p$ is surjective, and is an isomorphism since $|A_1|=|A_2|$. 

It remains to prove that each of (a) and (b) implies the existence of the 
homomorphism $\varphi$. Fix an $\F_p\genG$-linear isomorphism 
$\4\varphi\:\4A_1\too\4A_2$. 

\smallskip

\noindent\textbf{(a) } Assume $A_1$ and $A_2$ are homocyclic of the 
same exponent $p^k$, and for $i=1,2$, $\sigma\cdot A_i\nleq pA_i$. Choose 
$\xa_1\in A_1$ such that $\sigma\cdot\xa_1\notin pA_1$, and let $\xa_2\in A_2$ 
be such that $\4\varphi(\4\xa_1)=\4\xa_2$. Thus for $i=1,2$, 
$\4{\sigma\cdot\xa_i}\ne0$, and so $\{u(\4\xa_i)\,|\,u\in\UU\}$ 
is a basis for $\F_p\UU\cdot\4\xa_i$, and $\Ann_{\F_p\UU}(\4\xa_i)=0$. Note 
that $\xa_i\notin Z_i+S'_i$, since $\sigma\cdot Z_i\le pZ_i$ and 
$\sigma\cdot S'_i=0$.

We claim that each element of $C_{\4A_i}(\UU)$ lifts to an element of 
$C_{A_i}(\UU)$; i.e., that 
	\beqq C_{\4A_i}(\UU)=\4Z_i. \label{e:Zbar} \eeqq 
To see this, let $g\in A_i$ be such that $\4g\in C_{\4A_i}(\UU)$. By 
Lemma \ref{min.act.|Z0|=p}(d), $g=\xi\cdot\xa_i+z$ for some $\xi\in\Z_p\UU$ 
and $z\in Z_i$. Then $\4{\xi\cdot\xa_i}$ is fixed by $\UU$ since $\4g$ is, 
and $C_{\4{\Z_p\UU\cdot\xa_i}}(\UU)=\gen{\4{\sigma\cdot\xa_i}}$ since 
$\Ann_{\F_p\UU}(\4\xa_i)=0$. Hence there is $k\in\Z$ such that 
$\4{k\sigma\cdot\xa_i}=\4{\xi\cdot\xa_i}$; and $\4g=\4{k\sigma\cdot\xa_i+z}$ 
where $k\sigma\cdot\xa_i+z\in Z_i$. This proves \eqref{e:Zbar}.

Set $m=\rk(A_1)-p=\rk(A_2)-p$ (possibly $m=0$). By \eqref{e:Zbar}, we can choose 
elements $x_1,\dots,x_m\in Z_1$ such that 
$\4A_1=\F_p\UU\cdot\4\xa_1\oplus\gen{\4x_1,\dots,\4x_m}$. By \eqref{e:Zbar} 
again, there are elements $y_1,\dots,y_m\in Z_2$ such that 
$\4\varphi(\4x_i)=\4y_i$ for each $i$. Then 
	\[ \{u(\4\xa_1)\,|\,u\in\UU\}\cup\{\4x_1,\dots,\4x_m\} 
	\qquad\textup{and}\qquad
	\{u(\4\xa_2)\,|\,u\in\UU\}\cup\{\4y_1,\dots,\4y_m\} \]
are bases for $\4A_1$ and $\4A_2$, respectively, and $\4\varphi$ sends the 
first basis to the second. Since $A_1$ and $A_2$ are both homocyclic of 
exponent $p^k$, the sets $\{u(\xa_1),x_1,\dots,x_m\,|\,u\in\UU\}$ and 
$\{u(\xa_2),y_1,\dots,y_m\,|\,u\in\UU\}$ 
are bases for $A_1$ and $A_2$, respectively, as $\Z/p^k$-modules. 
Thus $\4\varphi$ lifts to a $\Z_p\UU$-linear 
isomorphism $\varphi\:A_1\too A_2$, defined by setting 
$\varphi(u(\xa_1))=u(\xa_2)$ for $u\in\UU$ and $\varphi(x_i)=y_i$ for each 
$i$.

\smallskip

\noindent\textbf{(b) } Let $\xa_i\in A_i$ (for $i=1,2$) be such that 
$\Z_p\UU\cdot\xa_i=A_i$ and 
$\Ann_{\Z_p\UU}(\xa_1)=\Ann_{\Z_p\UU}(\xa_2)$. Let $\xi\in\Z_p\UU$ 
be such that $\4\varphi(\4\xa_1)=\4{\xi\cdot\xa_2}$. Thus $\4{\xi\cdot\xa_2}$ 
generates $\4A_2$ as an $\F_p\UU$-module, and since $(1-u)A_2\le S'_2$ for 
$1\ne u\in\UU$, $\xi$ is not in the ideal $(1-u)\Z_p\UU+p\Z_p\UU$ of index 
$p$ in $\Z_p\UU$. Since this is the unique maximal ideal in $\Z_p\UU$, $\xi$ 
is invertible, and we can replace $\xa_2$ by $\xi\cdot\xa_2$ without changing 
$\Ann_{\Z_p\UU}(\xa_2)$.

Let $\varphi\:A_1\too A_2$ be the unique $\Z_p\UU$-linear homomorphism such 
that $\varphi(\xa_1)=\xa_2$. Its reduction modulo $p$ is $\4\varphi$, since 
$\4\varphi(\4\xa_1)=\4\xa_2$ and $\4\varphi$ is $\F_p\UU$-linear. 
\end{proof}

It remains to prove the analogous uniqueness result for discrete 
$p$-tori.

\begin{Lem} \label{unique-G-p-torus}
Fix an odd prime $p$, a finite group $\genG\in\GG$, and $\UU\in\sylp{\genG}$. Let 
$A_1$ and $A_2$ be discrete, $\genG$-$p$-tori, and assume 
that 
\begin{enumi} 

\item $\Omega_1(A_1)$ and $\Omega_1(A_2)$ are faithful, minimally active, and 
indecomposable as $\F_p\genG$-modules; and 

\item $\Omega_1(A_1)\cong\Omega_1(A_2)$ as $\F_p\genG$-modules.

\end{enumi}
Then $A_1\cong A_2$ as $\Z_p\genG$-modules.
\end{Lem}

\begin{proof} By Lemma \ref{min.act.|Z0|=p}(b), we also have that 
\begin{enuma}
\item $[\UU,A_i]\cap C_{A_i}(\UU)$ has order $p$ for $i=1,2$.
\end{enuma}

Assume, for each $k\ge1$, that 
$\Omega_k(A_1)\cong\Omega_k(A_2)$ as $\Z_p\genG$-modules, and let $X_k$ be the 
set of $\Z_p\genG$-linear isomorphisms 
$\Omega_k(A_1)\Right2{\cong}\Omega_k(A_2)$. Then $X_k$ is finite since the 
$\Omega_k(A_i)$ are finite, $X_k\ne\emptyset$ by assumption, and if 
$k\ge2$, restriction to $\Omega_{k-1}(A_i)$ defines a map $X_k\too 
X_{k-1}$. So the inverse limit of the $X_k$ is nonempty, and each element 
in the inverse limit determines a $\Z_p\genG$-linear isomorphism $A_1\cong 
A_2$. 

It remains to show that $\Omega_k(A_1)\cong\Omega_k(A_2)$ for each $k$. 
Since $\Omega_k(A_i)/p\Omega_k(A_i)\cong\Omega_1(A_i)$ as 
$\F_p\genG$-modules (multiplication by $p^{k-1}$ defines an isomorphism), we 
have $\Omega_k(A_1)/p\Omega_k(A_1)\cong\Omega_k(A_2)/p\Omega_k(A_2)$, 
and both are faithful, minimally active, and indecomposable. 

Set $\sigma=\sum_{u\in\UU}u\in\Z_p\UU$, as usual. If $\rk(A_i)\ge p$, 
then by Lemma \ref{min.act.props}(a,b), $\sigma\cdot\Omega_k(A_i)\nleq 
p\Omega_k(A_i)$. Since $\Omega_k(A_1)$ and $\Omega_k(A_2)$ are both 
homocyclic of exponent $p^k$, they are isomorphic as $\Z_p\genG$-modules by 
Proposition \ref{A1=A2}(a).

If $\rk(A_i)=p-1$ for $i=1,2$, then by Proposition \ref{tori&lattices}, 
$A_i\cong(\Q_p/\Z_p)\otimes_{\Z_p}\Lambda_i$ for some $(p-1)$-dimensional 
$\Z_p\genG$-lattice $\Lambda_i$. Since $\genG$ acts faithfully on the lattices, 
$\Lambda_1\cong\Lambda_2\cong\Z_p[\zeta]$ as $\Z_p\UU$-modules by Lemma 
\ref{ZU-modules}(a,c) (where $\zeta$ is a primitive $p$-th root of unity). 
Hence for $i=1,2$, $A_i\cong\Q_p(\zeta)/\Z_p[\zeta]$, and 
$\Omega_k(A_i)\cong\Z_p[\zeta]/p^k\Z_p[\zeta]$, as $\Z_p\UU$-modules. So 
there is $\xa_i\in\Omega_k(A_i)$ such that 
$\Z_p\UU\cdot\xa_i=\Omega_k(A_i)$ and $\Ann_{\Z_p\UU}(\xa_i)$ is the 
ideal generated by $p^k$ and $\sigma$. Proposition \ref{A1=A2}(b) now 
applies to conclude that $\Omega_k(A_1)\cong\Omega_k(A_2)$ as 
$\Z_p\genG$-modules.
\end{proof}

\section{Reduced fusion systems over finite nonabelian $p$-groups with 
index $p$ abelian subgroup ($p$ odd)}
\label{s:finite}

Throughout this section, $p$ is an odd prime, and $A$ is finite. As noted 
in the introduction, the corresponding question for finite $2$-groups 
was answered in \cite[Proposition 5.2(a)]{AOV2}.

\begin{Lem}[{\cite[Lemma 2.2(d,e,f)]{indp2}}] \label{l1:s/a}
Assume the notation and hypotheses of \ref{n:not1}, and also that $p$ is 
odd and $|A|<\infty$. Set $A_0=ZS'$. Then the following hold. 
\begin{enuma} 

\item If $A\nnsg\calf$, then there are elements $\xx\in{}S{\sminus}A$ and 
$\xa\in{}A{\sminus}A_0$ such that $A_0\gen{\xx}$ and $S'\gen{\xa}$ are 
normalized by $\autf(S)$.  If some element of $S{\sminus}A$ has order $p$, 
then we can choose $\xx$ to have order $p$.

\item For each $P\in\EE\calf$ and each $\alpha\in N_{\autf(P)}(\Aut_S(P))$, 
$\alpha$ extends to some $\4\alpha\in\autf(S)$.

\item For each $x\in{}S{\sminus}A$ and each $g\in{}A_0$, $Z\gen{x}$ is 
$S$-conjugate to $Z\gen{gx}$, and $Z_2\gen{x}$ is $S$-conjugate to 
$Z_2\gen{gx}$.
\end{enuma}
\end{Lem}

We now fix some more notation, based on Lemma \ref{l1:s/a}. 

\begin{Not} \label{n:not2a}
Assume Notation \ref{n:not1}. Assume also that $p$ is odd, $S$ is finite, 
and $A\nnsg\calf$, and hence that $|Z_0|=|A/ZS'|=p$ by Lemma \ref{|Z0|=p}.  
Fix $\xa\in{}A\sminus ZS'$ and $\xx\in{}S\sminus A$, chosen such that 
$ZS'\gen{\xx}$ and $S'\gen{\xa}$ are each normalized by $\autf(S)$, and 
such that $\xx^p=1$ if any element of $S\sminus A$ has order $p$ (Lemma 
\ref{l1:s/a}(a)).  For each $i=0,1,\dots,p-1$, define 
	\[ H_i = Z\gen{\xx\xa^i}\in\calh
	\qquad\textup{and}\qquad 
	B_i = Z_2\gen{\xx\xa^i}\in\calb \,. \]
Let $\calh_i$ and $\calb_i$ denote the $S$-conjugacy classes of $H_i$ 
and $B_i$, respectively, and set
	\[ \calh_*=\calh_1\cup\cdots\cup\calh_{p-1} \quad\textup{and}\quad
	\calb_*=\calb_1\cup\cdots\cup\calb_{p-1}. \]
For each $P\le S$, set 
	\[ \autf^{(P)}(S) = \bigl\{ \alpha\in\autf(S) \,\big|\, 
	\alpha(P)=P,~ \alpha|_P\in O^{p'}(\autf(P)) \bigr\}. \]
\end{Not}

When $|Z_0|=p$, then by Lemma \ref{l1:s/a}(c), $\calh=\calh_0\cup\calh_*$ 
and $\calb=\calb_0\cup\calb_*$. Note that for $x,x'\in S\sminus A$, 
$Z\gen{x}$ is $S$-conjugate to $Z\gen{x'}$ or $Z_2\gen{x}$ is $S$-conjugate to 
$Z_2\gen{x'}$ only if $x'x^{-1}\in ZS'$. So in fact, each of the sets $\calh$ and 
$\calb$ is a union of $p$ distinct $S$-conjugacy classes: the classes 
$\calh_i$ and $\calb_i$ for $0\le i\le p-1$.

\begin{Lem}[{\cite[Lemma 2.5(a,b)]{indp2}}] \label{l3:s/a}
Let $p$ be an odd prime, let $S$ be a finite nonabelian $p$-group 
with a unique abelian subgroup $A\nsg{}S$ of index $p$, and let $\calf$ be 
a saturated fusion system over $S$ such that $A\nnsg\calf$.  We use the 
conventions of Notation \ref{n:not1} and \ref{n:not2}, set $A_0=ZS'$, and let 
$m\ge3$ be such that $|A/Z|=p^{m-1}$.  Then the following hold.
\begin{enuma}
\item $\5\mu|_{\outff(S)}$ is injective.

\item Fix $\alpha\in\Aut(S)$, set $(r,s)=\mu(\alpha)$, and let $t$ be such 
that $\alpha(g)\in{}g^tA_0$ for each $g\in{}A{\sminus}A_0$.  Then $s\equiv 
tr^{m-1}$ (mod $p$).

\end{enuma}
\end{Lem}

\begin{Lem}[{\cite[Lemma 2.6(a)]{indp2}}] \label{l4:s/a}
Let $p$ be an odd prime, let $S$ be a finite nonabelian $p$-group 
with a unique abelian subgroup $A\nsg{}S$ of index $p$, and let $\calf$ be 
a saturated fusion system over $S$. We use the notation of Notation 
\ref{n:not1} and \ref{n:not2}. Let $m$ be such that $|A/Z|=p^{m-1}$. Fix 
$P\in\calh\cup\calb$, and set
	\[ H_P=N_{\autf(S)}(P)\,,
	\qquad
	\5H_P=\bigl\{\alpha|_P\,\big|\,\alpha\in{}H_P\bigr\}\,,
	\qquad\textup{and}\qquad
	t = \begin{cases} -1 & \textup{if $P\in\calh$} \\
	0 & \textup{if $P\in\calb$.}
	\end{cases}
	\]
If $P\in\EE\calf$, then $\autf^{(P)}(S)\le\autff(S)$ and
$\mu(\autf^{(P)}(S))=\Delta_t$.  If $P\in\calh_*$ or $P\in\calb_*$,
then $m\equiv t$ (mod $p-1$).  
\end{Lem}

\begin{Thm}[{\cite[Theorem 2.8]{indp2}}] \label{t3:s/a}
Fix an odd prime $p$, and a finite nonabelian $p$-group $S$ which 
contains a unique abelian subgroup $A\nsg{}S$ of index $p$. Let $\calf$ be 
a reduced fusion system over $S$ for which $A$ is $\calf$-essential.  We 
use the notation of \ref{n:not1}, \ref{n:not2}, and \ref{n:not2a}, and also 
set $A_0=ZS'$, $\EE0=\EE\calf{\sminus}\{A\}$, and $G=\autf(A)$. Thus 
$\UUU=\Aut_S(A)\in\sylp{G}$. Let $m\ge3$ be such that $|A/Z|=p^{m-1}$. Then 
the following hold:
\begin{enuma}  

\item $Z_0=C_A(\UUU)\cap[\UUU,A]$ has order $p$, and hence 
$A_0=C_A(\UUU)[\UUU,A]$ has index $p$ in $A$.  

\item There are no nontrivial $G$-invariant subgroups of $Z=C_A(\UUU)$, 
aside (possibly) from $Z_0$.  

\item $[G,A]=A$.

\item One of the conditions (i)--(iv) holds, described in 
Table \ref{tbl:(d)}, where $\sigma=\sum_{u\in\UUU}u\in\Z_p\UUU$. 
\begin{table}[ht]
\[ \renewcommand{\arraystretch}{1.5}
\newcommand{\halfup}[1]{\raisebox{2.2ex}[0pt]{$#1$}}
\begin{array}{|c|c|c|c|c|c|} \hline
 & \mu_A(\autff(A)) & G=O^{p'}(G)X~ \textup{ where} & 
\textup{$m$ (mod ~ ${p-1}$)} & \sigma\cdot A & \EE0  \\ \hline\hline

\textup{(i)} & \Delta & X=\autff(A) & \equiv0 & \le\Fr(Z) &
\calh_0\cup\calb_*  \\\hline

\textup{(ii)} & \Delta & X=\autff(A) & \equiv-1 & \le\Fr(Z) &
\calb_0\cup\calh_*  \\\hline

\textup{(iii$'$)} &  & & \equiv-1 & \le\Fr(Z) & \bigcup_{i\in I}\calh_i 
\\\cline{1-1}\cline{4-6}

\textup{(iii$''$)} & \halfup{\ge\Delta_{-1}} & 
\halfup{X=\mu_A^{-1}(\Delta_{-1})} & - & - & \calh_0 \\\hline

\textup{(iv$'$)} &  &  {X=\mu_A^{-1}(\Delta_{0})}
& \equiv0 & \le\Fr(Z) & \bigcup_{i\in I}\calb_i \\\cline{1-1}\cline{4-6}

\textup{(iv$''$)} & \halfup{\ge\Delta_{0}} & 
\textup{$Z_0$ not $G$-invariant} & - & - & \calb_0 \\\hline
 
\hline
\end{array}
\]
\caption{} \label{tbl:(d)}
\end{table}

\end{enuma}

Conversely, for each $G$, $A$, $\UUU\in\sylp{G}$, and 
$\EE0\subseteq\calh\cup\calb$ which satisfy conditions (a)--(d), where 
$|\UUU|=p$ and $\UUU\nnsg{}G$, there is a simple fusion system $\calf$ over 
$A\sd{}\UUU$ with $\autf(A)=G$ and $\EE\calf=\EE0\cup\{A\}$, unique up to 
isomorphism. When $A$ is not elementary abelian, all such fusion systems 
are exotic, except for the fusion systems of the simple groups listed in 
Table \ref{tbl:type3}. \\ 
Such a fusion system $\calf$ has a proper strongly closed subgroup if and 
only if $A_0=C_A(\UUU)[\UUU,A]$ is $G$-invariant, and $\EE0=\calh_i$ or 
$\calb_i$ for some $i$, in which case $A_0H_i=A_0B_i$ is strongly closed. 
\end{Thm}

\begin{table}[ht]
\begin{small} 
\[ \renewcommand{\arraystretch}{1.5}
\begin{array}{|c|c|c|c|c|c|c|c|} \hline
\Gamma & p & \textup{conditions} & \rk(A) & e & m &
G=\Aut_\Gamma(A) & \EE0 \\ \hline\hline

\PSL_p(q) & p & p^2|(q{-}1),\ p>3 & p{-}1 & v_p(q{-}1) & e(p{-}1)-1 & 
\Sigma_p & \calh_0\cup\calh_* \\ \hline

\PSL_n(q) & p & p^2|(q{-}1),\ p{<}n{<}2p & n{-}1 & v_p(q{-}1) & e(p{-}1)+1 & 
\Sigma_n & \calb_0 \\ \hline

P\Omega_{2n}^+(q) & p & p^2|(q{-}1),\ p{\le} n{<}2p & n & v_p(q{-}1) & 
e(p{-}1)+1 & C_2^{n-1}\rtimes\Sigma_n & \calb_0 \\ \hline

\lie2F4(q) & 3 & \raisebox{3pt}{\hbox to 2cm{\hrulefill}} & 2 & v_3(q{+}1) 
& 2e & GL_2(3) & \calb_0\cup\calb_* \\ \hline 

E_n(q) & 5 & n=6,7,\ p^2|(q{-}1) & n & v_5(q{-}1) & 4e+1 & W(E_n) & \calb_0 
\\ \hline

E_n(q) & 7 & n=7,8,\ p^2|(q{-}1) & n & v_7(q{-}1) & 6e+1 & W(E_n) & \calb_0 
\\ \hline

E_8(q) & 5 & v_5(q^2+1)\ge2 & 4 & v_5(q^4{-}1) & 4e & 
(4\circ2^{1+4}).\Sigma_6 & \calb_0\cup\calb_* \\ \hline

\hline
\end{array}
\]
\end{small}
\caption{In this table, $e$ is such that $p^e$ is the exponent of $A$, 
and we restrict to the cases where $e\ge2$.  In all 
cases except when $\Gamma\cong\PSL_p(q)$, $A$ is homocyclic.} 
\label{tbl:type3}
\end{table}

\smallskip

We now look for a more precise description of the group $A$ when it is 
finite but not elementary abelian. 
The following notation will be useful when describing more precisely 
elements and subgroups of $A$. 

\begin{Not} \label{n:not3} 
Assume Notation \ref{n:not1} and \ref{n:not2a}, and set 
$\uu=c_\xx\in\UUU=\Aut_S(A)$.  Set 
$\sigma=1+\uu+\uu^2+\ldots+\uu^{p-1}\in\Z_p\UUU$. Regard $A$ as a 
$\Z_p\UUU$-module, and define 
	\[ \Psi\: \Z_p\UUU \Right4{} A \]
by setting $\Psi(\xi)=\xi\cdot\xa$. Thus $\Psi\bigl(\sum_{i=0}^{p-1}n_i\uu^i\bigr)=
\prod_{i=0}^{p-1}\uu^i(\xa)^{n_i}$ for $n_i\in\Z_p$. 

Set $\zeta=e^{2\pi{}i/p}$, $R=\Z_p[\zeta]$, and $\pp=(1{-}\zeta)R$.  Thus 
$\pp$ is the unique maximal ideal in $R$. We identify 
$R=\Z_p\UUU/\sigma\Z_p\UUU$, by sending $\zeta\in{}R$ to the class of $\uu$ 
modulo $\gen{\sigma}$.  
\end{Not}

The basic properties of $\Psi$, and the role of $\Psi(\sigma)$, are 
described in the following lemma. Recall that 
$\mho^k(P)=\gen{g^{p^k}\,|\,g\in P}$, when $P$ is a $p$-group and $k\ge1$.

\begin{Lem} \label{A/Fr(A)}
Assume Notation \ref{n:not1} and \ref{n:not3}, where $A\in\EE\calf$, and 
$A\nnsg\calf$ is finite and not elementary abelian. Let $m$ be such that 
$|A/Z|=p^{m-1}$. Then 
\begin{enuma} 

\item $\Im(\Psi)\cap Z=Z_0\gen{\Psi(\sigma)}$;

\item $\Psi$ induces an isomorphism $A/Z\cong R/\pp^{m-1}$ 
via the identification  $R=\Z_p\UUU/\gen{\sigma}$; and 

\item $\Psi((1-\uu)^m)=1$ and $Z_0=\gen{\Psi((\uu-1)^{m-1})}$.

\end{enuma}
Furthermore, the following all hold.

\begin{enuma}\setcounter{enumi}{3}

\item The homomorphism $\Psi$ is surjective if and only if $\rk(A)\le p$, 
if and only if $Z$ is cyclic. If $\Psi(\sigma)\in\Fr(A)$, then $\rk(A)<p$ 
and $\Psi$ is surjective.

\item Either \smallskip
\begin{itemize} 
\item $\Psi(\sigma)=1$, in which case $\rk(A)=p-1$, $Z=Z_0$, and $\Psi$ 
induces an isomorphism $A\cong R/\pp^m$ via the identification  
$R=\Z_p\UUU/\gen{\sigma}$; or 
\item $\Psi(\sigma)\notin\Fr(Z)$, in which case 
$\EE\calf\subseteq\{A\}\cup\calh_0$ or $\EE\calf\subseteq\{A\}\cup\calb_0$.
\end{itemize} 

\item If $\Psi(\sigma)\ne1$ and $\Psi(\sigma)\in 
Z_0$, then $m\equiv1$ (mod $p-1$). If $\Psi(\sigma)\notin Z_0$, then 
$\mu(\autff(S))=\Delta_{m-1}$; and either $m\equiv1$ (mod $p-1$) and 
$\EE\calf\sminus\{A\}=\calb_0$, or $m\equiv0$ (mod $p-1$) and 
$\EE\calf\sminus\{A\}=\calh_0$.

\item If $\Psi$ is not surjective, then $A$ is homocyclic. 

\end{enuma}
\end{Lem}

\begin{proof} Set $\4A=A/\Fr(A)$ for short. For $B\le A$ or $g\in A$, 
let $\4B\le\4A$ or $\4g\in\4A$ denote their images in $\4A$ under 
projection. Let $\4\Psi\:\Z_p\UUU\Right2{}\4A$ be the composite of $\Psi$ 
followed by projection to $\4A$. 

\smallskip

\noindent\textbf{(a) } Since $(1-\uu)\Z_p\UUU+\sigma\Z$ has index $p$ 
in $\Z_p\UUU$,
	\[ \Im(\Psi) 
	= \Psi\bigl((1-\uu)\Z_p\UUU\bigr)\gen{\Psi(\sigma)}\gen{\Psi(1)}
	= S'\gen{\Psi(\sigma)}\gen{\xa} , \]
where $\xa^p\in S'\gen{\Psi(\sigma)}$. Since $\xa\notin Z$ and 
$\Psi(\sigma)\in Z$, we have 
	\[ \Im(\Psi)\cap Z = C_{\Im(\Psi)}(\UUU) = 
	C_{S'\gen{\Psi(\sigma)}}(\UUU) = (S'\cap Z)\cdot\gen{\Psi(\sigma)} = 
	Z_0\gen{\Psi(\sigma)}. \]

\smallskip

\noindent\textbf{(b,c) } Since $\Psi(\sigma)\in Z$, $\Psi$ induces a 
homomorphism from $\Z_p\UUU/\gen{\sigma}\cong R$ to $A/Z$, which is onto 
since $A=Z\cdot\Im(\Psi)$ by Lemma \ref{min.act.|Z0|=p}(d). Since 
$|A/Z|=p^{m-1}$ by assumption (and since $\pp$ is the unique maximal ideal 
in $R$ that contains $p$), we have $A/Z\cong R/\pp^{m-1}$. Hence 
$\Psi((\uu-1)^{m-1})\in Z$ and $\Psi((\uu-1)^{m-2})\notin Z$, and the 
latter implies that $\Psi((\uu-1)^{m-1})\ne1$. Thus in all cases (and since 
$|Z_0|=p$), $\Psi((\uu-1)^m)=1$ and $\gen{\Psi((\uu-1)^{m-1})}=Z_0$. 

\smallskip

\noindent\textbf{(d) } If $\rk(A)\le p$, then $\rk(\4A)\le p$, and by 
\cite[Proposition 3.7(a)]{indp2}, $\4A|_\UUU$ is indecomposable. Hence 
$\4\Psi$ is onto in this case, and so $\Psi$ is also onto. 
Conversely, if $\rk(A)>p=\rk(\Z_p\UUU)$, then $\Psi$ is clearly not 
surjective.

By Lemmas \ref{min.act.}(b) and \ref{min.act.props}(c), 
$\Omega_1(Z)=C_{\Omega_1(A)}(\UUU)$ has rank $1$ if and only if 
$\rk(\Omega_1(A))\le p$. Hence $Z$ is cyclic if and only if 
$\rk(A)\le p$.

If $\Psi(\sigma)\in\Fr(A)$, then $\rk(\Im(\4\Psi))\le p-1$. Hence $\4A$ has 
no nontrivial Jordan block of rank $p$, and by \cite[Proposition 
3.7(a)]{indp2} again, $\4A$ is indecomposable as an $\F_p\UUU$-module. So 
$\rk(A)=\rk(\4A)<p$, and $\Psi$ is onto. 

\smallskip

\noindent\textbf{(e) } If $\Psi(\sigma)\in\Fr(Z)\le\Fr(A)$, then $\Psi$ is 
surjective by (a), so $Z=Z_0\gen{\Psi(\sigma)}\le Z_0\cdot\Fr(Z)$, and hence 
$Z=Z_0$ and $\Psi(\sigma)\in\Fr(Z_0)=1$. Thus $\Psi$ factors through a 
surjection $\Psi^*\:\Z_p\UUU/\gen{\sigma}\cong R \Right2{}A$, and induces an 
isomorphism $A\cong R/I$ for some ideal $I$ in $R$. Since $\pp$ is the only 
prime ideal in $R$ of $p$-power index (and $|R/\pp|=p$), and since 
$|A|=p^{m-1}|Z|=p^m$ (recall $Z=Z_0$ by (a)), we have $I=\pp^m$. 

Since $\EE\calf\not\subseteq\{A\}$ (Notation \ref{n:not3}), $\xx^p=1$ by 
Notation \ref{n:not3} and Lemma \ref{l1:s/a}. For each $b\in A$, 
	\[ (b\xx)^p=(b\xx)^p\xx^{-p}=b\cdot\9{\xx}b\cdot\9{\xx^2}b\cdots
	\9{\xx^{p-1}}b 
	=\textstyle\prod_{i=0}^{p-1}\uu^i(b). \]
If $\Psi(\sigma)=\prod_{i=0}^{p-1}\uu^i(\xa)\notin\Fr(Z)$, then 
$\prod_{i=0}^{p-1}\uu^i(b)\ne1$ for each $b\in A\sminus 
ZS'=\bigcup_{i=1}^{p-1}\xa^iZS'$. So by Lemma \ref{l2:s/a}, no member of 
$\calh_*\cup\calb_*$ can be essential, and 
$\EE\calf\subseteq\{A\}\cup\calh_0\cup\calb_0$. 

\smallskip

\noindent\textbf{(f) } Assume $A\nnsg\calf$, and thus 
$\EE\calf\not\subseteq\{A\}$. Fix $P\in\EE\calf\cap(\calh\cup\calb)$ and 
$\alpha\in\autf^{(P)}(S)\le\autff(S)$ (Lemma \ref{l4:s/a}). 

Set $\mu(\alpha)=(r,s)$, and let $t$ be as in 
Lemma \ref{l3:s/a}(b).  Thus $s\equiv tr^{m-1}$ (mod $p$) and 
$\alpha(\xa)\equiv\xa^t$ (mod $ZS'$), so $\alpha(\xa)=\Psi(\xi)$ for some 
$\xi\equiv t$ (mod $\gen{1-\uu,p}$).  Also, $\alpha(\xx)\in\xx^rA$, so 
$\alpha(\uu^i(g))=\uu^{ri}(\alpha(g))$ for all $i$ and $g\in{}A$.  Thus 
	\begin{multline} 
	\alpha(\Psi(\sigma)) = \prod_{i=0}^{p-1}\alpha(\uu^i(\xa)) 
	= \prod_{i=0}^{p-1}\uu^{ri}(\alpha(\xa))
	= \Psi\Bigl(\sum_{i=0}^{p-1}\xi\uu^{ri}\Bigr) \\
	= \Psi(\xi\sigma) 
	\equiv \Psi(t\sigma) \,. \qquad
	\textup{(mod $\Psi(p\sigma)$)}
	\label{e:l5x} \end{multline}
In other words, $\alpha(\Psi(\sigma))\equiv \Psi(\sigma)^t$ (mod 
$\gen{\Psi(\sigma)^p}$). 

If $\Psi(\sigma)\ne1$ and $\Psi(\sigma)\in Z_0$, then $t\equiv s$ (mod 
$p$) by \eqref{e:l5x} (and by definition of $\mu$), and 
hence $r^{m-1}\equiv1$ (mod $p$). Since this holds for arbitary $\alpha$ 
and hence for arbitary $r$ prime to $p$ by \cite[Lemma 2.6(a)]{indp2} and since 
$P\in\EE\calf\cap(\calh\cup\calb)$, it follows that $m\equiv1$ (mod $p-1$). 

Now assume $\Psi(\sigma)\notin Z_0$. By \eqref{e:l5x} and since 
$[\alpha,Z]\le Z_0$, we have $t\equiv1$ and $s\equiv r^{m-1}$. Since this 
holds for arbitrary $\alpha\in\autf^{(P)}(S)$ (in particular, for arbitrary 
$r$ prime to $p$), it follows that 
$\mu(\autf^{(P)}(S))\le\mu(\autff(S))\le\Delta_{m-1}$, with equality by 
Lemma \ref{l4:s/a}. So by Lemma \ref{l4:s/a}, $P\notin\calh_*\cup\calb_*$, 
and either $P\in\calh_0$ and $\Delta_{m-1}=\Delta_{-1}$ (so $m\equiv0$ (mod 
$p-1$)); or $P\in\calb_0$ and $\Delta_{m-1}=\Delta_0$ (so $m\equiv1$ (mod 
$p-1$)). 

\smallskip

\noindent\textbf{(g) } Assume that $\Psi$ is not onto, and hence by (d) that 
$\rk(A)\ge p+1$ and $\Psi(\sigma)\notin\Fr(A)$. Let $k\ge2$ be such that 
$A$ has exponent $p^k$. If $A/Z$ has strictly smaller exponent, then 
$1\ne\mho^{k-1}(A)\le Z$, and thus $\mho^{k-1}(A)$ is an $\F_pG$-submodule 
of the minimally active, indecomposable module $\Omega_1(A)$ upon which 
$\UUU$ acts trivially. If $\rk(A)=p+1$, this contradicts Lemma 
\ref{rk(V)=p+1}, while if $\rk(A)\ge p+2$, this is impossible since 
$\Omega_1(A)$ is simple by \cite[Proposition 3.7(c)]{indp2}. Thus $A/Z\cong 
R/\pp^{m-1}$ also has exponent $p^k\ge p^2$, and hence $m-1\ge p$. So by 
(c), and since $(\uu-1)^p\in p\Z_p\UUU$, we have 
$Z_0=\gen{\Psi((\uu-1)^{m-1})}\le\Fr(A)$. 
 
Now, $\rk(A/Z)=\rk(R/\pp^{m-1})=p-1$ since $m\ge p$, and 
$\rk(Z)=\rk(C_{\Omega_1(A)}(\UUU))=\rk(A)-(p-1)$ by Lemma 
\ref{min.act.props}(c). If $Z_0$ is a direct factor in $Z$, then 
$\rk(Z/Z_0)=\rk(Z)-1$, so $\rk(A/Z_0)\le\rk(A/Z)+\rk(Z/Z_0)=\rk(A)-1$. Thus 
no minimal generating set for $A/Z_0$ lifts to a generating set for $A$, so 
$Z_0\nleq\Fr(A)$, which contradicts what we just showed. Thus $Z_0$ is not 
a direct factor in $Z$, and so $\EE\calf\cap\calh=\emptyset$ by Lemma 
\ref{l2:s/a}(a). 

In particular,  $m\equiv1$ (mod $p-1$) by (f), and hence $A/Z\cong 
R/\pp^{m-1}$ is homocyclic of rank $p-1$ and exponent $p^k$. Thus $A$ and 
$A/Z$ are both $\Z/p^k$-modules and $A/Z$ is free, so $A\cong Z\times(A/Z)$ 
as abelian groups. Since $\mho^{k-1}(A)\cap Z=C_{\mho^{k-1}(A)}(\UUU)\ne1$, 
this shows that $\rk(\mho^{k-1}(A))\ge p$. 

If $A$ is not homocyclic, then $\mho^{k-1}(A)<\Omega_1(A)$ is a 
nontrivial proper $\F_pG$-submodule, where $\Omega_1(A)$ is faithful, 
minimally active, and indecomposable by Lemma \ref{min.act.}(b). Hence 
$\rk(A)=\dim(\Omega_1(A))=p+1$, since $\Omega_1(A)$ is simple if 
$\dim(\Omega_1(A))\ge p+2$ by \cite[Proposition 3.7(c)]{indp2}. So 
$\dim(\mho^{k-1}(A))\le p-1$ by Lemma \ref{rk(V)=p+1}. This contradicts 
what was shown in the last paragraph, and we conclude that $A$ is homocyclic. 
\end{proof}

\begin{Lem} \label{l:(u-1)^m}
Let $p$ be an odd prime, let $\UUU$ be a group of order $p$, and set 
$\sigma=\sum_{u\in\UUU}u\in\Z\UUU$. Then 
for each $1\ne u\in\UUU$ and each $k\ge1$, 
	\[ (u-1)^{k(p-1)}\equiv(-1)^{k-1}(p^{k-1}\sigma-p^k) 
	\pmod{p^k(u-1)\Z\UUU}. \]
\end{Lem}

\begin{proof} Since $\binom{p-1}{k}\equiv(-1)^k$ (mod $p$) for each $0\le 
k\le p-1$, we have $(u-1)^{p-1}\equiv\sigma$ (mod $p\Z\UUU$). Hence 
	\beqq (u-1)^{p-1}\equiv \sigma-p \pmod{p(u-1)\Z\UUU} 
	\label{e:(u-1)^m} \eeqq
since they are congruent modulo $p$ and modulo $u-1$.  This proves the 
lemma when $k=1$. 

When $k>1$, \eqref{e:(u-1)^m} together with the congruence for 
$(u-1)^{(k-1)(p-1)}$ give 
	\begin{align*} 
	(u-1)^{k(p-1)} & =(u-1)^{p-1}\cdot(u-1)^{(k-1)(p-1)} \\
	&\equiv (u-1)^{p-1}\cdot(-1)^{k-2}(p^{k-2}\sigma-p^{k-1})  
	&&\pmod{(u-1)^{p-1}\cdot p^{k-1}(u-1)} \\
	&\equiv (\sigma-p)\cdot(-1)^{k-2}(p^{k-2}\sigma-p^{k-1}) 
	&&\pmod{p(u-1)(p^{k-2}\sigma-p^{k-1})} \\
	&= (-1)^{k-1}(p^{k-1}\sigma-p^k) ;
	\end{align*}
and the congruences hold modulo $p^k(u-1)$ since $p(u-1)$ divides 
$(u-1)^p$ and $(u-1)\sigma=0$.
\end{proof}

\begin{Prop} \label{l:fin(abc)}
Assume the notation of \ref{n:not1}, \ref{n:not2}, \ref{n:not2a}, and 
\ref{n:not3}.  Assume also that $A$ is finite and not elementary abelian, 
that $A\in\EE\calf$, and that $O_p(\calf)=1$.  Let $m\ge3$ be such that 
$|A/Z|=p^{m-1}$, and let $k\ge2$ be such that $A$ has exponent $p^k$. 
Then one of the following holds, as summarized in Table \ref{tbl:fin(abc)}, 
where $G=\autf(A)$.
\begin{enumerate}[\rm(a) ]

\item If $\Psi(\sigma)=1$, then $\Psi$ is onto, 
$\Ker(\Psi)=\gen{\sigma,(\uu-1)^m}$, $\rk(A)=p-1$, 
$Z=Z_0=\gen{\Psi((\uu-1)^{m-1})}$, 
and $A\cong R/\pp^m$ as $\Z_p\UUU$-modules.

\item If $\Psi(\sigma)\notin\Fr(Z)$ and $A$ is homocyclic, then 
$\rk(A)\ge\rk(\Im(\Psi))=p$, $\rk(Z)=\rk(A)-p+1$, and $\Im(\Psi)$ and $Z$ 
are both direct factors in $A$ and homocyclic of exponent $p^k$. Also, 
$\EE\calf=\{A\}\cup\calb_0$. Either $\Psi$ is onto and $\rk(A)=p$, or 
$\Psi$ is not onto and $\rk(A)>p$. If $\rk(A)\ge p+2$, then 
$A/\Fr(A)\cong\Omega_1(A)$ are irreducible $\F_p[\autf(A)]$-modules. 

\item If $\Psi(\sigma)\notin\Fr(Z)$ and $A$ is not homocyclic, then $\Psi$ 
is onto, $\rk(A)=p-1$, $m\equiv1$ (mod $p-1$), 
$\Ker(\Psi)=\gen{p^{k},p^{k-1}-\ell\sigma}$ for some $\ell$ prime to $p$, and 
$\EE\calf=\{A\}\cup\calh_0$. Also, $A\cong(C_{p^{k-1}})^{p-2}\times 
C_{p^{k}}$, where $\mho^{k-1}(A)=Z=Z_0=\gen{\Psi(\sigma)}$. If $k=2$, then 
$\ell\not\equiv1$ (mod $p$). 

\end{enumerate}
\end{Prop}

\begin{table}[ht]
\begin{center} \renewcommand{\arraystretch}{1.5}
\begin{tabular}{|c||c|c|c|}
\hline
Case & (a) & (b) & (c) \\\hline\hline
$\Psi(\sigma)$ & $\Psi(\sigma)=1$ & 
\multicolumn{2}{c|}{$\Psi(\sigma)\notin\Fr(Z)$}  \\\hline
\small{$A$ homocyclic?} & \dbl{\textup{yes if $(p-1)\mid m$}}{\textup{no if 
$(p-1)\nmid m$}} & yes & no \\\hline
$\Psi$ onto? & yes & \dbl{\textup{yes if $\rk(A)=p$}}{\textup{no if 
$\rk(A)>p$}} & yes \\\hline
$\rk(A)$ & $p-1$ & $r\ge p$ & $p-1$ \\\hline
$\Ker(\Psi)$ & $\gen{(\uu-1)^m,\sigma}$ & $p^k\Z_p\UUU$ & 
$\gen{p^{k},p^{k-1}-\ell\sigma}$\ \ $(p\nmid \ell)$ \\\hline
$A$ & $\cong R/\pp^m$ & $\cong (C_{p^k})^{r}$ & 
$\cong(C_{p^{k-1}})^{p-2}\times C_{p^{k}}$ \\\hline
$Z$ & $=Z_0$ & $\cong(C_{p^k})^{r-p+1}$ & $=Z_0$ \\\hline
$Z_0$ &  $\gen{\Psi((\uu-1)^{m-1})}$ & $\gen{\Psi(p^{k-1}\sigma)}$ 
& $\gen{\Psi(p^{k-1})}=\gen{\Psi(\sigma)}$ \\\hline
$m$ &  & $k(p-1)+1$ & $(k-1)(p-1)+1$ \\\hline
$\EE\calf{\sminus}\{A\}$ & (see Table \ref{tbl:(d)}) &  $\calb_0$ 
& $\calh_0$ \\\hline
\end{tabular}
\end{center}
\caption{} \label{tbl:fin(abc)}
\end{table}

\begin{proof}  If $\EE\calf=\{A\}$, then $A\nsg\calf$ by Proposition 
\ref{AFT-E}(c), contradicting the assumption that $O_p(\calf)=1$.  Thus 
$\EE\calf\supsetneqq\{A\}$.

\smallskip

\noindent\boldd{Case 1: $\Psi(\sigma)\in\Fr(Z)$.} In this case, $\Psi$ is 
surjective by Lemma \ref{A/Fr(A)}(d) and since $\Fr(Z)\le\Fr(A)$. By Lemma 
\ref{A/Fr(A)}(e), $\Psi(\sigma)=1$, $\rk(A)=p-1$, $Z=Z_0$, and $A\cong 
R/\pp^m$. In particular, $\Ker(\Psi)=\gen{\sigma,(\uu-1)^m}$, while 
$Z_0=\gen{\Psi((\uu-1)^{m-1})}$ by Lemma \ref{A/Fr(A)}(c). We are thus in 
the situation of (a). 

\smallskip

\noindent\boldd{Case 2: $\Psi(\sigma)\notin\Fr(Z)$ and $A$ is 
homocyclic.} Recall that $k\ge2$ is such that $A$ has exponent $p^k$.

If $\rk(A)<p$, then $\Psi$ is onto by Lemma \ref{A/Fr(A)}(d), and 
$\Omega_1(Z)=C_{\Omega_1(A)}(\UUU)$ has rank $1$ by Lemma 
\ref{min.act.props}(c) and since $\Omega_1(A)$ is minimally active and 
indecomposable by Lemma \ref{min.act.}(b). Thus $Z$ is cyclic, and since 
$A$ is homocyclic of rank at least $2$, $A/Z\cong R/\pp^{m-1}$ also has 
exponent $p^k\ge p^2$. Hence $\rk(A)=\rk(A/Z)=p-1$. Also, 
$(A/Z)\big/\mho^{k-1}(A/Z)\cong(C_{p^{k-1}})^{p-1}$, so $|Z|=p$, and 
$Z=Z_0$. Thus $|A|=|A/Z|\cdot|Z|=p^m$, and $m\equiv0$ (mod $p-1$) since $A$ 
is homocyclic of rank $p-1$. But then $\Psi(\sigma)\notin Z_0$ by Lemma 
\ref{A/Fr(A)}(f), a contradiction. 

Thus $\rk(A)\ge p$, and $\Psi(\sigma)\notin\Fr(A)$ by Lemma 
\ref{A/Fr(A)}(d). So the homomorphism $\4\Psi\:\F_p\UUU\too A/\Fr(A)$ is 
injective, and $A/\Fr(A)$ contains a Jordan block $\Im(\4\Psi)$ of rank 
$p$. Since $A$ is homocyclic of exponent $p^k$, $|\Im(\Psi)|\ge 
p^{pk}$, and thus $\Ker(\Psi)=p^k\Z_p\UUU$. So 
$\Im(\Psi)\cong\Z/p^k\UUU$. Also, 
	\[ p^{m-1} = |A/Z| = |\Im(\Psi)/\gen{\Psi(\sigma)}| = p^{k(p-1)}, 
	\]
and so $m=k(p-1)+1\equiv1$ (mod $p-1$). 

Now, $Z_0=\gen{\Psi(p^{k-1}\sigma)}$, and $\Psi(\sigma)\notin Z_0$ since 
$k\ge2$. Hence $\EE\calf\sminus\{A\}=\calb_0$ by Lemma \ref{A/Fr(A)}(f). 
Also, $\Psi$ is surjective if and only if $\rk(A)=p$ (Lemma 
\ref{A/Fr(A)}(d)), and we are in the situation of case (b). 

\smallskip

\noindent\boldd{Case 3: $\Psi(\sigma)\notin\Fr(Z)$ and $A$ is not 
homocyclic.} Set $k'=[m/(p-1)]$. Since $m\equiv0,1$ (mod $p-1$) by 
Lemma \ref{A/Fr(A)}(f), $A/Z\cong R/\pp^{m-1}$ has exponent $p^{k'}$, and 
hence $\mho^{k'}(A)\le Z$. So $p^{k'}(\uu-1)\in\Ker(\Psi)$. 

Now, $\Psi$ is onto by Lemma \ref{A/Fr(A)}(g), and hence 
$\rk(A)\le p$ and $Z$ is cyclic by Lemma \ref{A/Fr(A)}(d). If 
$\Psi(\sigma)\notin Z_0$, then $Z_0<Z$ and is not a direct factor, so 
$\EE\calf\cap\calh=\emptyset$ by Lemma \ref{l2:s/a}(a). Thus 
	\beqq \textup{$\Psi$ onto \quad$\implies$\quad $\Psi(\sigma)\in Z_0$ or 
	$\EE\calf\cap\calh=\emptyset$.} \label{e:Z=Z0orNoH} \eeqq

\smallskip

\noindent\boldd{Case 3.1: $\mho^{k'}(A)\ne1$.} Since 
$\mho^{k'}(A)\le Z$ is invariant under the action of $G=\autf(A)$, Lemma 
\ref{O_p(F)=1} implies that $\mho^{k'}(A)=Z_0$ and 
$\EE\calf\cap\calh\ne\emptyset$. In particular, 
$\mho^{k'}(A)=\gen{\Psi(p^{k'})}$ has order $p$, and $\Psi(\sigma)\in Z_0$ 
by \eqref{e:Z=Z0orNoH}. Thus $\gen{\Psi(\sigma)}=Z_0=\gen{\Psi(p^{k'})}$, 
so $p^{k'}-\ell\sigma\in\Ker(\sigma)$ for some $\ell$ prime to $p$. 

Now, $\rk(A)<p$ by Lemma \ref{A/Fr(A)}(d) and since $\Psi(\sigma)\in 
Z_0\le\Fr(A)$. Also, $Z=Z_0\gen{\Psi(\sigma)}=Z_0$ by Lemma 
\ref{A/Fr(A)}(a), and $m\equiv1$ (mod $p-1$) by Lemma \ref{A/Fr(A)}(f) and 
since $\Psi(\sigma)\in Z_0$. So $|A|=|A/Z|\cdot|Z|=p^m=p^{k'(p-1)+1}$, and 
$A/\mho^{k'}(A)$ has exponent $k'$, rank at most $p-1$, and order $p^{k'(p-1)}$. 
This proves that $A/\mho^{k'}(A)$ is homocyclic of rank $p-1$ and exponent 
$p^{k'}$, and hence that $A\cong(C_{p^{k'}})^{p-2}\times C_{p^{k'+1}}$. Thus 
$k'=k-1$ (recall $A$ has exponent $p^k$). This also 
shows that $\gen{p^{k-1},\sigma}$ has index $p$ in $\Ker(\Psi)$, and hence 
(since $p^{k-1}-\ell\sigma$ is in the kernel) that 
$\Ker(\Psi)=\gen{p^{k},p^{k-1}-\ell\sigma}$.

If $k=2$ and $\ell\equiv1$ (mod $p$), then $\Ker(\Psi)=\gen{p^2,p-\sigma}$, 
and $(\uu-1)^{p-1}\in\Ker(\Psi)$ by Lemma \ref{l:(u-1)^m}. But then 
$\Psi((\uu-1)^{p-2})\in Z$ where $Z=Z_0=\mho^1(A)$, and this is impossible 
since $A/\mho^1(A)$ has rank $p-1$.

Finally, $\EE\calf=\{A\}\cup\calh_0$ by Lemma \ref{A/Fr(A)}(e) and since 
$\EE\calf\cap\calh\ne\emptyset$. We are thus in the situation of case (c). 

\smallskip

\noindent\boldd{Case 3.2: $\mho^{k'}(A)=1$.} 
Since $\Psi(\sigma)\in Z_0\sminus1$ or $\EE\calf\cap\calh=\emptyset$ by 
\eqref{e:Z=Z0orNoH}, we have $m\equiv1$ (mod $p-1$) by Lemma 
\ref{A/Fr(A)}(f). Thus $m-1=k'(p-1)$, so $1\ne 
Z_0=\gen{\Psi((1-\uu)^{m-1})}=\gen{\Psi(p^{k'-1}\sigma)}$ by Lemma 
\ref{l:(u-1)^m}.  Since $p^{k'-1}\sigma\notin\Ker(\Psi)$, we have 
$|Z|\ge|\Psi(\sigma)|=p^{k'}$, and $|A|=|A/Z|\cdot|Z|\ge p^{m-1+k'}=p^{k'p}$. 
Hence $A=\Im(\Psi)$ is homocyclic of rank $p$ and exponent $p^{k'}$, 
contradicting our assumption.
\end{proof}

We now give some examples to show that all cases listed in Proposition 
\ref{l:fin(abc)} can occur.

\begin{Ex} \label{ex:fin(abc)}
We list here some examples of pairs $(G,A)$ satisfying the hypotheses of 
Theorem \ref{t3:s/a}. In all cases, we assume that $G\in\GGG$ and 
$\UUU\in\sylp{G}$. By Lemma 
\ref{min.act.}(b), $\Omega_1(A)$ and $A/\Fr(A)$ must be minimally active 
and indecomposable.
\begin{enumerate}[{ \ }(a): ] 

\item Each $\Q_pG$-module of dimension $p-1$ whose restriction to $\UUU$ is 
isomorphic to the canonical action on $\Q_p(\zeta)$ can be used to 
construct homocyclic examples of arbitrary exponent, by adding scalars as 
needed to meet one of the conditions in Table \ref{tbl:(d)}. 

More interesting are examples where $A$ is not homocyclic. By Proposition 
\ref{l:fin(abc)}, $\Omega_1(A)$ and $A/\Fr(A)$ must be not only minimally 
active and indecomposable of dimension $p-1$, but also not simple. By 
Table \ref{tbl:reps}, this occurs only when $A_p\le G\le\Sigma_p\times 
C_{p-1}$, $\SL_2(p)\le G\le\GL_2(p)$, or $\PSL_2(p)\le G\le\PGL_2(p)\times 
C_{p-1}$. By Proposition \ref{dim(p-1)}(a), for each minimally active, 
indecomposable $\F_pG$-module $V$ of dimension $p-1$, there is a 
$\Z_pG$-lattice $\Lambda$ such that $\Lambda/p\Lambda\cong V$. If 
$0\ne V_0<V$ is a nontrivial proper $\F_pG$-submodule, and $\Lambda_0<\Lambda$ 
is such that $p\Lambda<\Lambda_0$ and $\Lambda_0/p\Lambda\cong V_0$, then 
we can take $A\cong\Lambda_0/p^k\Lambda$ for arbitrary $k\ge2$.

\item These are homocyclic, and there are many such examples, obtained from 
the $\F_pG$-modules in Table \ref{tbl:reps} of dimension at least $p$ (all 
of them are reductions of lattices in $\Q_pG$-modules). Lemma 
\ref{min.act.}(c) together with Theorem \ref{t3:s/a} imply, very roughly, 
that each $\F_pG$-module that yields simple fusion systems with elementary 
abelian $A$ and with $\EE\calf\subseteq\{A\}\cup\calb$ will also give 
simple fusion systems with $A$ of exponent $p^k$ for arbitary $k>1$. Some 
of the resulting fusion systems are realizable (see Table \ref{tbl:type3}), 
while ``most'' are exotic. 

\item Fix an odd prime $p$, $k\ge2$, and $\ell$ prime to $p$ such that 
$\ell\not\equiv1$ (mod $p$) if $k=2$. Set $\4G=\Sigma_p\times C_{p-1}$ and 
$G_0=O^{p'}(\4G)\cong A_p$, and set $\UUU=\Gen{(1\,2\cdots p)}\in\sylp{\4G}$. 
Let $\Lambda\cong(\Z_p)^p$ be the $\Z_p\4G$-lattice upon which $\Sigma_p$ 
acts by permuting a $\Z_p$-basis $\{e_1,\dots,e_p\}$, and where the factor 
$C_{p-1}$ acts via multiplication by $(p-1)$-st roots of unity in 
$\Z_p^\times$. Now define 
	\[ A = \Lambda \big/ \Gen{p^k\Lambda, 
	p^{k-1}e_i-\ell(e_1+\dots+e_p) \,\big|\, 1\le i\le p}\,, \]
and let $\4e_i\in A$ be the class of $e_i\in\Lambda$. 
This defines a finite $\Z_p\4G$-module of rank $p-1$ and exponent 
$p^k$, as described in the last column in Table \ref{tbl:fin(abc)}, where 
$\Psi\:\Z_p\UUU\too A$ is defined by setting $\Psi(\xi)=\xi\cdot\4e_1$. 

Set $Z=C_A(\UUU)$. Note that $|p^{k-1}A|=p$, and 
$A/p^{k-1}A\cong\Lambda/\gen{p^{k-1}\Lambda,\4e_1+\dots+\4e_p}\cong 
\Z_p[\zeta]/(p^{k-1})$. Hence $Z\ge p^{k-1}A$, and $|Z|>p$ 
only if $p^{k-2}(\4e_1+2\4e_2+\dots+p\4e_p)\in Z$. But this last is 
the case only if $p^{k-2}((\4e_1+\dots+\4e_p)-p\4e_1)=0$, which is not 
possible since we assumed that either $k\ge3$ or $\ell\not\equiv1$ (mod 
$p$). Thus $|Z|=p$ in all cases.

Now set $S=A\rtimes\UUU$, and set $\4G^\vee=N_{\4G}(\UUU)$. (Note that 
$Z=Z_0$ in the notation of \ref{n:not1}.) Define 
$\mu_A\:\4G^\vee\too\Delta$ as in Notation \ref{n:not2}. One easily checks 
that $\mu_A(\4G^\vee)=\Delta$. Set $G=G_0\mu_A^{-1}(\Delta_{-1})$. This now 
defines an action which satisfies the conditions in Theorem \ref{t3:s/a}, 
including condition (d.iii$''$).

\end{enumerate}
\end{Ex}

We now combine Theorem \ref{t3:s/a} with Proposition \ref{l:fin(abc)} 
to prove our main result on simple fusion systems over finite 
$p$-groups with an abelian subgroup of index $p$ and exponent at least 
$p^2$. Recall that $\EE\calf$ is the set of essential subgroups 
in a fusion system $\calf$ (Definition \ref{d:subgroups}), and that $\GGG$ 
is a certain class of finite groups (Definition \ref{d:min.act.}).

\begin{Th} \label{ThA}
Fix an odd prime $p$.
\begin{enuma} 

\item Let $\calf$ be a simple fusion system over a finite 
nonabelian $p$-group $S$ with an abelian subgroup $A<S$ of index 
$p$ such that $A\in\EE\calf$. Let $k$ be such that $A$ has exponent 
$p^k$, and assume $k\ge2$ ($A$ is not elementary abelian). Set 
$G=\autf(A)$, $\UUU=\Aut_S(A)\in\sylp{G}$, $V=\Omega_1(A)$, and 
$V_0=\mho^{k-1}(A)\le V$. Let $G^\vee_{(V)}=\bigl\{\alpha\in 
N_G(\UUU)\,\big|\, [\alpha,C_V(\UUU)]\le[\UUU,V]\bigr\}$ and 
$\mu_V\:G^\vee_{(V)}\too\Delta$ be as in Notation \ref{n:not2}. Let 
$\Psi\:\Z_p\UUU\too A$ be as in Notation \ref{n:not3}. Then $G\in\GGG$, 
restriction defines an isomorphism $G\cong\autf(V)$, and \\
	\[ \parbox{\shorter}{$V$ is a faithful, minimally active, 
	indecomposable $\F_pG$-module, and $[G,V_0]=V_0$. Also, 
	one of the cases in Table \ref{tbl:ThA} holds, where 
	$V$ has no $1$-dimensional submodule in cases \textup{(iv$'$)} and 
	\textup{(iv$''$)}.} 
	\tag*{\starfin} \]
\begin{small} 
\begin{table}[ht]
\[ \renewcommand{\arraystretch}{1.5}\renewcommand{\arraycolsep}{1mm}
\newcommand{\halfup}[1]{\raisebox{2.2ex}[0pt]{$#1$}}
\begin{array}{|c|c||c|c|c|c|c|c|} \hline
\textup{Tbl.\ref{tbl:(d)}} & \textup{Tbl.\ref{tbl:fin(abc)}} & \dim(V) & 
r=\dim(V_0) & \Ker(\Psi) & \mu_V(G^\vee_{(V)}) & G= & 
\EE\calf{\sminus}\{A\} \\\hline

& \textup{(b)} & \ge p & V_0=V & \gen{p^k} &  &  & \\\cline{2-5}

\halfup{\textup{(iv$''$)}} & & & \mbox{\tiny{$1\le r\le p-1$}} &  & \ge\Delta_0 
& O^{p'}(G)\mu_V^{-1}(\Delta_0) & \halfup{\calb_0} \\ 
\cline{1-1}\cline{4-4}\cline{8-8}

\textup{(iv$'$)} &  &  &  &  &  &
& \bigcup_{i\in I}\calb_i \\\cline{1-1}\cline{6-8}

\textup{(i)} &  &  & \halfup{V_0=V} & \raisebox{-2pt}[0pt]{$\langle 
p^k,\sigma,$} \hfill & \Delta & O^{p'}(G)\cdot G^\vee_{(V)} & 
\calh_0\cup\calb_* \\\cline{1-1}\cline{4-4}\cline{6-8}

\textup{(ii)} & \halfup{\textup{(a)}} & \halfup{p-1} & & 
\raisebox{2pt}[0pt]{$p^{k-1}(\uu-1)^r\rangle$} & \Delta & O^{p'}(G)\cdot 
G^\vee_{(V)} & \calb_0\cup\calh_* \\\cline{1-1}\cline{6-8}

\textup{(iii$'$)} &  &  & 
\raisebox{2.2ex}[0pt]{$r=p-2$} &  & 
&  & \bigcup_{i\in I}\calh_i \\\cline{1-1}\cline{4-4}\cline{8-8}
 &  &  & \mbox{\tiny{$1\le r\le p-1$}} &  & \ge\Delta_{-1} & 
O^{p'}(G)\mu_V^{-1}(\Delta_{-1})  &  \\ \cline{2-5}
\halfup{\textup{(iii$''$)}} & \textup{(c)} & p-1 & r=1 & 
\gen{p^{k},p^{k-1}-\ell\sigma} & & & \halfup{\calh_0} \\\hline
\end{array}
\]
\caption{The sets in the last column are as defined in Notation 
\ref{n:not1} and \ref{n:not2a}.} \label{tbl:ThA}
\end{table}
\end{small}

\item Conversely, assume that $G\in\GGG$ and $\UUU\in\sylp{G}$, and that 
$V$ is a faithful, minimally active, indecomposable $\F_pG$-module 
satisfying \starfin\ for some submodule $0\ne V_0\le V$, where 
$G^\vee$ is the subgroup of all elements $g\in N_G(\UUU)$ such that 
$[g,C_V(\UUU)]\le[\UUU,V]$. Then for each $k\ge2$, there is a simple fusion 
system $\calf$ over a finite $p$-group $S$ containing an abelian subgroup 
$A$ of index $p$, and such that $G\cong\autf(A)$, $A$ has exponent 
$p^k$, $\Omega_1(A)\cong V$ and $\mho^{k-1}(A)\cong V_0$ as 
$\F_pG$-modules, and with $\EE\calf\sminus\{A\}$ as described in Table 
\ref{tbl:(d)}. Furthermore, any other simple fusion system with these 
properties, and with the same essential subgroups, is isomorphic to 
$\calf$. 

\end{enuma}
\end{Th}

\begin{proof} \textbf{(a) } Set $\UUU=\Aut_S(A)\in\sylp{G}$ and $Z=C_A(\UUU)$.

Under the above assumptions, $G\in\GGG$ by Lemma \ref{min.act.}(a), $V$ is 
faithful, minimally active, and indecomposable by Lemma \ref{min.act.}(b), 
and $\dim(V)=\rk(A)\ge p-1$ by Proposition \ref{l:fin(abc)}. In 
particular, restriction induces a monomorphism $G=\autf(A)\too\autf(V)$, 
and this is an isomorphism by the extension axiom (Definition 
\ref{d:saturated}) and since $A=C_S(V)$. From now on, we identify 
$G=\autf(V)$.

Set $G^\vee_{(A)} = \bigl\{\alpha\in N_G(\UUU) \,\big|\, 
[\alpha,C_A(\UUU)]\le[\UUU,A] \bigr\}$. Thus $G^\vee_{(A)}\le 
G^\vee_{(V)}$, and $\mu_A\:G^\vee_{(A)}\too\Delta$ is as in Notation 
\ref{n:not2}. For some $\Delta_x\in\{\Delta,\Delta_0,\Delta_{-1}\}$, 
$\mu_A(G^\vee_{(A)})\ge\Delta_x$ and $G=O^{p'}(G)\mu_A^{-1}(\Delta_x)$ by 
Theorem \ref{t3:s/a}(d). So the same holds if we replace $G^\vee_{(A)}$ by 
$G^\vee_{(V)}$ and $\mu_A$ by $\mu_V$.

Assume we are in case (a) of Table \ref{tbl:fin(abc)}. Then $Z=Z_0$ has order 
$p$, so $|A|=p^m$, and $m\equiv0$ (mod $p-1$) if and only if $A$ is 
homocyclic. Also, $\Psi(\sigma)=1$. Thus we are in case (i), (iii$''$), 
(iv$'$), or (iv$''$) of Table \ref{tbl:(d)} if $A$ is homocyclic, or in case 
(ii), (iii$'$), (iii$''$), or (iv$''$) if $A$ is not homocyclic. Since 
$m\equiv0$ (mod $p-1$) and $\Psi(\sigma)=1$ when $A$ is homocyclic, (iv$''$) 
is a special case of (iv$'$). The other information follows from the two 
earlier tables. 

Now assume we are in case (b) or (c) in Table \ref{tbl:fin(abc)}. Then 
$\Psi(\sigma)\notin\Fr(Z)$, so this corresponds to case (iii$''$) or (iv$''$) 
in Table \ref{tbl:(d)}. Since $\EE\calf\sminus\{A\}=\calh_0$ in cases (c) 
and (iii$''$), and $\EE\calf\sminus\{A\}=\calb_0$ in cases (b) and (iv$''$), 
these are the only possible correspondences. In case (b), $A$ is 
homocyclic, and so we have $V_0=V$ in Table \ref{tbl:ThA}. In case 
(c), $A\cong(C_{p^{k-1}})^{p-2}\times C_{p^{k}}$, so $V=\Omega_1(A)$ has a 
submodule $V_0=\mho^k(A)$ of rank $1$. 

There is a surjective homomorphism $\psi\:A\Right2{}V_0$ of 
$\Z_pG$-modules, defined by setting $\psi(a)=a^{p^{k-1}}$. Since $[G,A]=A$ 
by Theorem \ref{t3:s/a}(c), we have $[G,V_0]=V_0$. 

If $V=\Omega_1(A)$ has a 1-dimensional $\F_pG$-submodule $W$, then $W\in 
C_V(\UUU)\le C_A(\UUU)$, so $W=Z_0=C_A(\UUU)\cap[\UUU,A]$ by Theorem 
\ref{t3:s/a}(b), which is impossible in cases (iv$'$) and (iv$''$) of Theorem 
\ref{t3:s/a}(d).

\smallskip
 
\noindent\textbf{(b) } Fix $G\in\GGG$ and $\UUU\in\sylp{G}$, and let 
$V$ be an $\F_pG$-module that satisfies \starfin, where $G^\vee$ is the 
subgroup of all elements $g\in N_G(\UUU)$ such that 
$[g,C_V(\UUU)]\le[\UUU,V]$. Fix $k\ge2$.

Assume we have chosen a finite $\Z_pG$-module $A$ such that 
$\Omega_1(A)\cong V$ as $\F_pG$-modules, and such that either $\dim(V)=p-1$ 
and $|C_A(\UUU)|=p$, or $\dim(V)\ge p$ and $A$ is homocyclic. Let 
$G^\vee_{(A)}\le G^\vee_{(V)}$ be as in the proof of (a). If 
$\rk(V)=\rk(A)=p-1$, then since $C_A(\UUU)\cong C_V(\UUU)$ has order $p$, 
$G^\vee_{(A)}=G^\vee_{(V)}=N_G(\UUU)$ and $\mu_A=\mu_V$. So the properties 
of $G^\vee_{(V)}$ and $\mu_V$ in Table \ref{tbl:ThA} also hold for 
$G^\vee_{(A)}$ and $\mu_A$. 

If $\rk(V)=\rk(A)\ge p$ and $A$ is homocyclic, then 
$\mu_V(G^\vee_{(V)})\ge\Delta_0$, and $G=O^{p'}(G)\mu_V^{-1}(\Delta_0)$. In 
such cases, we could have $G^\vee_{(A)}<G^\vee_{(V)}$, but for each 
$\alpha\in\mu_V^{-1}(\Delta_0)$ of order prime to $p$, $\alpha$ acts 
trivially on $C_V(\UUU)\cong\Omega_1(C_A(\UUU))$ by definition of $\mu_V$ 
(and Lemma \ref{mod-Fr}), and hence also acts trivially on $C_A(\UUU)$ (see 
\cite[Theorem 5.2.4]{Gorenstein}). Thus 
$\mu_V^{-1}(\Delta_0)=\mu_A^{-1}(\Delta_0)$, and so the information in 
Table \ref{tbl:ThA} still holds if we replace $G^\vee_{(V)}$ and $\mu_V$ by 
$G^\vee_{(A)}$ and $\mu_A$.

\noindent\textbf{Cases \ref{l:fin(abc)}(a,b): } 
Assume that we are in Case (a) or (b) in Proposition \ref{l:fin(abc)}. By 
Proposition \ref{dim(p-1)}(a), there is a $\Z_pG$-lattice $\Lambda$ such 
that $\Lambda/p\Lambda\cong V$ as $\F_pG$-modules. Let 
$\Lambda_0\le\Lambda$ be such that $\Lambda_0\ge p\Lambda$ and 
$\Lambda_0/p\Lambda\cong V_0$, and set $A=\Lambda_0/p^k\Lambda$. Set 
$S=A\rtimes\UUU$, $Z=Z(S)=C_A(\UUU)$, $S'=[S,S]=[\UUU,A]$, and 
$Z_0=Z\cap S'$. 

We first check that conditions (a)--(d) in Theorem \ref{t3:s/a} all hold. 
Condition (d) follows immediately from \starfin. Condition (a) ($|Z_0|=p$) 
follows from Lemma \ref{min.act.|Z0|=p}(c) and since $A$ is defined to be a 
quotient group of a $\Z_pG$-lattice. 

Assume $1\ne B\le Z$ is $G$-invariant; we claim that $B=Z_0$. If 
$\rk(A)=p-1$, then $Z=Z_0$, and there is nothing to prove. If not, then 
$\rk(A)\ge p$, we are in case (iv$''$), and so $V$ has no 1-dimensional 
$\F_pG$-submodule. Thus $\rk(Z)\ge\rk(B)\ge2$, and hence $\dim(V)\ge p+1$. 
If $\dim(V)=p+1$, then by Lemma \ref{rk(V)=p+1}, every nontrivial 
$\F_pG$-submodule has nontrivial action of $\UUU$, contradicting the 
assumption $B\le Z$. If $\dim(V)\ge p+2$, then $V$ is simple by 
\cite[Proposition 3.7(c)]{indp2}. So Condition \ref{t3:s/a}(b) 
holds in all cases.

If $\rk(A)\ge p$, then $V_0=V$, $A/\Fr(A)\cong V_0$, and so $[G,A]=A$ 
since $[G,V_0]=V_0$. If $\rk(A)=p-1$, then by Lemma \ref{ZU-modules}(a,b,c), 
$\Lambda_0|_\UUU\cong\Z_p[\zeta]$ as $\Z_p\UUU$-modules, and the radical of 
$\Lambda_0|_\UUU$ has index $p$. Thus $p\Lambda$ is contained in the 
radical, and in $[G,A]=A$ since $V_0\cong\Lambda_0/p\Lambda$ and 
$[G,V_0]=V_0$. This proves \ref{t3:s/a}(c).

By Theorem \ref{t3:s/a}, there is a unique simple fusion system $\calf$ 
over $A\rtimes\UUU$ such that $G=\autf(A)$, $A\in\EE\calf$, and 
$\EE\calf\sminus\{A\}$ is as described in Table \ref{tbl:ThA}. 
Since $A$ is unique (up to isomorphism of $\Z_pG$-modules) by Proposition 
\ref{A1=A2}(a) (in case \ref{l:fin(abc)}(b)) or \ref{A1=A2}(b) (in case 
\ref{l:fin(abc)}(a)), this shows that $\calf$ is uniquely determined by $V$.

\noindent\textbf{Case \ref{l:fin(abc)}(c): } Assume that we are in Case (c) 
in Proposition \ref{l:fin(abc)}. In particular, $\dim(V)=p-1$ and 
$\dim(V_0)=1$. By Lemma \ref{l:top&bottom}(c), there is a projective, 
minimally active $\F_pG$-module $W>V$ such that $\dim(W)=p$ and thus 
$\dim(W/V)=1$. 

By Proposition \ref{dim(p-1)}(b), there is a $\Z_pG$-lattice $\Lambda$ such 
that $\Lambda/p\Lambda\cong W$, and such that $\Lambda$ has a 
$\Z_pG$-submodule $\Lambda_0=C_{\Lambda}(G_0)$ of rank $1$. In particular, 
$\Lambda$ is free as a $\Z_p\UUU$-module since $W$ is free as an 
$\F_p\UUU$-module. Let $\Lambda_V<\Lambda$ be the $\Z_pG$-sublattice of 
index $p$ such that $\Lambda_V/p\Lambda\cong V$. Define 
	\[ \5A = \Lambda \big/ \bigl( p^k\Lambda+p^{k-1}\Lambda_V+p\Lambda_0 
	\bigr) \cong C_{p^k}\times(C_{p^{k-1}})^{p-2} \times C_p  \,. \]
Then $\Omega_1(\5A)$ is a $p$-dimensional $\F_pG$-module, and contains a 
2-dimensional submodule 
	\begin{align*} 
	\5A_0 &= \bigl( p^{k-1}\Lambda+\Lambda_0 \bigr) \big/ 
	\bigl( p^k\Lambda+p^{k-1}\Lambda_V+p\Lambda_0 \bigr) \\
	&\cong \bigl(p^{k-1}\Lambda/(p^k\Lambda+p^{k-1}\Lambda_V)\bigr) 
	\oplus \bigl(\Lambda_0/p\Lambda_0\bigr)
	\cong (W/V) \oplus V_0 \,. 
	\end{align*}

Now, $V_0\cong W/V$ as $\F_p[N_G(\UUU)]$-modules by Lemma 
\ref{l:top&bottom}(a) and since $\dim(W)=p$. Also, $O^{p'}(G)$ acts 
trivially on each of them and $G=O^{p'}(G)N_G(\UUU)$ by the Frattini 
argument, so $V_0$ and $W/V$ are isomorphic as $\F_pG$-modules, and any 
$\F_p$-linear isomorphism is $\F_pG$-linear. Hence for fixed 
$a\in\Lambda\sminus\Lambda_V$ and fixed $\ell$ prime to $p$ such that 
$\ell\not\equiv1$ (mod $p$) if $k=2$, the quotient group 
	\[ A = \5A/\5A_1 \qquad\textup{where}\qquad
	\5A_1 = \Gen{[p^{k-1}a-\ell\sigma\cdot a]} \le \5A_0 \]
is a quotient group of $\5A$ where the two summands of $\5A_0$ have been 
identified, and hence is a $\Z_pG$-module. Here, as usual, 
$\sigma=\sum_{u\in\UUU}u\in\Z_p\UUU$. Note that $\5A_1$ is independent of 
the choice of $a$, and depends on $\ell$ only modulo $p$.

Since $\5A/\5A_0\cong\Lambda/(p^{k-1}\Lambda+\Lambda_0)$, where 
$\Lambda/\Lambda_0\cong\Z_p[\zeta]$ as $\Z_p\UUU$-modules by Lemma 
\ref{ZU-modules}(c), we have that $|C_{\5A/\5A_0}(\UUU)|=p$, and is 
generated by the class of $p^{k-2}(u-1)^{p-2}a$ for $1\ne u\in\UUU$.  
The class of $p^{k-2}(u-1)^{p-2}a$ in $A$ is fixed by $\UUU$ if and only if 
as classes in $\5A$, 
	\[ [p^{k-2}(u-1)^{p-1}a] = [p^{k-2}(\sigma-p)a] = 
	[p^{k-2}\sigma\cdot a] - [p^{k-1}a] \in \5A_1 \]
(where the second equality holds by Lemma \ref{l:(u-1)^m}). But this fails 
to hold under our hypotheses: either $k>2$, in which case 
$[p^{k-2}\sigma\cdot a]=0$ and $[p^{k-1}a]\ne0$ in $\5A$; or $k=2$ and 
$\ell\not\equiv1$ (mod $p$), in which case $[\sigma\cdot 
a]-[pa]\notin\5A_1$. Thus in all such cases, $C_A(\UUU)=\5A_0/\5A_1$ and 
has order $p$. 

Since $V_0\cong \Lambda_0/p\Lambda_0\cong\5A_0/\5A_1=C_A(\UUU)$ as 
$\F_pG$-modules, and since $V$ and $\Omega_1(A)$ are both 
$(p-1)$-dimensional minimally active $\F_pG$-modules, Lemma 
\ref{l:top&bottom}(d) applies to show that $V\cong\Omega_1(A)$ as 
$\F_pG$-modules. 

By construction, $(G,A)$ satisfies all of the conditions in case (c) of 
Proposition \ref{l:fin(abc)}, as well as condition (d.iii$''$) of Theorem 
\ref{t3:s/a}. Since $|Z|=|C_A(\UUU)|=p$, conditions (a) and (b) in 
\ref{t3:s/a} also hold: $|Z_0|=p$, and no nontrivial subgroup of $Z$ is 
$G$-invariant except possibly $Z_0=Z$. Finally, \ref{t3:s/a}(c) holds 
($[G,A]=A$) since $A/[\UUU,A]=A/\Omega_{k-1}(A)$ has order $p$, and has 
nontrivial action of $G$ since $G$ acts nontrivially on $Z_0=\mho^{k-1}(A)$ 
(since $\mu_A(G^\vee)\ge\Delta_{-1}$).

By Theorem \ref{t3:s/a}, there is a unique simple fusion system $\calf$ 
over $A\rtimes\UUU$ such that $G=\autf(A)$, $A\in\EE\calf$, and 
$\EE\calf=\{A\}\cup\calh$. Since $A$ is unique up to isomorphism of 
$\Z_pG$-modules by Proposition \ref{A1=A2}(b), this shows that $\calf$ is 
uniquely determined by $V$.
\end{proof}

\section{Simple fusion systems over nonabelian discrete $p$-toral groups 
with abelian subgroup of index $p$}
\label{s:infinite}

We now focus on the case where $A$ and $S$ are infinite. Since most of the 
results in Section \ref{s:fin+inf} assume Notation \ref{n:not1}, and in 
particular that $S$ contains a unique abelian subgroup of index $p$, we 
begin by proving that this always holds when $O_p(\calf)=1$.

\begin{Lem} \label{|A/Z|=infty}
Let $\calf$ be a saturated fusion system over an infinite discrete 
$p$-toral group $S$ with an abelian subgroup $A$ of index $p$. Assume also 
that $O_p(\calf)=1$. Then $|A/Z(S)|=\infty$.
\end{Lem}

\begin{proof} Let $S_0\nsg S$ denote the identity component of $S$. 
If $|A/Z(S)|<\infty$, then $S_0\le Z(S)$, so $S_0$ is contained in (and 
is characterisitic in) each $P\in\EE\calf$. Thus 
$S_0\nsg\calf$ by Proposition \ref{AFT-E}(c), so $S_0\le O_p(\calf)=1$, 
contradicting the assumption that $|A|=\infty$.
\end{proof}

\begin{Cor} \label{>1abel}
If $\calf$ is a saturated fusion system over an infinite discrete $p$-toral group 
$S$ with an abelian subgroup $A$ of index $p$, and $O_p(\calf)=1$, then $A$ is 
the only abelian subgroup of index $p$ in $S$. 
\end{Cor}

\begin{proof} Since $|A/Z(S)|=\infty$ by Lemma \ref{|A/Z|=infty}, this follows 
from Lemma \ref{1abel}.
\end{proof}

\begin{Lem} \label{disc.torus2}
Assume Notation \ref{n:not1}. If $|A|=\infty$ and $O_p(\calf)=1$, then $A/Z$ 
and $S'$ are both discrete $p$-tori of rank $p-1$. 
\end{Lem}

\begin{proof} Since $|A/Z|=\infty$ by Lemma \ref{|A/Z|=infty} and 
$C_{A/Z}(\UUU)=Z_2/Z$ has order $p$ by Lemma \ref{|Z0|=p}, 
$A/Z\cong(\Z/{p^\infty})^{p-1}$ by Lemma \ref{ZU-modules}(d). Also, 
$A/Z\cong S'$ by Lemma \ref{1abel}.
\end{proof}

\begin{Lem} \label{A0=A}
Assume Notation \ref{n:not1}. If $|A|=\infty$ and $O_p(\calf)=1$, then 
$A=ZS'$. As one consequence, each of the sets $\calb$ and $\calh$ consists 
of  one $S$-conjugacy class.
\end{Lem}

\begin{proof} Fix a generator $u\in\UUU$, and define 
$\chi\:A/Z\Right2{}A/Z$ by setting $\chi(aZ)=[a,u]Z$. Then 
$\Im(\chi)=ZS'/Z$, and $\Ker(\chi)=Z_2/Z$ has order $p$ by Lemma 
\ref{|Z0|=p}. Since $A/Z$ is a discrete $p$-torus by Lemma 
\ref{disc.torus2}, $\chi$ must be onto, and hence $ZS'=A$. 

Thus if $P=Z\gen{x}$ and $Q=Z\gen{y}$ are two members of $\calh$, 
where $yx^{-1}\in A$, then there are $z\in Z$ and $a\in A$ such that 
$yx^{-1}=axa^{-1}x^{-1}z$. Then $y\in\9axZ$, so $Q=\9aP$, and $P$ and $Q$ 
are $S$-conjugate. A similar argument shows that all members of $\calb$ are 
$S$-conjugate.
\end{proof}

Part of the next lemma follows from Lemma \ref{O_p(F)=1} when $p$ is odd. 
But since we also need it here when $p=2$, we prove it independently of the 
earlier lemma.

\begin{Lem} \label{Z0=Z}
Assume Notation \ref{n:not1}, and also that $|A|=\infty$, $O_p(\calf)=1$, 
and $A\notin\EE\calf$. Then $Z=Z_0$, $\EE\calf=\calh$, and $A=S'$ is a 
discrete $p$-torus of rank $p-1$.
\end{Lem}

\begin{proof} Since $A\notin\EE\calf$, Lemmas \ref{B&H} and \ref{exclusion} 
imply that $\EE\calf=\calb$ or $\EE\calf=\calh$. If $\EE\calf=\calb$, then 
since $Z=Z(S)$ is normalized by $\autf(S)$ and by $\autf(P)$ for each 
$P\in\calb$ ($Z$ is characteristic in $P$ by Lemma \ref{l2:s/a}(b)), 
$Z\nsg\calf$ by Proposition \ref{AFT-E}(c), contradicting the assumption 
that $O_p(\calf)=1$. Thus $\EE\calf=\calh$.

For $P=Z\gen{x}\in\calh$, by Lemma \ref{l2:s/a}(a), $P=P_1\times P_2$, 
where $P_1=C_P(O^{p'}(\autf(P)))<Z$, $Z=P_1\times Z_0$, and $Z_0<P_2\cong 
C_p\times C_p$, and where each factor $P_i$ is normalized by $\autf(P)$. If 
$P^*\in\calh$ is another member, then $P^*=\9gP$ for some $g\in S$ by Lemma 
\ref{A0=A}, and $P_1=\9gP_1=C_{\9gP}(O^{p'}(\autf(\9gP)))$ is also 
normalized by $\autf(\9gP)$. Finally, $P_1$ is normalized by $\autf(S)$ 
since $\autf(S)=\Inn(S)\cdot N_{\autf(S)}(P)$ by the Frattini argument, and 
thus $P_1\nsg\calf$ by Proposition \ref{AFT-E}(c). So $P_1\le 
O_p(\calf)=1$, and hence $Z=Z_0$.

Since $Z=Z_0$, we have $A=S'$ by Lemma \ref{A0=A}, and so $A$ is a discrete 
$p$-torus of rank $p-1$ by Lemma \ref{disc.torus2}.
\end{proof}

When $A$ is finite and $p$ is odd, it was shown in \cite[Lemma 2.4]{indp1} 
that $O_p(\calf)=1$ and $A\notin\EE\calf$ imply $Z=Z_0$. When $A$ is finite 
and $p=2$, this is not true: for each odd prime $p$, the $2$-fusion system 
of $\PSL_2(p^2)\gen{\phi}$, where $\phi$ is a field automorphism 
of order $2$, is a counterexample. Note that in this case, $S\cong D\times 
C_2$ for some dihedral $2$-group $D$ (whose order depends on $p$).

The case $p=2$ is now very easy to handle.

\begin{Thm} \label{t:disc.2-toral}
Let $S$ be an infinite nonabelian discrete $2$-toral group with 
abelian subgroup $A<S$ of index $2$.  Let $\calf$ be a saturated fusion 
system over $S$ such that $O_2(\calf)=1$. Then $\calf$ is isomorphic to the 
$2$-fusion system of $\SO(3)$ (if $A$ is not $\calf$-essential) or of 
$\PSU(3)$ (if $A$ is $\calf$-essential).
\end{Thm}

\begin{proof} Recall that $A$ is the unique abelian subgroup of index 
$2$ in $S$ by Corollary \ref{>1abel}. So we can use Notation \ref{n:not1}. 
Also, $|Z_0|=2$ by Lemma \ref{|Z0|=p}, and $\EE\calf\sminus\{A\}=\calh$ or 
$\calb$ by Lemmas \ref{B&H} and \ref{A0=A} (and since $A\nsg\calf$ if 
$\EE\calf\subseteq\{A\}$).

\smallskip

\noindent\textbf{Case 1: } Assume first that $A\notin\EE\calf$. Then by 
Lemma \ref{Z0=Z}, $\EE\calf=\calh$, $Z=Z_0$, and $A=S'$ is a discrete 
$p$-torus of rank $1$. Thus $A\cong\Z/{2^\infty}$, where 
this group is inverted by the action of $S/A$. Also, for each 
$P\in\EE\calf=\calh$, $P\cong C_2\times C_2$ and hence 
$\autf(P)=\Aut(P)\cong\Sigma_3$. Thus there is a unique choice of fusion 
system $\calf$ on $S$, and it must be isomorphic to the fusion system of 
$\SO(3)$.

\smallskip

\noindent\textbf{Case 2: } Now assume that $A$ is $\calf$-essential, and 
set $G=\autf(A)$. Then $\Aut_S(A)\in\syl2G$ has order $2$, so $|G|=2m$ for 
some odd $m$. By Proposition \ref{GonA}, we can write $G=G_1\times 
G_2$ and $A=A_1\times A_2$, where $G_i$ acts faithfully on $A_i$ and 
trivially on $A_{3-i}$ for $i=1,2$, where $|G_1|$ is odd, 
$G_2\cong\Sigma_3$, and $A_2\cong C_{2^k}\times C_{2^k}$ for some $1\le 
k\le\infty$. 

Now, $|A_2|=\infty$ since $A_1\le Z$ and $|A/Z|=\infty$ (Lemma 
\ref{|A/Z|=infty}). Hence $Z_0$ is not 
a direct factor in $Z$, so $\calh\cap\EE\calf=\emptyset$ by Lemma 
\ref{l2:s/a}(a), and $\EE\calf=\{A\}\cup\calb$. Each $\alpha\in\autf(S)$ 
normalizes $A$ and hence normalizes $A_1$. For each $P\in\calb$ and each 
$\alpha\in\autf(P)$, $\alpha(Z)=Z$ since $Z=Z(P)$, $\alpha|_Z\in\autf(Z)$ 
extends to $\4\alpha\in\autf(S)$, and thus $\alpha(A_1)=A_1$. So 
$A_1\nsg\calf$, and $A_1\le O_2(\calf)=1$.

Thus $A\cong(\Z/{2^\infty})^2$ and $\autf(A)=G\cong\Sigma_3$. For each 
$P\in\calb=\EE\calf\sminus\{A\}$, $P\in(\Z/{2^\infty})\times_{C_2}Q_8$, and 
the subgroup isomorphic to $Q_8$ is unique. Hence $\outf(P)\cong\Sigma_3$ 
is uniquely determined, and $\calf$ is determined uniquely by $\autf(A)$. 
So $\calf$ is the 2-fusion system of $\PSU(3)$.
\end{proof}

We now focus on the cases where $p$ is odd.

\begin{Prop} \label{A-torus}
Assume Notation \ref{n:not1}. Assume also that $p$ is odd, $|A|=\infty$, 
and $O_p(\calf)=1$. Then $A$ is a discrete $p$-torus. 
If $\rk(A)\ge p$, then $Z$ is also a discrete $p$-torus, and has rank 
$\rk(A)-p+1$. 
\end{Prop}

\begin{proof} We first apply Lemma \ref{A2<A1<A}, with $A_1=A$ and 
$A_2$ the identity component of $A$. Since $S'$ is a discrete $p$-torus by 
Lemma \ref{disc.torus2}, we have $A_2Z\ge ZS'=A=A_1$ by Lemma \ref{A0=A}. 
So by Lemma \ref{A2<A1<A}, $A\le A_2Z_0=A_2$, the last equality since 
$Z_0\le S'\le A_2$. Thus $A$ is a discrete $p$-torus.

Set $G=\autf(A)$ and $V=\Omega_1(A)$, and choose $1\ne 
u\in\UUU\in\sylp{G}$. By Lemma \ref{min.act.}(a,b), $G\in\GGG$ (so 
$|\UUU|=p$), and $V$ is faithful, minimally active, and indecomposable as 
an $\F_pG$-module. So if $\dim(V)=\rk(A)\ge p$, then by Lemma 
\ref{min.act.props}(a,b), the action of $u$ on $V$ has a Jordan block of 
length $p$, and hence $\dim(\Omega_1(Z))=\dim(C_{V}(\UUU))=\rk(A)-p+1$. 

Let $Z_1,Z_2\le Z$ be such that $Z_1$ is a discrete $p$-torus, $Z_2$ is a 
finite abelian $p$-group, and $Z=Z_1\times Z_2$. Since $\rk(A/Z)=p-1$ by 
Lemma \ref{disc.torus2}, $\rk(Z_1)=\rk(A)-p+1=\dim(\Omega_1(Z))$. So 
$Z_2=1$, and $Z=Z_1$ is a discrete $p$-torus.
\end{proof}

\begin{Lem} \label{str.cl.}
Assume Notation \ref{n:not1}, and also that $p$ is odd, $|A|=\infty$, and 
$O_p(\calf)=1$. Then no proper nontrivial subgroup of $S$ is strongly 
closed in $\calf$. Thus $\calf$ is simple if and only if it contains no 
proper normal subsystem over $S$. 
\end{Lem}

\begin{proof} Assume that $1\ne{}Q\le S$ is strongly closed in $\calf$.  If 
$Q\le Z$, then $Q$ is contained in all $\calf$-essential subgroups, so 
$Q\nsg\calf$ by Proposition \ref{AFT-E}(c), contradicting the assumption 
that $O_p(\calf)=1$.  Thus $Q\nleq{}Z$.  

Now, $(QZ/Z)\cap Z(S/Z)\ne1$ since $Q\nsg{}S$, so $Q\cap{}Z_2\nleq Z$.  Fix 
$g\in(Q\cap Z_2)\sminus Z$.  Then $Q\ge[g,S]=Z_0$ since $Q\nsg S$.

Fix $P\in\EE\calf\sminus\{A\}\subseteq\calb\cup\calh$ (recall 
$A\nnsg\calf$). If $P\in\calb$, then the $\autf(P)$-orbit of $g\in(Q\cap 
Z_2)\sminus Z$ is not contained in $A$.  If $P\in\calh$, then the 
$\autf(P)$-orbit of $Z_0\le Q$ is not contained in $A$.  So in either case, 
$Q\nleq{}A$.  Hence $Q\ge[Q,S]\ge[\UUU,A]=S'$.  

Set $G_0=O^{p'}(G)$. Since $Q\cap{}A$ is normalized by the action of $G$, 
and contains $[\UUU,A]$ where $G_0$ is the normal closure of $\UUU$ in $G$, 
$Q\ge[G_0,A]$. Since $C_A(G_0)\le C_A(\UUU)=Z$ and $C_A(G_0)$ is normalized 
by $G$, $C_A(G_0)\le Z_0\le[\UUU,A]$ by Lemma \ref{O_p(F)=1}. So by Lemma 
\ref{l0:xx}, $[G_0,A]\ge C_A(\UUU)=Z$. Thus $Q\ge ZS'=A$, and so $Q=S$ 
since $Q\nleq A$.

The last statement is immediate.
\end{proof}

\begin{Lem} \label{l4x:s/a}
Assume Notation \ref{n:not1} and \ref{n:not2}, 
and also that $|A|=\infty$ and $S$ splits over $A$. 
\begin{enuma} 

\item The kernel of $\mu\:\autv(S)\too\Delta$ does not contain any elements 
of finite order prime to $p$.

\item Fix $Q\in\calb\cup\calh$, and set $t=0$ if $Q\in\calb$, $t=-1$ if 
$Q\in\calh$. Assume that $\mu(\autff(S))\ge\Delta_t$. Then there are unique 
subgroups $\til{Z}\le Z$ and $\til{Q}\ge Q\cap S'$ which are normalized by 
$N_{\autf(S)}(Q)$, and are such that the following hold.
\begin{enumi} 
\item If $Q\in\calh$, then $Q=\til{Z}\times \til{Q}$ and $\til{Q}\cong 
C_p\times C_p$. 

\item If $Q\in\calb$, then $\til{Z}=Z$, $Q=Z\til{Q}$, $Z\cap 
\til{Q}=Z_0=Z(\til{Q})$, and $\til{Q}$ is extraspecial of order $p^3$ and 
exponent $p$.

\item Thus in all cases, $\Out(\til{Q})\cong\GL_2(p)$. If 
$\alpha\in\Aut(Q)$ is such that $\alpha|_{\til{Z}}=\Id$, 
$\alpha(\til{Q})=\til{Q}$, $\alpha(Q\cap A)=Q\cap A$, and 
$\alpha|_{\til{Q}}\in O^{p'}(\Aut(\til{Q}))$, then $\alpha$ extends to some 
$\4\alpha\in N_{\autff(S)}(Q)$.

\end{enumi}
\end{enuma}
\end{Lem}

\begin{proof} \textbf{(a) } Fix $\alpha\in\Ker(\mu)$ of finite order prime 
to $p$. Then $\alpha$ induces the identity on $Z/Z_0$ since 
$\alpha\in\autff(S)$, and on $Z_0$ and $S/A$ since $\mu(\alpha)=(1,1)$. 
In particular, $\alpha|_Z=\Id$ by Lemma \ref{mod-Fr}, and $\alpha|_A$ is 
$\Z_p\UUU$-linear.

Fix $x\in S\sminus A$, and let $\psi\in\End(A)$ be the homomorphism 
$\psi(g)=[g,x]$. Then $\psi$ commutes with $\alpha|_A$ since $\alpha(x)\in 
xA$, and $\psi$ induces an injection from $Z_i(S)/Z_{i-1}(S)$ into 
$Z_{i-1}(S)/Z_{i-2}(S)$ for each $i\ge2$. Since $\alpha|_Z=\Id$, this shows 
that $\alpha$ induces the identity on $Z_i(S)/Z_{i-1}(S)$ for each $i$, and 
hence that $\alpha|_{Z_i(S)}=\Id$ for each $i$ by Lemma \ref{mod-Fr} again. 
Thus $\alpha|_{\Omega_1(A)}=\Id$ since $\Omega_1(A)\le Z_i(S)$ for some 
$i$, so $\alpha|_{\Omega_m(A)}=\Id$ for each $m\ge1$ by \cite[Theorem 
5.2.4]{Gorenstein}. So $\alpha|_A=\Id$, and $\alpha=\Id_S$ by Lemma 
\ref{mod-Fr} again.

\smallskip

\noindent\textbf{(b) } This proof is essentially the same as that of 
\cite[Lemma 2.6(b)]{indp2} (a similar result but with $|A|<\infty$). 
We sketch an alternative argument here. 

Set $K=N_{\autff(S)}(Q)$ for short. By the Frattini argument and since 
all members of the $\Aut(S)$-orbit of $Q$ are $S$-conjugate to $Q$ 
(Lemma \ref{A0=A}), $\autff(S)=\Inn(S)\cdot K$. Hence $\mu(K)\ge\Delta_t$. 
Also, $|N_{\Inn(S)}(Q)|\le|N_S(Q)/Z|<\infty$ since $|N_S(Q)/Q|=p$ by Lemma 
\ref{B&H}, and $N_{\Inn(S)}(Q)=\Aut_{N_S(Q)}(S)$ is normal of index prime 
to $p$ in $K$ since $\Inn(S)\nsg\autf(S)$ has finite index prime to $p$ by 
the Sylow axiom.

By the Schur-Zassenhaus theorem, there is $K_0<K$ of order prime to $p$ 
such that $K=K_0\cdot\Aut_{N_S(Q)}(S)$. Set $\til{Z}=C_Q(K_0)$ and 
$\til{Q}=[K_0,Q]$. 
\begin{itemize}

\item If $Q\in\calh$ and $\mu(K_0)=\mu(K)\ge\Delta_{-1}$, then $K_0$ acts 
nontrivially on $Q/Z$ and on $Z_0$, and trivially on $Z/Z_0$. Hence 
$\til{Z}\le Z$ and $\til{Q}\cap A=Z_0$. Also, $Q=\til{Z}\times\til{Q}$ and 
$Z=\til{Z}\times Z_0$ by \cite[Theorem 5.2.3]{Gorenstein} (applied to the 
subgroups $\Omega_m(Q)$ for $m\ge1$), and $\til{Q}\cong C_p\times C_p$ 
since $S$ splits over $A$. In particular, 
$\Out(\til{Q})\cong\GL_2(p)$.

If $\beta\in K_0$ is such that $\mu(\beta)$ generates $\Delta_{-1}$, 
then $\mu(\beta)=(r,r^{-1})$ for some generator $r$ of $(\Z/p)^\times$, so 
$\beta|_{\til{Q}}$ acts on $\til{Q}\cong C_p\times C_p$ as the matrix 
$\mxtwo{r}00{r^{-1}}$ for an appropriate choice of basis. Thus 
$\Aut_S(\til{Q})\gen{\beta|_{\til{Q}}}=N_{O^{p'}(\Aut(\til{Q}))}(Z_0)$ 
where $O^{p'}(\Aut(\til{Q}))\cong\SL_2(p)$, proving (iii) in this case.

\item If $Q\in\calb$ and $\mu(K_0)=\mu(K)\ge\Delta_0$, then $\alpha$ acts 
nontrivially on $Q/Z_2$, and trivially on $Z_0$ and $Z/Z_0$. Hence $\alpha$ 
acts trivially on $Z$ by Lemma \ref{mod-Fr}, nontrivially on $Z_2/Z$, and 
so $\til{Z}=Z$ and $\til{Q}\cap A=Z_2\cap S'$. Also, $\til{Q}$ is 
extraspecial of order $p^3$ and (since $S$ splits over $A$) exponent $p$. 
In particular, $\Out(\til{Q})\cong\GL_2(p)$. By \cite[Theorem 
5.2.3]{Gorenstein}, applied to the abelian $p$-groups $\Omega_m(Q/Z_0)$ for 
$m\ge1$, $Q=\til{Z}\til{Q}$ and $\til{Z}\cap\til{Q}=Z_0$. 

If $\beta\in K_0$ is such that $\mu(\beta)$ generates $\Delta_0$, 
then $\mu(\beta)=(r,1)$ for some generator $r$ of $(\Z/p)^\times$, so 
$[\beta|_{\til{Q}}]$ has order $(p{-}1)$ in 
$O^{p'}(\Aut(\til{Q}))/\Inn(\til{Q})\cong\SL_2(p)$. So 
$\Aut_S(\til{Q})\gen{\beta|_{\til{Q}}}=N_{O^{p'}(\Aut(\til{Q}))}(Q\cap A)$ 
in this case, again proving (iii).

\end{itemize}

Since $\til{Z}\le Z\le C_Q(N_S(Q))$ and $\til{Q}\ge Q\cap S'\ge[N_S(Q),Q]$ 
in all cases, we have $\til{Z}=C_Q(K)$ and $\til{Q}=[K,Q]$. Thus $\til{Z}$ 
and $\til{Q}$ are independent of the choice of $K_0$, and are normalized by 
$K=N_{\autff(S)}(Q)$. Since $\autff(S)$ is normal in $\autf(S)$ (the kernel 
of a homomorphism to $\Aut(Z/Z_0)$), we see that $K$ is normal in 
$N_{\autf(S)}(Q)$, and hence $\til{Z}$ and $\til{Q}$ are also normalized by 
$N_{\autf(S)}(Q)$. These are easily seen to be the unique subgroups that 
satisfy the required conditions.
\end{proof}

When $\calf$ is a fusion system over a discrete $p$-toral group $S$, 
then for each $Q\le S$, we define another subgroup $Q^\bullet\le S$ as 
follows. Let $T$ be the identity component of $S$. If $m\ge0$ is the 
smallest integer such that $g^{p^m}\in T$ for each $g\in S$, and 
$Q^{[m]}=\mho^m(Q)=\gen{g^{p^m}\,|\,g\in Q}$, then $Q^\bullet\defeq Q\cdot 
I(Q^{[m]})_0$, where $I(Q^{[m]})=C_{T}(C_{\autf(T)}(Q^{[m]}))$ and 
$I(Q^{[m]})_0$ is its identity component. Thus $Q\le Q^\bullet\le QT$ for 
each $Q$. See \cite[Definition 3.1]{BLO3} or \cite[Definition 3.1]{BLO6} 
for more detail, as well as the motivation for this construction.

\begin{Lem} \label{Q.=Q}
Assume the notation and hypotheses of \ref{n:not1}, and also that $|A|=\infty$. For each 
$Q\in\calb\cup\calh\cup\{A,S\}$, $Q^\bullet=Q$.
\end{Lem}

\begin{proof} By Proposition \ref{A-torus}, $A$ is the identity component 
of $S$. Thus $A$ and $G$ play the role of $T$ and $W=\autf(T)$ in 
\cite[Definition 3.1]{BLO6}. Since $Q\le Q^\bullet\le QA$ for 
each $Q\le S$, we have $A^\bullet=A$ and $S^\bullet=S$.

Now assume $Q\in\calh\cup\calb$. By assumption, $S/A$ has exponent $p=p^1$. 
Since $Q/Z\cong C_p$ or $C_p^2$ (Lemma \ref{l2:s/a}), 
$Q^{[1]}\defeq\gen{g^p\,|\,g\in Q}\le Z$. Hence $C_G(Q^{[1]})\ge\UUU$, and 
$I(Q^{[1]})\defeq C_A(C_G(Q^{[1]}))\le C_A(\UUU)=Z$. It follows that 
$Q^\bullet\le Q\cdot I(Q^{[1]})=Q$. 
\end{proof}

We are now ready to prove our main theorem used to construct simple fusion 
systems over infinite discrete $p$-toral groups with abelian subgroup of 
index $p$.

\begin{Thm} \label{t:s/a}
Fix an odd prime $p$. Let $S$ be an infinite nonabelian discrete $p$-toral 
group which contains an abelian subgroup $A\nsg{}S$ of index $p$, and let 
$\calf$ be a simple fusion system over $S$. Assume Notation 
\ref{n:not1} and \ref{n:not2} (where the uniqueness of $A$ 
follows from Corollary \ref{>1abel}). Then the following hold:
\begin{enuma}  

\item $\UUU\in\sylp{G}$ and $S$ splits over $A$.

\item $A$ is a discrete $p$-torus of rank at least $p-1$,
$Z_0=C_A(\UUU)\cap[\UUU,A]$ has order $p$, and 
$A=C_A(\UUU)\cdot[\UUU,A]$. 

\item There are no non-trivial $G$-invariant subgroups of $Z=C_A(\UUU)$, 
aside (possibly) from $Z_0$.  

\item Either 
\begin{enumi} 

\item $\EE\calf\sminus\{A\}=\calh$, $\rk(A)=p-1$, 
$\mu_A(\autff(A))\ge\Delta_{-1}$, and 
$G=O^{p'}(G)\cdot\mu_A^{-1}(\Delta_{-1})$; or 

\item $\EE\calf\sminus\{A\}=\calb$, $\rk(A)\ge p-1$, 
$\mu_A(\autff(A))\ge\Delta_{0}$, $\mu_A(\autff(A))=\Delta_{0}$ if 
$\rk(A)\ge p$, $G=O^{p'}(G)\cdot\mu_A^{-1}(\Delta_{0})$, and $Z_0$ is not 
$G$-invariant.

\end{enumi}
Here, we regard $\mu_A$ as a homomorphism defined on $\autff(A)$.
\end{enuma}

Conversely, let $S$ be an infinite discrete $p$-toral group containing a 
unique abelian subgroup $A\nsg S$ of index $p$, let $G\le\Aut(A)$ be 
such that $\Aut_S(A)\in\sylp{G}$, and adopt the notation in 
\ref{n:not1} and \ref{n:not2}. Assume that (a)--(d) hold, with $\autff(A)$ 
replaced by $G\cap\autv(A)$ and $\EE\calf\sminus\{A\}$ replaced by 
$\EE0=\calh$ or $\calb$ in (d). Then there is a unique simple fusion system 
$\calf$ over $S$ such that $G=\autf(A)$ and $\EE\calf\sminus\{A\}=\EE0$. 
\end{Thm}

\begin{proof} We prove in Steps 1 and 3 that conditions (a)--(d) are 
necessary, and prove the converse in Step 2. 

\smallskip

\noindent\textbf{Step 1: } Assume that $\calf$ is a simple fusion system 
over $S$. We must show that conditions (a)--(d) hold. By Corollary 
\ref{>1abel}, $A$ is the unique abelian subgroup of index $p$ in $S$.

\smallskip

\noindent\textbf{(a,b,c) } Point (a) holds by Corollary 
\ref{c2:s/a} and since $A$ is fully automized. The last two statements in 
(b) hold by Lemmas \ref{|Z0|=p} and \ref{A0=A}, and (c) holds by Lemma 
\ref{O_p(F)=1} and since $O_p(\calf)=1$. Finally, $A$ is a discrete 
$p$-torus by Proposition \ref{A-torus}, and $\rk(A)\ge p-1$ by Lemma 
\ref{disc.torus2}.

\smallskip

\noindent\textbf{(d) } Since $A\nnsg\calf$, there is 
$P\in\EE\calf\cap(\calh\cup\calb)$. Set $t=0$ if $P\in\calb$, and $t=-1$ if 
$P\in\calh$.

Set $H=\autf(P)$ and $H_0=O^{p'}(H)$. By Lemma \ref{l2:s/a}, 
$H_0/\Inn(P)\cong\SL_2(p)$, and acts trivially on $Z/Z_0$. Since 
$\Aut_S(P)\in\sylp{H}=\sylp{H_0}$, we can choose $\alpha\in 
N_{H_0}(\Aut_S(P))$ of order $p-1$ in $H/\Inn(P)$. By the 
extension axiom, $\alpha$ extends to an element of $\autf(N_S(P))$, and 
since $P$ is maximal among $\calf$-essential subgroups by Lemmas \ref{B&H} 
and \ref{exclusion}, $\alpha=\5\alpha|_P$ for some $\5\alpha\in\autf(S)$. 
Set $\alpha_0=\5\alpha|_A\in G$.

Now, $\alpha_0$ induces the identity on $Z/Z_0$ since $\alpha$ does, and 
$\alpha_0\in N_G(\UUU)$ since it extends to $S$. Thus 
$\alpha_0\in\autff(A)$. By Lemma \ref{l2:s/a}, $\alpha$ acts as an element 
of $\SL_2(p)$ on $P/Z\cong C_p^2$ (if $P\in\calb$) or on $P/P_1\cong C_p^2$ 
where $Z=P_1\times Z_0$ (if $P\in\calh$). Hence for some $s\in(\Z/p)^\times$ 
of order $p-1$, $\mu_A(\alpha_0)=(s,s^{-1})$ if $P\in\calh$, or 
$\mu_A(\alpha_0)=(s,1)$ if $P\in\calb$. Thus $\mu_A(\autff(A))\ge\Delta_t$ 
in either case.

Assume $\rk(A)\ge p$. By Proposition \ref{A-torus}, $Z$ is a discrete 
$p$-torus. Hence each element of $\autff(A)$ induces the identity on $Z_0$ 
since it induces the identity on $Z/Z_0$, and 
$\mu_A(\autff(A))\le\Delta_0$.

Set $G_0=O^{p'}(G)\cdot\mu_A^{-1}(\Delta_t)$. Since 
$\Ker(\mu_A|_{\autff(A)})=\UUU$ by Lemma \ref{l4x:s/a}(a), and since 
$\mu_A(\autff(A))\ge\Delta_t$ by assumption, we have 
$G\ge\mu_A^{-1}(\Delta_t)$, and hence $G\ge G_0$. We will show in Step 3 
(with the help of the constructions in Step 2) that $G=G_0$, thus finishing 
the proof of (d). 

\smallskip

\noindent\textbf{Step 2: } Now assume that $S$, $A$, and $G$ are as above, 
and set $G^\vee=G\cap\autv(A)$. Assume that (a) and (b) hold, and 
also that $\mu_A(G^\vee)\ge\Delta_t$ for some $t\in\{0,-1\}$. (Note that 
$G^\vee\le N_G(\UUU)$ since each element is the restriction of an 
automorphism of $S$.) We must show that these are realized by a unique 
saturated fusion system $\calf$, which is simple if (c) and (d) hold. Set 
$\EE0=\calh$ if $t=-1$, or $\EE0=\calb$ if $t=0$.

Set $\Gamma=A\sd{}G$, and identify $S=A\sd{}\UUU\in\sylp{\Gamma}$. Choose a 
generator $\xx\in\UUU<S$. Set $Z=Z(S)$, $Z_2=Z_2(S)$, as in Notation 
\ref{n:not1}.

Set $Q=Z\gen{\xx}$ if $\EE0=\calh$, or $Q=Z_2\gen{\xx}$ if $\EE0=\calb$. 
Thus $Q\in\EE0$, and each member of $\EE0$ is $S$-conjugate to $Q$ by Lemma 
\ref{A0=A}. Set $K=\Aut_\Gamma(Q)$. By assumption, there is $\alpha\in 
N_{G^\vee}(\UUU)$ such that $\mu_A(\alpha)$ generates $\Delta_t$, and 
$\alpha$ extends to some $\4\alpha\in\autv(S)$ such that 
$\4\alpha(\xx)\in\UUU$. In particular, $\4\alpha(Q)=Q$, and 
$\4\alpha\in\Aut_\Gamma(S)$. 

By Lemma \ref{l4x:s/a}(b), applied with $\calf_S(\Gamma)$ in the role of 
$\calf$, there are unique subgroups $\til{Z}\le Z$ and $\til{Q}\ge Q\cap 
S'$, both normalized by $N_{\Aut_\Gamma(S)}(Q)$, and such that 
$Q=\til{Z}\times\til{Q}$ and $\til{Q}\cong C_p\times C_p$ if $Q\in\calh$; 
or $\til{Z}=Z$, $Q=Z\til{Q}$, $Z\cap\til{Q}=Z_0$, and $\til{Q}$ is 
extraspecial of order $p^3$ and exponent $p$ if $Q\in\calb$. Let 
$\Theta\le\Aut(Q)$ be the unique subgroup containing $\Inn(Q)$ that acts 
trivially on $\til{Z}$, normalizes $\til{Q}$, and is such that 
$\Theta/\Inn(Q)\cong\SL_2(p)$.

We next claim that
\begin{enumerate}[(1) ]

\item each $\alpha\in\Aut_\Gamma(Q)$ extends to some 
$\4\alpha\in N_{\Aut_\Gamma(S)}(Q)$; 

\item $\Aut_\Gamma(Q)$ normalizes $\Theta$; 

\item $\Aut_S(Q)\in\sylp{\Theta}=\sylp{\Theta\Aut_\Gamma(Q)}$; and 

\item $N_\Theta(\Aut_S(Q))\le\Aut_\Gamma(Q)$.

\end{enumerate}
Point (1) holds since $S=QA$ where $A\nsg\Gamma$, and hence $N_\Gamma(Q)\le 
N_\Gamma(S)$. By assumption, each element of $N_{\Aut_\Gamma(S)}(Q)$ 
normalizes $\til{Z}$ and $\til{Q}$, and hence normalizes $\Theta$. Thus (2) 
follows from (1). Since $|\Out_S(Q)|=|N_S(Q)/Q|=p$ by Lemma \ref{B&H}, this 
acts trivially on $Z\ge\til{Z}$ and normalizes $\til{Q}$, and 
$\Out(\til{Q})\cong\GL_2(p)$ where $\sylp{\GL_2(p)}=\sylp{\SL_2(p)}$, we 
see that $\Out_S(Q)\in\sylp{\Theta/\Inn(Q)}$ and hence that 
$\Aut_S(Q)\in\sylp\Theta$. Also, $\Theta$ has index prime to $p$ in 
$\Theta\Aut_\Gamma(Q)$ since $\Aut_S(Q)\in\sylp{\Aut_\Gamma(Q)}$, and this 
proves (3). Finally, (4) follows from Lemma \ref{l4x:s/a}(b.iii).

Set $\calf=\gen{\calf_S(\Gamma),\Theta}$: the smallest fusion system over 
$S$ which contains $\calf_S(\Gamma)$ and such that $\autf(Q)\ge\Theta$. Set 
$\calk=\{S,A\}\cup\EE0$. Then $\calk$ is invariant under $\calf$-conjugacy, 
and is closed in the space of all subgroups of $S$ \cite[Definition 
1.11]{BLO6}. Thus condition (i) in \cite[Theorem 4.2]{BLO6} holds 
for $\calk$; and condition (iii) holds ($P\in\calk$ and $P\le Q\le 
P^\bullet$ imply $Q\in\calk$) since $P=P^\bullet$ for each $P\in\calk$ 
(Lemma \ref{Q.=Q}). 

By Lemma \ref{exclusion}, if $\EE0=\calb$, then the members of $\calh$ 
are not $\calf$-centric. So in all cases, if $P\le S$ is 
$\calf$-centric and $P\notin\calk$, then $P$ is not contained in any member 
of $\EE\calf=\EE0\cup\{A\}$, and hence $\Out_S(P)\nsg\outf(P)$. This proves 
condition (iv) in \cite[Theorem 4.2]{BLO6}: 
	\[ O_p(\outf(P))\cap\Out_S(P)\ne1 \]
whenever $P\le S$ is $\calf$-centric and not in $\calk$.

We refer to \cite[Definition 1.11]{BLO6} for the definitions of 
``$\calk$-generated'' and ``$\calk$-saturated''. By construction, $\calf$ 
is $\calk$-generated. To show that $\calf$ is $\calk$-saturated, we 
must prove that each $P\in\calk$ is fully automized and receptive in 
$\calf$ (Definition \ref{d:saturated}). If $P=A$ or $P=S$, then 
$\autf(P)=\Aut_\Gamma(P)$, and this is easily checked. So it remains to 
show this when $P=Q$. By (2), $\autf(Q)=\Theta\cdot\Aut_\Gamma(Q)$. So $Q$ 
is fully automized by (3). If $\alpha\in N_{\autf(Q)}(\Aut_S(Q))$, then 
$\alpha\in\Aut_\Gamma(Q)$ by (4), and hence extends to some 
$\4\alpha\in\Aut_\Gamma(S)$ by (1). So $Q$ is also receptive. This finishes 
the proof of condition (ii) in \cite[Theorem 4.2]{BLO6}, and hence $\calf$ 
is saturated by that theorem.

Now assume (c) and (d) hold; we must prove that $\calf$ is simple. By (c), 
there are no non-trivial $G$-invariant subgroups of $Z$ except possibly for 
$Z_0$. Also, $\EE\calf\supseteq\calh$ in case (d.i), and $Z_0$ 
is not $G$-invariant in case (d.ii). Hence $O_p(\calf)=1$ by Lemma 
\ref{O_p(F)=1}. By Lemma \ref{str.cl.}, $\calf$ is simple if there are no 
proper normal fusion subsystems in $\calf$ over $S$. 

Assume $\calf_0\nsg\calf$ is a normal fusion subsystem over $S$, and set 
$G_0=\Aut_{\calf_0}(A)$. Then $G_0\nsg G$, and $G_0\ge O^{p'}(G)$ since it 
is the normal closure of $\UUU=\Aut_S(A)$. Also, 
$\mu_A(\autv_{\calf_0}(A))\ge\Delta_t$ by Step 1, applied with $\calf_0$ 
in the role of $\calf$. Since $\mu_A$ is injective on $G^\vee/\Aut_S(A)$ 
by Lemma \ref{l4x:s/a}(a), we have 
$G_0\ge O^{p'}(G)\cdot\mu_A^{-1}(\Delta_t)=G$. Thus $G_0=G$, and 
$\Aut_{\calf_0}(S)=\autf(S)$ by the extension axiom, so $\calf_0=\calf$ by 
the Frattini condition on a normal subsystem (see Definition \ref{d:E<|F}). 
This finishes the proof that $\calf$ is simple.

The uniqueness of $\calf$ follows from the uniqueness of $\til{Z}$ and 
$\til{Q}$ in Lemma \ref{l4x:s/a}(b). 

\smallskip

\noindent\textbf{Step 3: } We return to the situation of Step 1, where it 
remains only to prove that $G_0=G$. By Step 2, there is a unique saturated 
fusion subsystem $\calf_0\le\calf$ over $S$ such that 
$\EE{\calf_0}=\EE\calf$ and $\Aut_{\calf_0}(A)=G_0$. The invariance 
condition on $\calf_0\le\calf$ (Definition \ref{d:E<|F}) holds by the 
uniqueness of $\calf_0$, and the Frattini condition holds since 
$G=O^{p'}(G)N_G(\UUU)\le G_0N_G(\UUU)$ (where each element of $N_G(\UUU)$ 
extends to an element of $\autf(S)$ by the extension axiom). Thus 
$\calf_0\nsg\calf$, so $\calf_0=\calf$ since $\calf$ is simple, and hence 
$G_0=G$. 
\end{proof}

As a special case, we next show that for each prime $p$, there is (up to 
isomorphism) a unique simple fusion system over an infinite discrete 
$p$-toral group with abelian subgroup of index $p$ which is not essential.

\begin{Thm} \label{t:AnotinE}
For each odd prime $p$, there is, up to isomorphism, a unique 
simple fusion system $\calf$ over an infinite nonabelian discrete $p$-toral 
group $S$ which contains an abelian subgroup $A<S$ of index $p$ that 
is not $\calf$-essential. The following hold for each such $p$, 
$\calf$, $S$, and $A$:
\begin{enuma} 

\item The group $A$ is a discrete $p$-torus of rank $p-1$, and $S$ splits 
over $A$. Also, $\autf(A)\cong C_p\rtimes C_{p-1}$, $\outf(S)\cong 
C_{p-1}$, and $\EE\calf=\calh$ (defined as in Notation \ref{n:not1}).

\item Fix a prime $q\ne p$, set $\Gamma=\PSL_p(\4\F_q)$, let 
$\til{A}<\Gamma$ be the subgroup of classes of diagonal matrices of 
$p$-power order, and set $\til{S}=\til{A}\gen{\til{x}}\in\sylp{\Gamma}$ for 
some permutation matrix $\til{x}$ of order $p$. Then there is an 
isomorphism $S\cong \til{S}$ that restricts to an isomorphism 
$A\cong\til{A}$, and induces isomorphisms 
$\autf(S)\cong\Aut_{\Gamma}(\til{S})$ and 
$\autf(A)\cong\Aut_{N_{\Gamma}(\til{S})}(\til{A})$.

\item If $p=3$, then $\calf$ is isomorphic to the $3$-fusion system of 
$\PSU(3)$, and also to the $3$-fusion system of $\PSL_3(\4\F_q)$ for each 
prime $q\ne3$. For $p\ge5$, $\calf$ is not realized by any compact Lie 
group, nor by any $p$-compact group.

\end{enuma}
\end{Thm}

\begin{proof} We use the notation of \ref{n:not1} and \ref{n:not2}. In 
particular, $G=\autf(A)$. 

\smallskip

\noindent\textbf{(a,b) } Assume $\calf$ is a simple fusion system over an 
infinite discrete $p$-toral group $S$ with an abelian subgroup $A<S$ of 
index $p$ such that $A\notin\EE\calf$. By Lemma \ref{Z0=Z}, $Z=Z_0$, 
$\EE\calf=\calh$, and $A=S'$ is a discrete $p$-torus of rank $p-1$. In 
particular, $|Z|=|Z_0|=p$. Also, $S$ splits over $A$ by Corollary 
\ref{c2:s/a}. 

It remains to describe $G=\autf(A)$ and $\autf(S)$ and prove (b). Since 
$A\notin\EE\calf$, $\UUU\nsg G$, and $\outf(S)\cong G/\UUU$. (Each 
$\alpha\in\autf(A)$ extends to $\autf(S)$ by the extension axiom.) 

Since $S$ splits over $A$, each $\alpha\in N_{\Aut(A)}(\UUU)$ extends to an 
automorphism of $S$. Since $Z=Z_0$, this implies that 
$N_{\Aut(A)}(\UUU)=\autv(A)$. Also, $\mu_A(G)=\Delta_{-1}$ 
by Theorem \ref{t:s/a}(d) and since $O^{p'}(G)=\UUU\le\Ker(\mu_A)$. 

Let $R=\Z_p[\zeta]$ and $\pp=(1-\zeta)R$ be as in Notation \ref{n:not3}, 
regarded as $\Z_p\UUU$-modules. By Proposition \ref{tori&lattices}, 
$A\cong(\Q_p/\Z_p)\otimes_{\Z_p}\Lambda$ and 
$\Lambda\cong\Hom_{\Z_p}(\Q_p/\Z_p,A)$ for some $(p-1)$-dimensional 
$\Z_pG$-lattice $\Lambda$, and $\Lambda|_\UUU\cong R$ as $\Z_p\UUU$-modules by 
Lemma \ref{ZU-modules}(c). These isomorphisms induce isomorphisms of 
automorphism groups
	\begin{align*} 
	\autv(A) = N_{\Aut(A)}(\UUU) &\cong N_{\Aut(R)}(\UUU) 
	\cong C_{\Aut(R)}(\UUU)\rtimes\Gal(\Q_p(\zeta)/\Q_p) \\
	&\cong R^\times\rtimes\Gal(\Q_p(\zeta)/\Q_p) 
	\cong \bigl((1+\pp)\times\F_p^\times\bigr)\rtimes
	\Gal(\Q_p(\zeta)/\Q_p) ,
	\end{align*}
and these send $\Ker(\mu_A)\le\autv(A)$ (the group of automorphisms 
of $A$ that commute with $\UUU$ and are the identity on $Z=Z_0$) onto 
$1+\pp$, and send the subgroup $\scal(\F_p^\times)\le\autv(A)$ of scalar 
multiplication by $(p-1)$-st roots of unity onto $\F_p^\times$. Thus 
	\[ \autv(A) = N_{\Aut(A)}(\UUU) 
	= \bigl( \Ker(\mu_A) \times \scal(\F_p^\times) \bigr) \rtimes W \]
for a certain subgroup $W\cong C_{p-1}$. 
Set $\4G=(\UUU\times\scal(\F_p^\times))\cdot W < \autv(A)$: a subgroup of 
order $p(p-1)^2$.

Since $\autv(A)/\Ker(\mu_A)$ has order $(p-1)^2$, and since 
$\Ker(\mu_A)\cong(1+\pp)$ is an abelian pro-$p$-group and hence uniquely 
$m$-divisible for each $m$ prime to $p$, we have $H^i(H;\Ker(\mu_A))=0$ for 
each $H\le \autv(A)/\Ker(\mu_A)$ and each $i>0$. Hence for each subgroup 
$K<\autv(A)$ of order prime to $p$, $K\cap\Ker(\mu_A)=1$ since 
$\Ker(\mu_A)$ is a pro-$p$-group, $\Ker(\mu_A)\cdot K$ splits over 
$\Ker(\mu_A)$ with a splitting unique up to conjugacy, and thus $K$ is 
conjugate by an element of $\Ker(\mu_A)$ to a subgroup of 
$\scal(\F_p^\times)\cdot W$. In particular, $G$ is conjugate to a subgroup 
of $\4G$, and we can assume (without changing the isomorphism type of 
$\calf$) that $G\le\4G$. Finally, one easily sees that $\mu_A$ sends $\4G$ 
onto $\Delta$ with kernel $\UUU$, and hence that 
$G=(\mu_A|_{\4G})^{-1}(\Delta_{-1})\cong C_p\rtimes C_{p-1}$ is uniquely 
determined. 

A natural isomorphism $A\cong\til{A}$ is most easily seen by identifying 
$\Gamma=\PGL_p(\4\F_q)$, so that $\til{A}$ is the quotient of 
$(\Z/p^\infty)^p$ by the diagonal $\Z/p^\infty\cong O_p(Z(\Gamma))$,  with 
the permutation action of $\Aut_\Gamma(\til{A})\cong\Sigma_p$. This then 
extends to an isomorphism of $S\cong A\rtimes\UUU$ with 
$\til{S}=\til{A}\gen{\til{x}}$, and of $\autf(A)$ with 
$\Aut_{N_\Gamma(\til{S})}(\til{A})$.

\smallskip

\noindent\boldd{Existence and uniqueness of $\calf$: } Let $A$, $\UUU\nsg 
G$, and $S=A\rtimes\gen{x}$ be as described in the proof of (b). Since 
$\mu_A(G)=\Delta_{-1}$, conditions (a)--(d) in Theorem \ref{t:s/a} all hold 
with $\EE\calf=\calh$. So such an $\calf$ exists by that theorem. It is 
unique up to isomorphism by the uniqueness in the theorem and by the 
restrictions shown in the proof of (b).

\smallskip

\noindent\textbf{(c) } If $\calf$ is realized by a compact Lie group 
or a $p$-compact group, then by Proposition \ref{G&F} and since all 
elements in $S$ are $\calf$-conjugate to elements in $A$, $\calf$ is 
realized by a connected, simple $p$-compact group, and the action of the 
Weyl group $\autf(A)$ on $\Q\otimes_{\Z}\Hom(A,\Q_p/\Z_p)$ is generated by 
pseudoreflections. But if $p\ge5$, then $\autf(A)\cong C_p\rtimes C_{p-1}$ 
contains no pseudoreflections other than the identity. So $p=3$, and we 
easily check that $\calf$ is realized by $\PSU(3)$, or by $\PSL_3(\4\F_q)$ 
for $q\ne3$. 
\end{proof}

We can now describe the simple fusion systems over discrete 
$p$-toral groups with discrete $p$-torus of index $p$ in terms of the 
classification of certain faithful, minimally active, indecomposable 
modules carried out in \cite{indp2}.

\begin{Th} \label{ThB}
Fix an odd prime $p$.
\begin{enuma} 

\item Let $\calf$ be a simple fusion system over an infinite 
nonabelian discrete $p$-toral group $S$ with an abelian subgroup 
$A<S$ of index $p$. Assume also that $A$ is $\calf$-essential. Set 
$G=\autf(A)$ and $V=\Omega_1(A)$, let $\calh$ and $\calb$ be as in Notation 
\ref{n:not1}, and let $G^\vee=\autff(A)$ and 
$\mu_A\:G^\vee\too\Delta$ be as in Notation \ref{n:not2}. Then $A$ is a 
discrete $p$-torus, $S$ splits over $A$, $G\in\GGG$, and for 
some $t\in\{0,-1\}$,
	\beq \parbox{\shorter}{$V$ is a faithful, minimally active, 
	indecomposable $\F_pG$-module. Either $\dim(V)=p-1$, 
	$\mu_A(G^\vee)\ge\Delta_t$ and $G=O^{p'}(G)\mu_A^{-1}(\Delta_t)$; 
	or $\dim(V)\ge p$, $t=0$, $\mu_A(G^\vee)=\Delta_0$, and 
	$G=O^{p'}(G)\cdot G^\vee$. Also, $\EE\calf=\{A\}\cup\calh$ if 
	$t=-1$, while $\EE\calf=\{A\}\cup\calb$ if $t=0$. If $t=0$, then 
	$V$ contains no $1$-dimensional $\F_pG$-submodule.} 
	\tag*{\starinf} \eeq

\item Conversely, assume that $G\in\GGG$, $\UUU\in\sylp{G}$, and 
$t\in\{0,-1\}$, and that $V$ is an $\F_pG$-module that satisfies \starinf, 
where $G^\vee$ is the subgroup of all elements $\alpha\in N_G(\UUU)$ such 
that $[\alpha,C_V(\UUU)]\le[\UUU,V]$. Then there are a discrete 
$G$-$p$-torus $A$ and a simple fusion system $\calf$ over $S=A\rtimes\UUU$ 
such that $\autf(A)=\Aut_G(A)\cong G$, such that $\Omega_1(A)\cong V$ as 
$\F_pG$-modules, and such that $\EE\calf=\{A\}\cup\calh$ if $t=-1$, or 
$\EE\calf=\{A\}\cup\calb$ if $t=0$. Furthermore, any other simple fusion 
system with these properties is isomorphic to $\calf$.

\item Among the fusion systems specified in (b), the only ones that 
are realized as fusion systems of compact Lie groups or of $p$-compact 
groups are those listed in Table \ref{tbl:inf-real}.

\end{enuma}
\end{Th}

\begin{table}[ht]
\[ \renewcommand{\arraystretch}{1.5}
\begin{array}{|c|c|c|c|c||c|c|} \hline
p & \textup{conditions} & \rk(A) & 
G=\autf(A) & \EE0 & \textup{$p$-cpct. gp.} & \textup{tors. lin. gp.} 
\\ \hline\hline
p & p\ge5 & p{-}1 & \Sigma_p & \calh & \PSU(p) & \PSL_p(\4\F_q) \\ \hline
p & p<n<2p & n{-}1 & 
\Sigma_n & \calb & \PSU(n) & \PSL_n(\4\F_q) \\ \hline
p & \mbox{\Small{$p\le n<2p,~n\ge4$}} & n & C_2^{n-1}\rtimes\Sigma_n & \calb & \PSO(2n) & 
P\Omega_{2n}(\4\F_q) \\ \hline
p & \mbox{\Small{$\dbl{2<m\mid(p-1)}{p\le n<2p,~n\ge4}$}} & n & 
(C_m)^{n-1}\rtimes\Sigma_n & \calb & X(m,m,n) &  \\\hline
5 & n=6,7 & n & W(E_n) & \calb & 
E_n & E_n(\4\F_q) \\ \hline
7 & n=7,8 & n & W(E_n) & \calb & 
E_n & E_n(\4\F_q) \\ \hline
3 &  & 2 
& GL_2(3) & \calb & X_{12} & C_{F_4(K)}(\gamma) \\ \hline 
5 &  & 4 & 
(4\circ2^{1+4}).\Sigma_5 & \calb & X_{29} &  \\ \hline
5 &  & 4 & 
(4\circ2^{1+4}).\Sigma_6 & \calb & X_{31} & E_8(K) \\ \hline
7 &  & 6 & 6_1\cdot\PSU_4(3).2_2 & 
\calb & X_{34} &  \\\hline
\hline
\end{array}
\]
\caption{The sixth column lists a compact Lie group or a $p$-compact 
group that realizes the fusion system $\calf$ described in the first 
five columns. Here, $X(m,m,n)$ 
denotes the $p$-compact group with Weyl group $G(m,m,n)$ in the 
notation of \cite[\S\,2]{ST}, and $X_k$ the one with Weyl group 
number $k$ in \cite[Table VII]{ST}. In the last column, we 
give, in some cases, a torsion linear group that realizes $\calf$: 
$q\ne p$ is prime, $K\subseteq\overline\F_2$ 
is the union of the odd degree extensions of $\F_2$, and 
$\gamma\in\Aut(F_4(K))$ is a graph automorphism of order $2$. In the 
fourth column, 
$B.C$ means an extension of $B$ by $C$, and the subscripts in the 
entry $6_1\cdot\PSU_4(3).2_2$ are \emph{Atlas} notation \cite[p. 52]{Atlas}.} 
\label{tbl:inf-real}
\end{table}

\begin{proof} \textbf{(a) } Set $\UUU=\Aut_S(A)\in\sylp{G}$ and $Z=C_A(\UUU)$.

Under the above assumptions, $A$ is a discrete $p$-torus by Proposition 
\ref{A-torus}, $G\in\GGG$ by Lemma \ref{min.act.}(a), $V$ is faithful, 
minimally active, and indecomposable by Lemma \ref{min.act.}(b), and 
$\rk(V)=\rk(A)\ge p-1$ by Lemma \ref{disc.torus2}. Also, for some 
$t\in\{0,-1\}$, $\mu_A(G^\vee)\ge\Delta_t$ and 
$G=O^{p'}(G)\mu_A^{-1}(\Delta_t)$, and $\EE\calf$ is as described in 
\starinf, by Theorem \ref{t:s/a}(d). Since 
$A\nnsg\calf$, $S$ splits over $A$ by Corollary 
\ref{c2:s/a}. If $t=0$ and $V_0<V$ is a 1-dimensional $\F_pG$-submodule, 
then $V_0\le C_V(\UUU)\le Z$, which is impossible by Theorem 
\ref{t:s/a}(c,d.ii).

Assume $\rk(V)=\rk(A)\ge p$. Since $V$ is minimally active and 
indecomposable, $V|_\UUU$ is the direct sum of a free module $\F_p\UUU$ and 
an $\F_p$-vector space with trivial $\UUU$-action by Lemma 
\ref{min.act.props}, and hence $\Omega_1(Z)=C_V(\UUU)$ has rank 
$\rk(A)-p+1$. Also, $Z$ is a discrete $p$-torus by Proposition 
\ref{A-torus}, so for each $\alpha\in G^\vee=\autff(A)$, $\alpha$ acts via 
the identity on $Z_0$ since it acts via the identity on $Z/Z_0$ 
(see Notation \ref{n:not2}). Thus $\mu_A(\alpha)\in\Delta_0$ by 
definition of $\mu_A$, so $t=0$ and $\mu_A(G^\vee)=\Delta_0$ in this case, 
finishing the proof of \starinf.

\smallskip
 
\noindent\textbf{(b) } Fix $G\in\GGG$, $\UUU\in\sylp{G}$, and 
$t\in\{0,-1\}$, and let $V$ be an $\F_pG$-module that satisfies \starinf, 
where $G^\vee$ is the subgroup of all elements $g\in N_G(\UUU)$ such that 
$[g,C_V(\UUU)]\le[\UUU,V]$. By Proposition \ref{dim(p-1)}(a), there is a 
$\Z_pG$-lattice $\Lambda$ such that $\Lambda/p\Lambda\cong V$ as 
$\F_pG$-modules. Set $A=(\Q_p/\Z_p)\otimes_{\Z_p}\Lambda$: a discrete 
$G$-$p$-torus where $\Omega_1(A)\cong V$ as $\F_pG$-modules 
(see Proposition \ref{tori&lattices}). To simplify notation, we identify 
$V=\Omega_1(A)$. Set $S=A\rtimes\UUU$. Set $Z=Z(S)=C_A(\UUU)$, 
$S'=[S,S]=[\UUU,A]$, and $Z_0=Z\cap S'$. 

We next check that conditions (a)--(d) in Theorem \ref{t:s/a} all hold. 
Conditions (a) and (d) follow immediately from \starinf, and (b) 
($|Z_0|=p$) was shown in Lemma \ref{min.act.|Z0|=p}(b).

Assume $1\ne B\le Z$ is $G$-invariant. If $\rk(A)=p-1$, then $Z=Z_0$ 
has order $p$, so $B=Z_0$. Otherwise, by \starinf, $t=0$, and $V$ contains 
no $1$-dimensional $\F_pG$-submodule. Thus $\dim(\Omega_1(B))\ge2$, so 
$\dim(V)=\rk(A)\ge p+1$. If $\dim(V)\ge p+2$, then $V$ is simple by 
\cite[Proposition 3.7(c)]{indp2}, while if $\dim(V)=p+1$, then $V$ contains 
no nontrivial $\F_pG$-submodule with trivial $\UUU$-action by Lemma 
\ref{rk(V)=p+1}. Thus $B=Z_0$, and this proves condition \ref{t:s/a}(c).

By Theorem \ref{t:s/a}, there is a unique simple fusion system $\calf$ over 
$S$ such that $G=\autf(A)$, and $\EE\calf\sminus\{A\}=\calh$ (if $t=-1$) or 
$\calb$ (if $t=0$). Since $A$ is unique (up to isomorphism of 
$\Z_pG$-modules) by Lemma \ref{unique-G-p-torus}, this shows that $\calf$ 
is uniquely determined by $V$. 

\smallskip

\noindent\textbf{(c) } If $\calf$ is realized by a compact 
Lie group or a $p$-compact group, then by Proposition \ref{G&F} and since 
all elements in $S$ are $\calf$-conjugate to elements in $A$, $\calf$ is 
realized by a connected, simple $p$-compact group, and the action of the 
Weyl group $G=\autf(A)$ on $\Q\otimes_{\Z}\Hom(A,\Q_p/\Z_p)$ is irreducible 
as a group generated by pseudoreflections. Using the list of 
pseudoreflection groups and their realizability over $\Q_p$ compiled by 
Clark \& Ewing \cite{ClarkEwing}, as well as the assumption that 
$v_p(|G|)=1$, we see that $G$ must be one of the groups listed in Table 
\ref{tbl:inf-real}, or else one of the other groups $G(m,d,n)$ (of index 
$d$ in $C_m\wr\Sigma_n$) for $d\mid m\mid(p-1)$ with $d<m$. The latter are 
eliminated by the condition $G=O^{p'}(G)\cdot\mu_A^{-1}(\Delta_0)$ in 
(d.ii) (i.e., the fusion systems of the corresponding $p$-compact groups 
are not simple), and so we are left with the groups listed in the table. 

Since a $p$-compact group is determined by its Weyl group by \cite[Theorem 1.1]{AGMV}, 
it remains only to check, when $\rk(A)=p-1$ and based on the constructions of these groups, 
whether $\calb\subseteq\EE\calf$ or $\calh\subseteq\EE\calf$. This situation occurs only 
in the last four cases listed in the table, in which cases the $p$-compact group was 
constructed by Aguad\'e \cite[\S\S\,5--7,\,10]{Aguade}, and 
the use of $\SU(p)$ in his construction shows that extraspecial groups 
of order $p^3$ and exponent $p$ appear as essential subgroups.

In those cases where a torsion linear group is given in Table \ref{tbl:inf-real}, 
it is a union of a sequence of finite groups that by Table \ref{tbl:type3} realize 
a sequence of finite fusion subsystems of $\calf$. 
\end{proof}

The different situations handled in Theorem \ref{ThB} are partly 
summarized in Table \ref{tbl:inf.case}.
	\begin{table}[ht]
	\[ \renewcommand{\arraystretch}{1.5}\renewcommand{\arraycolsep}{1mm}
	\newcommand{\halfup}[1]{\raisebox{2.2ex}[0pt]{$#1$}}
	\begin{array}{|c|c|c|c|c|} \hline
	\dim(V) & \EE\calf\sminus\{A\} & \mu_A(G^\vee) & G= & 
	\textup{Condition} \\\hline
	& \calh &   \ge \Delta_{-1} & O^{p'}(G)\cdot
	\mu_A^{-1}(\Delta_{-1}) & - \\\cline{2-5}
	\halfup{p-1} &  \calb & \ge\Delta_0  & & \raisebox{-2pt}{\text{$V$ 
	contains no $1$-dimensional}} \\\cline{1-3}
	\ge p &  \calb & =\Delta_0 & \halfup{O^{p'}(G)\mu_A^{-1}(\Delta_0)} & 
	\raisebox{2pt}{\text{$\F_pG$-submodule}}\\\hline
	\end{array} \]
	\caption{} \label{tbl:inf.case}
	\end{table}

\section{Examples}
\label{s:examples}

Recall Definition \ref{d:min.act.}: for a given prime $p$, $\GG$ is the class of 
finite groups $G$ with $\UUU\in\sylp{G}$ of order $p$ and not normal, and 
$\GGG$ is the class of those $G\in\GG$ such that $\Aut_G(\UUU)=\Aut(\UUU)$. 
It remains now to describe explicitly which finite groups $G\in\GGG$ and 
$\F_pG$-modules $V$ can appear in Theorems \ref{ThA} and \ref{ThB}. 
This follows immediately from the work already done in \cite{indp2}, 
and is stated in Proposition \ref{p:reps} and Table \ref{tbl:reps}. As 
in \cite{indp2}, when $p$ is a fixed prime, we define, for each odd 
integer $i$ prime to $p$,
        \[ \Delta_{i/2} = \{(r^2,r^i)\,|\,r\in(\Z/p)^\times\}\,. \]
(Compare with the definition of $\Delta_i$ in Notation \ref{n:not2}.)

\begin{Prop} \label{p:reps}
Assume that $G\in\GGG$, and that $V$ is a faithful, minimally active, 
indecomposable $\F_pG$-module such that $\dim(V)\ge p-1$. If $\dim(V)\ge 
p$, then assume also that $\mu_V(G^\vee)\ge\Delta_0$; and if $\dim(V)=p$, 
then assume that $V$ contains no $1$-dimensional $\F_pG$-submodule. Then 
either 
\begin{enuma} 
\item the image of $G$ in $\PGL(V)$ is not almost simple, and $G\le\4G$ 
with the given action on $V$ for one of the pairs $(\4G,V)$ listed in Table 
\ref{tbl:reps} with no entry $G_0$; or 

\item the image of $G$ in $\PGL(V)$ is almost simple, and $G_0\le 
G\le\4G$ with the given action on $V$ for one of the 
triples $(G_0,\4G,V)$ listed in Table \ref{tbl:reps}.

\end{enuma}
In all cases, the entry under $\dim(V)$ gives the dimensions of the 
composition factors of $V$; thus a single number means that $V$ is simple. 
\end{Prop}

\begin{table}[ht]
\begin{small} 
\[ \renewcommand{\arraystretch}{1.4}\addtolength{\arraycolsep}{-2pt}
\newcommand{\halfup}[2][2.2]{\raisebox{#1ex}[0pt]{$#2$}}
\newcommand{\up}[2][2.2]{\raisebox{#1ex}[0pt]{$#2$}}
\newcommand{\sm}[1]{\textup{\begin{small}#1\end{small}}}
\newcommand{\Sm}[1]{\textup{\begin{Small}#1\end{Small}}}
\newcommand{\bnote}[1]{{}^{{\textup{[#1]}}}}
\newcommand{\Y}{\textup{yes}}
\begin{array}{|c|c|c|c|c|c|} \hline\hline
p & G_0 & \dim(V) & \4G & \mu_V(\4G\7) & \mu_V(G_0\7) \\\hline \hline

 & \SL_2(p)\textup{ or }\PSL_2(p) & p-1,~ p & \GL_2(p) \textup{ or}\hfill & \Delta & 
\Delta_{-1/2},~\frac12\Delta_0  \\ \cline{3-3}\cline{5-6}
\halfup{p} & \up[0.5]{(p\ge5)} & (p{-}n{-}1)/n & 
\up[0.5]{\PGL_2(p)\times C_{p-1}} & \Delta & \{(u^2,u^{n-1})\}  
\\\hline

 &  & (p{-}2)/1 &  & \Delta & \frac12\Delta_0 \\\cline{3-3}\cline{5-6}
\halfup{p}  & \halfup{A_p \quad(p\ge5)} & 1/(p{-}2) & 
\halfup{\Sigma_p\times C_{p{-}1}} & \Delta & \frac12\Delta_{-1} \\\hline

p & A_{p+1}\quad(p\ge5) & p & \Sigma_{p+1}\times C_{p{-}1} & \Delta & 
\frac12\Delta_0 \\\hline
p & A_n ~\Sm{$(p{+}2\le n\le 2p{-}1)$} & n-1 & \Sigma_n\times C_{p{-}1} & \Delta_0 
& \frac12\Delta_0  \\\hline
p & \textbf{---} & n & C_{p-1}\wr S_n~\Sm{$(n\ge 
p)$} & \Delta & \textbf{---}  \\\hline\hline
3 & \textbf{---} & 2 & \GL_2(3) & \Delta & \textbf{---}  \\\hline\hline
5 & 2\cdot A_6 & 4 & 4\cdot S_6  & \Delta  & \Delta_{1/2}  \\\hline
5 & \textbf{---} & 4 & (C_4\circ2^{1+4}).S_6 & \Delta & \textbf{---}  \\\hline
5 & \PSp_4(3)=W(E_6)' & 6 & W(E_6)\times4 & \Delta_0.2 & \frac12\Delta_0  \\\hline
5 & \Sp_6(2)=W(E_7)' & 7 & G_0\times4 & \Delta_0 & \frac12\Delta_0  \\\hline
\hline
7 & 6\cdot\PSL_3(4) & 6 & G_0.2_1 & \Delta  & 
\F_p^{\times2}\times\F_p^\times  \\\hline
7 & 6_1\cdot\PSU_4(3) & 6 & G_0.2_2 & \Delta  & 
\F_p^{\times2}\times\F_p^\times  \\\hline
7 & \PSU_3(3) & 6 & G_0.2\times6 & \Delta & \frac12\Delta_1  \\\hline
7 & \PSU_3(3) & 7 & G_0.2\times6 & \Delta & \frac12\Delta_0  \\\hline
7 & \SL_2(8) & 7 & G_0:3\times6 & \Delta  & \frac13\Delta_1  \\\hline
7 & \Sp_6(2)=W(E_7)' & 7 & G_0\times6 & \Delta & \Delta_3  \\\hline
7 & 2\cdot\Omega_8^+(2)=W(E_8)' & 8 & W(E_8)\times3 & \Delta_0.2 & 
\Delta_3  \\\hline
\hline
11 & \PSU_5(2) & 10 & G_0.2\times10 & \Delta  & \frac12\Delta_2  \\\hline
11 & 2\cdot M_{12} & 10,~10 & G_0.2\times5 & \Delta &
\Sm{$\Delta_{1/2} \,,~ \Delta_{7/2}$}  \\\hline
11 & 2\cdot M_{22} & 10,~10 & G_0.2\times5 & \Delta & 
\Sm{$\Delta_{1/2} \,,~ \Delta_{7/2}$}  \\\hline
\hline
13 & \PSU_3(4) & 12 & G_0.4\times12 & \Delta  & \frac13\Delta_1  \\\hline
\hline
\end{array}
\]
\end{small}
\caption{Pairs $(G,V)$, where $G\in\GGG$, $G\le\overline{G}$, $G\ge G_0$ 
when a quasisimple group $G_0$ is given, and where $V$ is a minimally 
active indecomposable module of dimension at least $p-1$, such that 
$\mu_V(G^\vee)\ge\Delta_0$ if $\dim(V)\ge p$, and such that $V$ does not 
have a 1-dimensional submodule if $\dim(V)=p$. In all cases, $\dim(V)$ 
gives the dimensions of the composition factors in $V$. Also, 
$\F_p^{\times2}=\{r^2\,|\,r\in\F_p^\times\}$. The notation $B.C$, $B:C$, and $B\cdot C$ 
for extensions is as in the \emph{Atlas} \cite[p. xx]{Atlas}, as well as the subscripts 
used to make precise certain central extensions or automorphism groups.}
\label{tbl:reps}
\end{table}

\begin{proof} We take as starting point the information in \cite[Table 
4.1]{indp2}. We drop from that table those cases where $\dim(V)<p-1$, 
and also those cases where $\dim(V)\ge p$ and $\mu_V(G^\vee)\ngeq\Delta_0$, 
or where $\dim(V)=p$ and $V$ contains a $1$-dimensional $\F_pG$-submodule. 

Since the table in \cite{indp2} is restricted to representations of 
dimension at least $3$, we must add those representations of dimension $2$ 
that appear. Since $\dim(V)\ge p-1$, this occurs only for $p=3$, and thus 
$G\le\GL_2(3)$. Since this group is solvable, the image of $G$ in $\PGL(V)$ 
cannot be almost simple, and so this case is covered by the unique 
row of the table restricted to $p=3$. 
\end{proof}

We now give two examples, in terms of the pairs $(\4G,V)$ that appear in 
Table \ref{tbl:reps}, to illustrate how this table can be used to list 
explicit fusion systems as described by Theorems \ref{ThA} and \ref{ThB}. 
When $V$ is an $\F_p$-vector space, we set 
$\Aut_{\scal}(V)=Z(\Aut(V))\cong\F_p^\times$: the group of automorphisms 
given by scalar multiplication. 

\newcommand{\Star}{--}  

\begin{Ex} \label{ex:Vsimple}
Fix an odd prime $p\ge5$ and a finite group $\4G\in\GGG$, and choose 
$\UUU\in\sylp{\4G}$. Let $V$ be a simple, $(p-1)$-dimensional, minimally active 
$\F_p\4G$-module, and assume that $\Aut_{\4G}(V)\ge\Aut_{\scal}(V)$. Then 
$\mu_V(\4G^\vee)=\Delta$ by \cite[Proposition 3.13(a)]{indp2}. Let $\Lambda$ 
be a $\Z_p\4G$-lattice such that $\Lambda/p\Lambda\cong V$ (see Proposition 
\ref{dim(p-1)}(a)). 
\begin{enuma}

\item By case (i\Star a) in Table \ref{tbl:ThA}, for each $k\ge2$, there is a 
unique simple fusion systems $\calf$ over 
$(\Lambda/p^k\Lambda)\rtimes\UUU$, with 
$\Aut_{\calf}(\Lambda/p^k\Lambda)\cong\4G$, and such that 
$\EE{\calf}=\{A\}\cup\calh_0\cup\calb_*$. 

\item By case (iv$'$\Star a) in Table \ref{tbl:ThA}, for each $k\ge2$ and each 
$\emptyset\ne I\subseteq\{0,1,\dots,p-1\}$, there is a 
unique simple fusion system $\calf_{I}$ over 
$(\Lambda/p^k\Lambda)\rtimes\UUU$, with 
$\Aut_{\calf_{I}}(\Lambda/p^k\Lambda)=G_0\mu_V^{-1}(\Delta_0)$, and such that 
$\EE{\calf_{I}}=\{A\}\cup\bigl(\bigcup_{i\in I}\calb_i\bigr)$. 

\item By case (iii$''$\Star a) in Table \ref{tbl:ThA}, for each $k\ge2$, there is a 
unique simple fusion system $\calf$ over $(\Lambda/p^k\Lambda)\rtimes\UUU$, with 
$\Aut_{\calf}(\Lambda/p^k\Lambda)=G_0\mu_V^{-1}(\Delta_{-1})$, and such that 
$\EE{\calf}=\{A\}\cup\calh_0$. 

\item Set $A=(\Q_p/\Z_p)\otimes_{\Z_p}\Lambda$, regarded as a discrete 
$\4G$-p-torus. By Theorem \ref{ThB}, there are unique simple fusion systems 
$\calf_B$ and $\calf_H$ over $A\rtimes\UUU$, with $\Aut_{\calf_B}(A)\cong 
O^{p'}(\4G)\mu_V^{-1}(\Delta_0)$ and $\EE{\calf_B}=\{A\}\cup\calb$, and 
$\Aut_{\calf_H}(A)\cong O^{p'}(\4G)\mu_V^{-1}(\Delta_{-1})$ and 
$\EE{\calf_H}=\{A\}\cup\calh$. 

\end{enuma}
Since $V$ is simple (since there is no $(p-2)$-dimensional submodule), none 
of the cases (ii\Star a), (iii$'$\Star a), or (iii$''$\Star c) in Table \ref{tbl:ThA} can 
occur with $G_0\le G\le\4G$ and $V\cong\Omega_1(A)$. Since $\dim(V)<p$, 
case (iv$''$\Star b) in Table \ref{tbl:fin(abc)} cannot occur. 
\end{Ex}

The last column in Table \ref{tbl:reps} can be used to help determine the 
subgroups $O^{p'}(\4G)\mu_V^{-1}(\Delta_t)$ for $i=0,-1$. For example:
\begin{itemize}

\item When $p=5$, $G_0\cong2\cdot A_6$, and $\dim(V)=4$, we have 
$\mu_V(G_0^\vee)=\Delta_{1/2}$: the subgroup of order $4$ in 
$\Delta=(\Z/5)^\times\times(\Z/5)^\times$ generated by the class of $(4,2)$. 
Since $\Delta_{1/2}\Delta_t=\Delta$ for $t=0,-1$, we have 
$G_0\mu_V^{-1}(\Delta_t)=\4G$.

\item When $p=7$, $G_0\cong6\cdot\PSL_3(4)$ or $6\cdot\PSU_4(3)$, and 
$\dim(V)=6$, we have that $\mu_V(G_0^\vee)$ has index $2$ in $\Delta$ and 
does not contain $\Delta_t$ for any $t$. So in all cases, 
$G_0\mu_V^{-1}(\Delta_t)=\4G$: an extension of the form $G_0.2$.

\item If $p=7$, $G_0\cong\PSU_3(3)$, and $\dim(V)=6$, then 
$\mu_V(G_0^\vee)=\frac12\Delta_1$: a subgroup of order $3$ that intersects 
trivially with $\Delta_t$ for $t=0,-1$. So in this case, 
$G_0\mu_V^{-1}(\Delta_t)$ has the form $G_0.2\times3$ (where the precise 
extension depends on $t$). 

\end{itemize}

We now look at one case where the $\F_pG$-module $V$ is not simple. 

\begin{Ex} \label{ex:V0<V}
Let $p$, $\4G$, $\UUU$, $V$, and $\Lambda$ be as in Example \ref{ex:Vsimple}, 
except that we assume that $V$ is indecomposable but not simple, and 
contains a $(p-2)$-dimensional submodule $V_0<V$. Let $\Lambda_0<\Lambda$ be 
a $\Z_p\4G$-sublattice of index $p$ such that $\Lambda_0/p\Lambda\cong V_0$. 
\begin{enuma}

\item[$\bullet$ ] There are simple fusion systems exactly as described in 
cases (a), (b), (c), and (d) in Example \ref{ex:Vsimple}. In 
addition, we have:

\setcounter{enumi}{4}

\item By case (ii\Star a) in Table \ref{tbl:ThA}, for each $k\ge2$, there is a 
unique simple fusion systems $\calf$ over $(\Lambda_0/p^k\Lambda)\rtimes\UUU$, 
with $\Aut_{\calf}(\Lambda_0/p^k\Lambda)\cong\4G$, 
and such that $\EE{\calf}=\{A\}\cup\calb_0\cup\calh_*$. 

\item By case (iii$'$\Star a) in Table \ref{tbl:ThA}, for each $k\ge2$ and each 
$\emptyset\ne I\subseteq\{0,1,\dots,p-1\}$, there is a unique simple fusion 
system $\calf_{I}$ over $(\Lambda_0/p^k\Lambda)\rtimes\UUU$, with 
$\Aut_{\calf_{I}}(\Lambda_0/p^k\Lambda)=O^{p'}(\4G)\mu_V^{-1}(\Delta_{-1})$, 
and such that $\EE{\calf_{I}}=\{A\}\cup\bigl(\bigcup_{i\in 
I}\calh_i\bigr)$.

\end{enuma}
Since there is no $1$-dimensional submodule, case 
(iii$''$\Star c) in Table \ref{tbl:ThA} cannot occur with 
$G_0\le G\le\4G$ and $V\cong\Omega_1(A)$.
Since $\dim(V)<p$, case (iv$''$\Star b) in Table \ref{tbl:fin(abc)} cannot occur. 
\end{Ex}

If we chose to restrict the above examples to the case $\dim(V)=p-1$, this 
is because when $\dim(V)$ is larger, there are far fewer possibilities. By 
Table \ref{tbl:ThA}, $A\cong\Lambda/p^k\Lambda$ for some $k\ge2$, and 
$\EE\calf=\{A\}\cup\calb_0$. Similarly, by Table \ref{tbl:inf.case}, there 
is only one possibility for $\calf$ when $A$ is a discrete $p$-torus with 
$\Omega_1(A)\cong V$.

\appendix

\section{Background on groups and representations}

We collect here some miscellaneous group theoretic results which were 
needed earlier. We begin with a few elementary properties of discrete 
$p$-toral groups that are easily reduced to the analogous statements 
about finite $p$-groups.

\begin{Lem} \label{mod-Fr}
Fix a prime $p$, a discrete $p$-toral group $P$, and a finite group 
$G\le\Aut(P)$ of automorphisms of $P$.  Let 
$1=P_0\nsg{}P_1\nsg\cdots\nsg{}P_m=P$ be a sequence of subgroups, all 
normal in $P$ and normalized by $G$.  Let $H\le{}G$ be the subgroup of 
those $g\in{}G$ which act via the identity on $P_i/P_{i-1}$ for each $1\le 
i\le m$.  Then $H$ is a normal $p$-subgroup of $G$, and hence $H\le 
O_p(G)$. 
\end{Lem}

\begin{proof}  See, e.g., \cite[Lemma 1.7(a)]{BLO3}.  
\end{proof}

\begin{Lem} \label{l0:xx}
Fix an abelian group $A$ each of whose elements has $p$-power order. Let 
$G\le\Aut(A)$ be a finite group of automorphisms, and choose $\UUU\in\sylp{G}$.  
Then
	\[ C_A(\UUU)\le[G,A] \ \Longleftrightarrow\  C_A(G)\le[G,A] 
	\ \Longleftrightarrow\ C_A(G)\le[\UUU,A]\,. \]
\end{Lem}

\begin{proof} This is shown in \cite[Lemma 1.9]{indp2} when $A$ is a finite 
abelian $p$-group, and the proof given there also applies when $A$ is 
infinite and $p$-power torsion. 
\end{proof}

\begin{Lem} \label{1abel}
Let $S$ be a nonabelian discrete $p$-toral group, with abelian subgroup 
$A<S$ of index $p$, and set $Z=Z(S)=C_S(A)$ and $S'=[S,S]=[S,A]$. 
Then $S'\cong A/Z$. Also, $A$ is the unique abelian subgroup of index $p$ in 
$S$ if and only if $|S'|=|A/Z|>p$. 
\end{Lem}

\begin{proof} Choose $1\ne u\in\Aut_S(A)$, and define $\varphi\:A\too A$ 
by setting $\varphi(a)=a-u(a)$. Then $Z=\Ker(\varphi)$ and $S'=\Im(\varphi)$, 
so $A/Z\cong S'$. 

If $|A/Z|=p$, then $S/Z\cong C_p\times C_p$ (it cannot be cyclic since $S$ 
is nonabelian), and each subgroup of index $p$ in $S$ containing $Z$ is 
abelian. Conversely, if $B$ is a second abelian subgroup of 
index $p$, then $Z=A\cap B$ since $S=AB$, so $|A/Z|=p$.
\end{proof}

We now turn attention to discrete $p$-toral groups and discrete 
$G$-$p$-tori. We start with the well known equivalence between 
discrete $G$-$p$-tori and $\Z_pG$-lattices (see Definition \ref{d:ZG-lattice}).

\begin{Prop} \label{tori&lattices}
Fix a prime $p$ and a finite group $G$. Then there is a natural bijection 
	\beq 
	\xymatrix@C=30pt@R=5pt{
	\left\{\parbox{42mm}{\textup{isomorphism classes of discrete 
	$G$-$p$-tori}}\right\} 
	\ar@<0.5ex>[r]^-{\cong} \ar@<-0.5ex>@{<-}[r] & 
	\left\{\parbox{65mm}{\textup{isomorphism classes of $\Z_pG$-lattices in 
	finitely generated $\Q_pG$-modules}}\right\} \\
	A \ar@{|->}[r] & \Hom_{\Z_p}(\Q_p/\Z_p,A) \\
	(\Q_p/\Z_p)\otimes_{\Z_p}\Lambda \ar@{<-|}[r] & \Lambda
	} \eeq
If $A$ is a discrete $G$-$p$-torus and 
$\Lambda=\Hom_{\Z_p}(\Q_p/\Z_p,A)$, then for each $n\ge1$, evaluation at 
$[1/{p^n}]\in\Q_p/\Z_p$ defines an $\F_pG$-linear isomorphism 
$\Lambda/p^n\Lambda\Right2{\cong}\Omega_n(A)$.
\end{Prop}

\begin{proof} If $\Lambda$ is a $\Z_pG$-lattice, then 
$(\Q_p/\Z_p)\otimes_{\Z_p}\Lambda$ is a discrete $G$-$p$-torus, 
and if $A$ is a discrete $G$-$p$-torus, then 
$\Hom_{\Z_p}(\Q_p/\Z_p,A)$ is a $\Z_pG$-lattice. It is an easy 
exercise to show that the natural homomorphisms 
	\[ (\Q_p/\Z_p)\otimes_{\Z_p}\Hom_{\Z_p}(\Q_p/\Z_p,A) 
	\Right6{\textup{eval}} A \]
and
	\[ \Lambda \Right9{\lambda\mapsto(r\mapsto r\otimes\lambda)} 
	\Hom_{\Z_p}\bigl(\Q_p/\Z_p, (\Q_p/\Z_p)\otimes_{\Z_p}\Lambda\bigr) \]
are isomorphisms for each $\Z_pG$-lattice $\Lambda$ and each discrete 
$G$-$p$-torus $A$. The last statement now follows from the 
short exact sequence
	\beq 0 \Right2{} (p^{-n}\Z_p)/\Z_p \Right4{\incl} \Q_p/\Z_p 
	\Right6{(x\mapsto p^nx)} \Q_p/\Z_p \Right2{} 0. \qedhere \eeq
\end{proof}

The next lemma is mostly a well known result in elementary number theory.

\begin{Lem} \label{ZU-modules}
Fix a prime $p$, and let $\UUU$ be a group of order $p$. Let $\zeta$ be a 
primitive $p$-th root of unity, and regard $\Q_p(\zeta)$ 
and $\Z_p[\zeta]$ as $\Z_p\UUU$-modules via some choice of isomorphism 
$\UUU\cong\gen{\zeta}$. 
\begin{enuma} 

\item There are exactly two irreducible $\Q_p\UUU$-modules up to 
isomorphism: a $1$-dimensional module with trivial $\UUU$-action, and a 
$(p-1)$-dimensional module isomorphic to $\Q_p(\zeta)$.

\item The ring $\Z_p[\zeta]$ is a local ring with maximal ideal 
$\pp=(1-\zeta)\Z_p[\zeta]$. Also, $p\Z_p[\zeta]=\pp^{p-1}$.

\item Let $M$ be a $(p-1)$-dimensional irreducible $\Q_p\UUU$-module, and 
let $\Lambda<M$ be a $\Z_p\UUU$-lattice. Then $\Lambda\cong\Z_p[\zeta]$ as 
$\Z_p\UUU$-modules, and hence $\Lambda/p\Lambda\cong\Z_p[\zeta]/\pp^{p-1}$ 
is indecomposable as an $\F_p\UUU$-module.

\item Let $B$ be an infinite abelian discrete $p$-toral group (written 
additively), upon which $\UUU$ acts with $|C_B(\UUU)|=p$. Assume also that 
$\prod_{u\in\UUU}u(x)=1$ for each $x\in B$. Then $B$ is a discrete 
$p$-torus of rank $p-1$, and $B\cong\Q_p(\zeta)/\Z_p[\zeta]$ as 
$\Z_p\UUU$-modules. 

\end{enuma}
\end{Lem}

\begin{proof} \textbf{(a,b) } By \cite[Proposition 
6-2-6]{Goldstein}, $(1-\zeta)\Z[\zeta]$ is the only prime ideal in 
$\Z[\zeta]$ containing $p\Z[\zeta]=(1-\zeta)^{p-1}\Z[\zeta]$. Hence 
$\Q_p\otimes_{\Q}\Q(\zeta)=\Q_p(\zeta)$, so
$\dim_{\Q_p}(\Q_p(\zeta))=p-1$, and $\Z_p[\zeta]$ is a local ring with 
maximal ideal $\pp=(1-\zeta)\Z_p[\zeta]$ where $\pp^{p-1}=p\Z_p[\zeta]$. 
This proves (b), and also that $\Q_p(\zeta)$ is an irreducible 
$(p-1)$-dimensional $\Q_p\UUU$-module. So the only other irreducible 
$\Q_p\UUU$-module is $\Q_p$ with the trivial action.

\noindent\textbf{(c) } By (a), we can assume that $M=\Q_p(\zeta)$. 
Thus $\Lambda$ is a $\Z_p\UUU$-lattice in $\Q_p(\zeta)$, so 
$p^m\Lambda\le\Z_p[\zeta]$ is an ideal for $m$ large enough. Hence 
$p^m\Lambda=\pp^k=(1-\zeta)^k\Z_p[\zeta]$ for some $k$, and $\Lambda$ is 
isomorphic to $\Z_p[\zeta]$ as a $\Z_p\UUU$-module. 

\smallskip

\noindent\textbf{(d) } Since $\prod_{u\in\UUU}u(x)=1$ for each $x\in 
B$, we can regard $B$ as a $\Z_p[\zeta]$-module. For each $n\ge1$, 
$\bigl|\Omega_n(B)/(1-\zeta)\Omega_n(B)\bigr|=|C_{\Omega_n(B)}(\UUU)|=|C_B(\UUU)|=p$, 
and so by (b), there is $r_n\in\Omega_n(B)$ that generates $\Omega_n(B)$ as a 
$\Z_p[\zeta]$-module. Let $R_n$ be the set of all such generators of 
$\Omega_n(B)$, and let $\varphi_n\:R_n\too R_{n-1}$ be the map 
$\varphi_n(r_n)=(r_n)^p$. The $(R_n,\varphi_n)$ thus form an inverse system of 
nonempty finite sets. An element $(r_n)_{n\ge1}$ in the inverse limit 
defines an isomorphism 
$\Q_p(\zeta)/\Z_p[\zeta]\cong(\Q_p/\Z_p)\otimes_{\Z_p}\Z_p[\zeta]\Right3{}B$ 
of $\Z_p[\zeta]$-modules (hence of $\Z_p\UUU$-modules), where $(1/p^n)\otimes\xi$ 
is sent to $\xi\cdot r_n$. 
\end{proof}

\begin{Lem} \label{|S|=p}
Fix a prime $p$, an abelian discrete $p$-toral group $A$ and 
a finite group of automorphisms $G\le\Aut(A)$. Assume, for 
$S\in\sylp{G}$, that $S\nnsg G$ and $|A/C_A(S)|=p$.  Then $|S|=p$.  
\end{Lem}

\begin{proof} This is shown in \cite[Lemma 1.10]{indp1} when $A$ is finite, 
and the general case follows since $G$ acts faithfully on $\Omega_k(A)$ for 
$k$ large enough.
\end{proof}

\begin{Prop} \label{GonA}
Fix an abelian discrete $p$-toral group $A$, and a subgroup $G\le\Aut(A)$.  
Assume the following.
\begin{enumi}  
\item Each Sylow $p$-subgroup of $G$ has order $p$ and is not normal in 
$G$. 
\item For each $x\in{}G$ of order $p$, $[x,A]$ has order $p$, and hence 
$C_A(x)$ has index $p$.  
\end{enumi}
Set $H=O^{p'}(G)$, $A_1=C_A(H)$, and $A_2=[H,A]$.  Then $G$ normalizes 
$A_1$ and $A_2$, $A=A_1\times A_2$, and $H\cong\SL_2(p)$ acts faithfully on 
$A_2\cong C_p^2$.  There are groups of automorphisms $G_i\le\Aut(A_i)$ 
($i=1,2$), such that $p\nmid|G_1|$, $G_2\ge\Aut_H(A_2)\cong\SL_2(p)$, and 
$G\nsg{}G_1\times{}G_2$ (as a subgroup of $\Aut(A)$) with index dividing 
$p-1$.
\end{Prop}

\begin{proof} This is shown in \cite[Lemma 1.11]{indp1} when $A$ is finite. 
The general case then follows by regarding $A$ as the union of the groups 
$\Omega_k(A)$ for $k\ge1$. 
\end{proof}

\end{document}